\documentclass{emsprocart}
\usepackage[czech,english]{babel}
\usepackage{amscd}
\usepackage{dsfont}
\usepackage{url}
\usepackage[all,2cell]{xy}
\usepackage{hyperref}
\UseTwocells

\title[Exact model categories, approximations, and cohomology of sheaves]%
{Exact model categories, approximation theory, and cohomology of quasi-coherent sheaves}

\author{Jan \v{S}\v{t}ov\'\i\v{c}ek}
\contact[stovicek@karlin.mff.cuni.cz]{%
Jan \v{S}\v{t}ov\'\i\v{c}ek, %
Department of Algebra, %
Faculty of Mathematics and Physics, %
Charles University in Prague, %
Sokolovsk\'a 83, %
186 75 Praha 8, %
Czech Republic%
}

\renewcommand{\iff}{if and only if }
\newcommand{\st}{such that }
\newcommand{\wrt}{with respect to }

\newcommand{\la}{\longrightarrow}
\newcommand{\dd}{\colon}
\newcommand{\ep}{\varepsilon}
\newcommand{\ph}{\varphi}
\newcommand{\tht}{\vartheta}
\newcommand{\dif}{\partial}

\newcommand{\wfs}{weak factorization system }
\newcommand{\wfss}{weak factorization systems }

\newcommand{\inv}{^{-1}}
\newcommand{\lifts}{\,\square\,}
\newcommand{\unit}{\mathds{1}}

\newcommand{\Z}{\mathbb{Z}}

\newcommand{\Spec}{\operatorname{Spec}}    
\newcommand{\res}{\operatorname{res}}      
\newcommand{\OO}{\mathcal{O}}              
\newcommand{\PP}[2]{\mathbb{P}^{#1}_{#2}}  

\newcommand{\Hom}{\operatorname{Hom}}
\newcommand{\HOM}{\mathcal{H}\mathnormal{om}}

\newcommand{\Ext}{\operatorname{Ext}}

\newcommand{\Ker}{\operatorname{Ker}}
\newcommand{\Img}{\operatorname{Im}}
\newcommand{\Coker}{\operatorname{Coker}}

\newcommand{\Cosyz}[2]{\mho_{#1}(#2)}

\newcommand{\li}{\varinjlim}
\newcommand{\card}[1]{\lvert #1 \rvert} 
 
\newcommand{\sgn}{\operatorname{sgn}}

\newcommand{\A}{\mathcal{A}}
\newcommand{\B}{\mathcal{B}}
\newcommand{\C}{\mathcal{C}}
\newcommand{\D}{\mathcal{D}}
\newcommand{\E}{\mathcal{E}}
\newcommand{\F}{\mathcal{F}}
\newcommand{\G}{\mathcal{G}}
\newcommand{\I}{\mathcal{I}}
\newcommand{\clL}{\mathcal{L}}
\newcommand{\M}{\mathcal{M}}
\newcommand{\clP}{\mathcal{P}}

\newcommand{\R}{\mathcal{R}}
\newcommand{\clS}{\mathcal{S}}

\newcommand{\U}{\mathcal{U}} 
\newcommand{\X}{\mathcal{X}}

\newcommand{\Subobj}[1]{\operatorname{Subobj}({#1})}
\newcommand{\Pow}[1]{\mathcal{P}({#1})}
\newcommand{\Infl}{\operatorname{Infl}}
\newcommand{\Defl}{\operatorname{Defl}}

\newcommand{\Modr}[1]{\mathrm{Mod}\textrm{-}{#1}}

\newcommand{\ModR}{\mathrm{Mod}\textrm{-}R}

\newcommand{\modR}{\mathrm{mod}\textrm{-}R}

\newcommand{\ProjR}{\mathrm{Proj}\textrm{-}R}

\newcommand{\VectR}{\mathrm{Vect}\textrm{-}R}

\newcommand{\InjR}{\mathrm{Inj}\textrm{-}R}

\newcommand{\FlatR}{\mathrm{Flat}\textrm{-}R}

\newcommand{\CotR}{\mathrm{Cot}\textrm{-}R}

\newcommand{\Qco}[1]{\operatorname{Qcoh}({#1})} 
\newcommand{\QcoR}{\operatorname{Qcoh}(R)} 
\newcommand{\Ab}{\mathrm{Ab}}
\newcommand{\Inj}{\operatorname{Inj}}
\newcommand{\Icell}{\I\textrm{-}\mathrm{cell}} 
\newcommand{\Filt}[1]{\operatorname{Filt}{#1}}

\newcommand{\op}{^\textrm{op}}

\newcommand{\Der}[1]{\mathbf{D}({#1})}
\newcommand{\Db}[1]{\mathbf{D}^\mathrm{b}({#1})}
\newcommand{\Htp}[1]{\mathbf{K}({#1})}
\newcommand{\Cpx}[1]{\mathbf{C}({#1})}
\newcommand{\Cac}[1]{\mathbf{C}_\mathrm{ac}({#1})}

\newcommand{\cof}{\mathrm{Cof}}
\newcommand{\Cof}{\mathcal{C}}
\newcommand{\tcof}{\mathrm{TCof}}

\newcommand{\we}{\mathrm{W}}
\newcommand{\We}{\mathcal{W}}
\newcommand{\tfib}{\mathrm{TFib}}

\newcommand{\fib}{\mathrm{Fib}}
\newcommand{\Fib}{\mathcal{F}}
\newcommand{\lhtp}{\sim_\ell}
\newcommand{\rhtp}{\sim_r}
\newcommand{\Ho}{\operatorname{Ho}}
\newcommand{\Lder}[1]{\mathbf{L}{#1}}
\newcommand{\Rder}[1]{\mathbf{R}{#1}}
\newcommand{\Lotimes}{\otimes^{\mathbf{L}}}
\newcommand{\RHOM}{\Rder\HOM}

\theoremstyle{plain}
\newtheorem{thm}{Theorem}[section]
\newtheorem{lem}[thm]{Lemma}
\newtheorem{prop}[thm]{Proposition}
\newtheorem{cor}[thm]{Corollary}

\theoremstyle{definition}
\newtheorem{defn}[thm]{Definition}
\newtheorem{constr}[thm]{Construction}
\newtheorem{nota}[thm]{Notation}

\theoremstyle{remark}
\newtheorem{rem}[thm]{Remark}

\newtheorem{expl}[thm]{Example}

\begin{document}

\begin{abstract}
Our aim is to give a fairly complete account on the construction of compatible model structures on exact categories and symmetric monoidal exact categories, in some cases generalizing previously known results. We describe the close connection of this theory to approximation theory and cotorsion pairs. We also discuss the motivating applications with the emphasis on constructing monoidal model structures for the derived category of quasi-coherent sheaves of modules over a scheme.
\end{abstract}



\maketitle

\setcounter{tocdepth}{1}
\tableofcontents

\section{Introduction}

The aim of the present notes, partly based on the doctoral course given by the author at the University of Padova in March 2012, is twofold:
\begin{enumerate}
\item to give a rather complete account on the construction of exact model structures and describe the link to cotorsion pairs and approximation theory;
\item to generalize the theory so that it applies to interesting recently studied classes of examples.
\end{enumerate}
The parts of the paper related to (1) are mostly not new, except for the presentation and various improvements. However, there does not seem to be a suitable reference containing all the story and, as the adaptation to the algebraic setting sometimes requires small changes in the available definitions related to model categories, it seemed desirable to write up the construction at a reasonable level of details. Some results related to (2), on the other hand, are to our best knowledge original.

\smallskip

The concept of a model category~\cite{QHtp,H2,Hir} has existed for half a century. Despite being intensively studied by topologists, it has not attracted much attention in the theory of algebraic triangulated categories. There are probably two reasons for this development: The foundation of the theory of abelian and exact model categories has been only given a decade ago by Hovey~\cite{H1} (see also~\cite{H3} for a nice overview), and the ``implementation details'' for their construction are rather recent, see~\cite{EGPT,G3,G2,G1,G4,SaSt,YL11}. In the meantime, a successful theory for algebraic triangulated categories has been developed, based on dg algebras and dg categories.

The exact model categories give in many respects a complementary approach to that of dg algebras, with different advantages and weaknesses, and it has a good potential for instance in Gorenstein homological algebra~\cite{Holm}, homological algebra in the category of quasi-coherent sheaves~\cite{EGT,EGPT,G2}, or in connection with recent developments regarding the Grothendieck duality~\cite{Jo,Mur,Nee3,Nee4,Nee}. Interestingly, although model structures in connection with singularity categories are first explicitly mentioned in~\cite{Beck12} by Becker, Murfet~\cite{Mur} and Neeman~\cite{Nee3,Nee4,Nee} still implicitly use parts of the theory which we are going to present.

While dg algebras provide perfect tools for constructing functors from single objects (say tilting or Koszul equivalences), the approach via models gives advantage in several theoretical questions. It may for example happen as in~\cite{Jo,Nee3} that a dg model for a given triangulated category is too complicated to understand, but the category itself has a rather easy description. There is, however, another important aspect---the model theoretic approach links the theory of triangulated categories to approximation theory~\cite{GT}, allowing deep insights on both sides. This is by no means to say that the dg and model techniques exclude each other---Keller in his seminal paper~\cite{K2} in fact constructed two model structures for the derived category of a small dg category, and this point of view has been for example used to prove non-trivial results about triangulated torsion pairs in~\cite[\S3.2]{StPo}.

Approximation theory is roughly speaking concerned with approximating general objects (modules, sheaves, complexes) by objects from special classes. Cofibrant and fibrant replacements in model categories are often exactly this kind of approximations. The central notion in that context is that of a cotorsion pair~\cite{Sal}, whose significance has been recognized both in abstract module theory~\cite{GT} and representation theory of finite dimensional algebras~\cite[Chapter 8]{AHK}. The approach to construct approximations which we follow here started in~\cite{ET01}, and the connection to model categories and Quillen's small object argument have been noticed in~\cite{H1,Ro} and in some form also in~\cite{BR}. It was soon realized that similar results hold also for sheaf categories, for instance in~\cite{EE,EEGO,EnoOy}.

It is fair to remark that there is an alternative approach to approximation theory, namely Bican's and El Bashir's proof of the Flat Cover Conjecture in~\cite{BEE} and its follow ups~\cite{ElB,Eno2,HoJo,Rump}, which does not seem to fit in our framework.

\smallskip

The first aim of ours, partly inspired by~\cite{EJv2}, but at a more advanced level, is to collect the essentials of the theory in one place together with a motivating and guiding example from~\cite{G2,EGPT}: to construct for an arbitrary scheme $X$ a model structure for $\Der{\Qco{X}}$, the unbounded derived category of quasi-coherent sheaves, which is compatible with the tensor product. Of course, $\Der{\Qco{X}}$ may not be the category we wish to work with as we also have the subcategory of the derived category of all sheaves consisting of objects with quasi-coherent cohomology, but for many schemes the two categories are equivalent by~\cite[Corollary 5.5]{BN93}.

In order to achieve the goal, we also discuss in detail an equivalent description of $\Qco{X}$ as the category of certain modules over a representation of a poset in the category of rings. The description is due to Enochs and Estrada~\cite{EE} and, although not very well suited for direct computations with coherent sheaves, it is excellent for theoretical questions regarding big sheaves. For example, it is a relatively straightforward task to prove that $\Qco{X}$ always is a Grothendieck category---compare to~\cite[B.2, p. 409]{TT}! This presentation is also quite accessible to the readers not acquainted with algebraic geometry.

\smallskip

As mentioned above, the other goal of the paper is to generalize the theory so that it is strong enough to apply to model structures in exact categories ``appearing in the nature.'' Our motivation involves in particular an interpretation of recent results about singularity categories~\cite{Jo,Mur,Nee3,Nee4,Nee} and using models in conjunction with dg categories~\cite{StPo}.

This program has been started by Saor\'\i{}n and the author in~\cite{SaSt,St2} and it follows the spirit of~\cite{G5}. It is also, in a way, not a compulsory part for the reader, as it should be manageable to read the paper as if it were written only for, say, module categories instead of more general exact categories. Even in this restriction the presented results are relevant.

The main problem which we address here is a suitable axiomatics for exact categories which allows to use Quillen's small object argument and deconstructibility techniques to construct cotorsion pairs and model structures. The best suited concept so far seems to be an exact category of Grothendieck type defined in this text, although the theory is not optimal yet. The main problem is that we do not know whether the important Hill Lemma (Proposition~\ref{prop:hill-lemma}) holds for these exact categories or in which way we should adjust the axioms to make it hold. As a consequence, some of our results including Proposition~\ref{prop:deconstr-advanced}, Corollary~\ref{cor:lhs-cot-converse} or~Theorem~\ref{thm:inj-model-for-D(E)} cannot be stated in as theoretically clean way as we would have wished. This is left as a possible direction for future research, where the promising directions include Enochs' filtration shortening techniques~\cite{Eno2}, or Lurie's colimit rearrangements from~\cite[\S A.1.5]{Lur09} or~\cite{Mak08}.

\subsection*{Acknowledgments}

This research was supported by grant GA~\v{C}R P201/12/G028 from Czech Science Foundation.

\section{Quasi-coherent modules}
\label{sec:Qco}

In order to have classes of examples at hand, we start with describing the categories of quasi-coherent modules over schemes and diagrams of rings.

\subsection{Grothendieck categories}
\label{subsec:groth}

Although the construction of model structures described later in this text has been motivated from the beginning by homological algebra in module and sheaf categories, several constructions work easily more abstractly for Grothendieck categories and, as we will discuss in Section~\ref{sec:groth-exact}, even for nice enough exact categories. Thus we start with the definition and basic properties of Groth\-end\-ieck categories.

\begin{defn} \label{def:groth} \cite{Gro}
An abelian category $\G$ is called a \emph{Grothendieck category} if
\begin{enumerate}
\item[(Gr1)] $\G$ has all small coproducts (equivalently: $\G$ is a cocomplete category).

\item[(Gr2)] $\G$ has exact direct limits. That is, given a direct system
\[ (0 \la X_j \overset{i_j}\la Y_j \overset{p_j}\la Z_j \la 0)_{j \in I} \]
of short exact sequences, then the colimit diagram
\[ 0 \la \li_{j \in I} X_j \la \li_{j \in I} Y_j \la \li_{j \in I} Z_j \la 0 \]
is again a short exact sequence in $\G$. This is sometimes called the AB5 condition following an equivalent requirement in~\cite[p. 129]{Gro}.

\item[(Gr3)] $\G$ has a generator. That is, there is an object $G \in \G$ \st every $X \in \G$ admits an epimorphism $G^{(I)} \to X \to 0$. Here, $G^{(I)}$ stands for the coproduct $\coprod_{j \in I} G_j$ of copies $G_j$ of $G$.
\end{enumerate}
\end{defn}

An important property of a Grothendieck category is that it always has enough injective objects, which is very good from the point of view of homological algebra. This is in fact a good reason to consider infinitely generated modules or sheaves of infinitely generated modules: injective objects are often infinitely generated in any reasonable sense. We summarize the comment in a theorem:

\begin{thm} \label{thm:enough-inj}
Let $\G$ be a Grothendieck category. Then each $X \in \G$ admits an injective envelope $X \to E(X)$. Moreover, $\G$ admits all small products (equivalently: it is complete) and has an injective cogenerator $C$. That is, $C$ is injective in $\G$ and each $X \in \G$ admits a monomorphism of the form $0 \to X \to C^I$.
\end{thm}

\begin{proof}
The fact that every object $X \in \G$ admits a monomorphism $0 \to X \to E$ with $E$ injective was shown already in~\cite[Th\'eor\`eme 1.10.1]{Gro}. The existence of injective envelopes and an injective cogenerator is proved in~\cite[Theorem 2.9]{Mitch} and~\cite[Corollary 2.11]{Mitch}, respectively. The fact that $\G$ has products and many other properties of $\G$ are clear from the Popescu-Gabriel theorem, see e.g.\ \cite[Theorem X.4.1]{S}.
\end{proof}

\subsection{Quasi-coherent modules over diagrams of rings}
\label{subsec:Qco-mod}

The simplest examples of Grothendieck categories are module categories $\G = \ModR$. In this section we construct more complicated examples, involving diagrams of rings and diagrams of modules over these rings. In fact, for suitable choices we obtain a category equivalent to the category of quasi-coherent sheaves over any given scheme. The presentation here is an adjusted version of~\cite[\S2]{EE}. Since the discussion in~\cite[\S2]{EE} is rather brief and many details are omitted, we will also discuss in~\S\ref{subsec:Qco-sheaves} the translation between quasi-coherent sheaves and the Grothendieck categories which we describe here.

\begin{defn} \label{def:poset-ring}
Let $(I,\le)$ be a partially ordered set. Then a \emph{representation} $R$ of the poset $I$ in the category of rings is given by the following data:
\begin{enumerate}
\item for every $i \in I$, we have a ring $R(i)$,
\item for every $i \le j$, we have a ring homomorphism $R^i_j: R(i) \to R(j)$, and
\item we require that for every triple $i \le j \le k$ the morphism $R^i_k\dd R(i) \to R(k)$ coincides with the composition $R^j_k \circ R^i_j$, and also that $R^i_i = 1_{R(i)}$.
\end{enumerate}
\end{defn}

\begin{rem} \label{rem:poset-rings}
If we view $I$ as a thin category in the usual way, then $R$ is none other than a covariant functor
\[ R\dd I \la \mathrm{Rings}. \]
\end{rem}

\begin{rem} \label{rem:poset-rings-nc}
Although all of our examples and the geometrically minded motivation will involve only representations of posets in the category of commutative rings, non-commutative rings can be potentially useful too. For instance, one can consider sheaves of algebras of differential operators and ring representations coming from them. In any case, the commutativity is not necessary for the basic properties which we discuss in this section, so we do not include it in our definition.
\end{rem}

Having defined representations of $I$ in the category of rings, we can define modules over such representations in a straightforward manner.

\begin{defn} \label{def:poset-module}
Let $R$ be a representation of a poset $I$ in the category of rings. A \emph{right $R$-module} is
\begin{enumerate}
\item a collection $(M(i))_{i \in I}$, where $M(i) \in \Modr{R(i)}$ for each $i \in I$
\item together with morphisms of the additive groups $M^i_j\dd M(i) \to M(j)$ for each $i \le j$
\item satisfying the compatibility conditions $M^i_k = M^j_k \circ M^i_j$ and $M^i_i = 1_{M(i)}$ for every triple $i \le j \le k$, and \st
\item the ring actions are respected in the following way: Given $x \in R(i)$ and $m \in M(i)$ for $i \in I$, then for any $j \ge i$ we have the equality
\[ M^i_j(m \cdot x) = M^i_j(m) \cdot R^i_j(x). \]
\end{enumerate}
\end{defn}


All our modules in the rest of the text are going to be right modules unless explicitly stated otherwise, so we will omit usually the adjective ``right''. In order to obtain a category, it remains to define morphisms of $R$-modules. The definition is the obvious one.

\begin{defn} \label{def:poset-morph}
Let $R$ be a representation of a poset $I$ in the category of rings and $M,N$ be $R$-modules. A morphism $f\dd M \to N$ is a collection of $(f(i)\dd M(i) \to N(i))_{i \in I}$, where $f(i)$ is a morphism of $R(i)$-modules for every $i \in I$, and the square
\[
\begin{CD}
M(i) @>{f(i)}>> N(i)     \\
@V{M^i_j}VV @VV{N^i_j}V  \\
M(j) @>{f(j)}>> N(j)     \\
\end{CD}
\]
commutes for every $i < j$.
\end{defn}

Let us denote the category of all $R$-modules by $\ModR$. As we quickly observe:

\begin{prop} \label{prop:poset-modules}
Let $(I,\le)$ be a poset and $R$ a representation of $I$ in the category of rings. Then $\ModR$ is a Grothendieck category. Moreover limits and colimits of diagrams of modules are computed component wise---we compute the corresponding (co)limit in $\Modr{R(i)}$ for each $i \in I$ and connect these by the (co)limit morphisms.
\end{prop}

\begin{proof}
Everything is very easy to check except for the existence of a generator in $\ModR$. In fact, there is a generating set $\{P_i \mid i \in I\}$ of projective modules described as follows:
\[
P_i(j) =
\begin{cases}
R(j) & \textrm{if } j \ge i, \\
0    & \textrm{otherwise}
\end{cases}
\]
and the homomorphism $P_i(j) \to P_i(k)$ for $j \le k$ either coincides with $R(j) \to R(k)$ if $i \le j \le k$ or vanishes otherwise.

One directly checks that there is a isomorphism
\[
\Hom_R(P_i,M) \cong M(i) \qquad \textrm{for each } i \in I \textrm{ and } M \in \ModR
\]
which assigns to every $f\dd P_i \to M$ the element $f(i)(1_{R(i)}) \in M(i)$. Moreover, the canonical homomorphism
\[
\coprod_{i \in I} P_i^{(M(i))} \la M
\]
is surjective for every $M \in \ModR$, so $G = \coprod_{i \in I} P_i$ is a projective generator.
\end{proof}

Although being valid Grothendieck categories, the categories $\ModR$ as above are not the categories of our interest yet. In order to get a description of categories of quasi-coherent sheaves as promised, we must consider certain full subcategories instead. In order for this to work, we need an extra condition on $R$:

\begin{defn} \label{def:flat} \cite{EE}
Let $R$ be a representation of a poset $I$ in rings. We call $R$ a \emph{flat} representation if for each pair $i < j$ in $I$, the ring homomorphism $R^i_j\dd R(i) \to R(j)$ gives $R(j)$ the structure of a flat left $R(i)$-module. That is,
\[ - \otimes_{R(i)} R(j)\dd \Modr{R(i)} \la \Modr{R(j)} \]
is an exact functor.
\end{defn}

As discussed later in \S\ref{subsec:Qco-sheaves}, the representations coming from structure sheaves of schemes always satisfy this condition. For such an $R$, we can single out the modules we are interested in:

\begin{defn} \label{def:Qco}
Let $R$ be a flat representation of $I$ in rings. A module $M \in \ModR$ is called \emph{quasi-coherent} if, for every $i < j$, the $R(j)$-module homomorphism
\begin{align*}
M(i) \otimes_{R(i)} R(j) &\la         M(j)              \\
   m \otimes x           &\longmapsto M^i_j(m) \cdot x
\end{align*}
is an isomorphism.

Denote the full subcategory of $\ModR$ formed by quasi-coherent $R$-modules by $\QcoR$.
\end{defn}

Again, we obtain a Grothendieck category.

\begin{thm} \label{thm:Qco-Grothendieck}
Let $(I,\le)$ be a poset and $R$ be a flat representation of $I$ in the category of rings. Then $\QcoR$ is a Grothendieck category. Moreover colimits of diagrams and limits of finite diagrams are computed component wise---that is, for each $i \in I$ separately.
\end{thm}

\begin{proof}
Again, the main task is to prove that $\QcoR$ has a generator and the rest is rather easy, since taking colimits and kernels (hence also finite limits) commutes with the tensor products $- \otimes_{R(i)} R(j)$, where $i,j \in I$ and $i \le j$. We omit the proof of the existence of a generator as it is rather technical, and refer to~\cite[Corollary 3.5]{EE} instead.
\end{proof}

\begin{rem} \label{rem:products}
Every Grothendieck category has small products by Theorem~\ref{thm:enough-inj}, and so must have them $\QcoR$. However, these are typically \emph{not} computed component wise and do not seem to be well understood. Since $\QcoR$ is a cocomplete category with a generator and the inclusion functor $\QcoR \to \ModR$ preserves small colimits, the inclusion $\QcoR \to \ModR$ has a right adjoint
\[ Q\dd \ModR \la \QcoR \]
by the special adjoint functor theorem~\cite[\S5.8]{McL2} (compare to~\cite[Lemma B.12]{TT}!) Following~\cite{TT}, we call such a $Q$ the \emph{coherator}. Clearly, if $(M_k)_{k \in K}$ is a collection of quasi-coherent $R$-modules, the product in $\QcoR$ is computed as $Q(\prod M_k)$, where $\prod M_k$ stands for the (component wise) product in $\ModR$. However, the abstract way of constructing $Q$ gives very little information on what $Q(\prod M_k)$ actually looks like. Some more information on this account is given in~\cite[B.14 and B.15]{TT}.
\end{rem}

Before discussing a general construction in the next section, we exhibit particular examples of flat representations of posets of geometric origin and quasi-coherent modules over them.

\begin{expl} \label{expl:P1k}
Consider the three element poset be given by the Hasse diagram
\[
\begin{CD}
\bullet @>>> \bullet @<<< \bullet
\end{CD}
\]
and a representation in the category of rings of the form
\[
\begin{CD}
R\dd \quad k[x] @>{\subseteq}>> k[x,x\inv] @<{\supseteq}<< k[x\inv],
\end{CD}
\]
where $k$ is an arbitrary commutative ring. Clearly $R$ is a flat representation since the inclusions are localization morphisms.

For each $n \in \Z$, we have a quasi-coherent $R$-module
\[
\begin{CD}
\OO(n) \dd \quad k[x] @>{\subseteq}>> k[x,x\inv] @<{x^n \cdot -}<< k[x\inv].
\end{CD}
\]
One can easily check that $\OO(m) \not\cong \OO(n)$ whenever $m \ne n$, since by direct computation $\Hom_R(\OO(m), \OO(n)) = 0$ for $m>n$.

In fact, the category $\QcoR$ is equivalent to the category of quasi-coherent sheaves over $\PP1k$, the projective line over $k$.
\end{expl}

\begin{expl} \label{expl:P2k}
Given a commutative ring $k$, let us now show a flat representation of a poset corresponding to the scheme $\PP2k$, the projective plane over $k$. The Hasse diagram of the poset has the following shape:
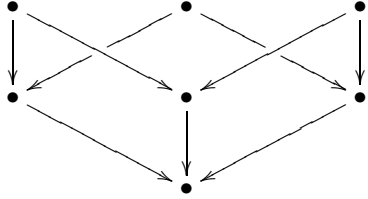
\[
\xymatrix{
\bullet \ar[drr] \ar[d] && \bullet \ar[dll]|\hole \ar[drr]|\hole && \bullet \ar[dll] \ar[d]   \\
\bullet \ar[drr]        && \bullet \ar[d]                        && \bullet \ar[dll]          \\
&& \bullet
}
\]

To describe the representation $R$ corresponding to $\PP2k$, it is enough to define the ring homomorphisms corresponding to arrows in the Hasse diagram. Such a description is given in the following diagram, where all the rings are subrings of $k[x_0^{\pm1},x_1^{\pm1},x_2^{\pm1}]$, the ring of Laurent polynomials in three indeterminates over $k$, and all the ring homomorphisms are inclusions:
\[
\xymatrix{
k[\frac{x_1}{x_0}, \frac{x_2}{x_0}] \ar[dr] \ar[d] &
k[\frac{x_0}{x_1}, \frac{x_2}{x_1}] \ar[dl]|\hole \ar[dr]|\hole  &
k[\frac{x_0}{x_2}, \frac{x_1}{x_2}] \ar[dl] \ar[d]
\\
k[\frac{x_2}{x_0}, (\frac{x_1}{x_0})^{\pm1}] \ar[dr] &
k[\frac{x_1}{x_2}, (\frac{x_0}{x_2})^{\pm1}] \ar[d]  &
k[\frac{x_0}{x_1}, (\frac{x_2}{x_1})^{\pm1}] \ar[dl]
\\
& k[(\frac{x_1}{x_0})^{\pm1}, (\frac{x_2}{x_0})^{\pm1}]
}
\]
\end{expl}

\subsection{Realizing modules over schemes as modules over diagrams}
\label{subsec:Qco-sheaves}

The aim of this section is to make precise how the category of quasi-coherent sheaves of modules over a given scheme can be described in terms of quasi-coherent modules over a flat ring representation of a poset. Here we assume some familiarity with the basic notions in place: those of a scheme and of a quasi-coherent sheaf of modules over a scheme~\cite{Hart2,GW}. On the other hand, understanding this part is not necessary for understanding most of the text which follows, so the reader can skip it and continue with Section~\ref{sec:groth-exact}.

Suppose that $(X,\OO_X)$ is a scheme, that is a ringed space which is locally isomorphic to $(\Spec R,\OO_{\Spec R})$ for a commutative ring $R$. Given this data, we first construct a representation of a poset in the category of commutative rings.

\begin{constr} \label{constr:scheme-to-rep}
Let $\U$ be a collection of open affine sets of $X$ satisfying the following two conditions:
\begin{enumerate}
\item $\U$ covers $X$; that is $X = \bigcup \U$.
\item Given $U,V \in \U$, then $U \cap V = \bigcup \{W \in \U \mid W \subseteq U \cap V\}$.
\end{enumerate}

It is always a safe choice to take the collection of all affine open sets, but often much smaller sets $\U$ will do. For projective schemes for example, we can always choose $\U$ to be finite.

Now $\U$ is a poset with respect to inclusion and we put $I = \U\op$, the opposite poset. Since $\OO_X$ is a sheaf of commutative rings, we in particular have a functor
\[ \U\op \la \mathrm{CommRings}  \]
which sends a pair $U \supseteq V$ of sets in $\U$ to the restriction $\res^U_V\dd \OO_X(U) \to \OO_X(V)$. By the very definition of $I$, this is the same as saying that we have a covariant functor $R\dd I \to \mathrm{CommRings}$ such that, in the notation of Definition~\ref{def:poset-ring}, we have $R(U) = \OO(U)$ and $R^U_V = \res^U_V$.
\end{constr}

A standard fact is that the representation of $I$ we get in this way is flat:

\begin{lem} \label{lem:flat-rep}
Let $R$ be the representation of $I$ in the category of rings as in Construction~\ref{constr:scheme-to-rep}. Then $R$ is flat.
\end{lem}

\begin{proof}
This is proved for instance in~\cite[Proposition 14.3(1) and (4)]{GW}. Upon unraveling the definitions, the statements relies on the following fact from commutative algebra, \cite[Theorem 7.1]{Mat}: Given a homomorphism $\ph\dd R \to S$ of commutative rings, then $S$ is flat over $R$ \iff $S_\mathfrak{q}$ is flat over $R_{\ph\inv(\mathfrak{q})}$ for every $\mathfrak{q} \in \Spec S$.
\end{proof}

It is now easy to construct a functor from the category $\Qco X$ of quasi-coherent sheaves of $\OO_X$-modules to the category $\QcoR$ of quasi-coherent modules over $R$.

\begin{constr} \label{constr:modules-vs-sheaves}
Let us adopt the notation from Construction~\ref{constr:scheme-to-rep}. Given $M \in \Qco X$ and two affine open sets $U \supseteq V$, then canonically $M(U) \otimes_{\OO_X(U)} \OO(V) \cong M(V)$---here we just apply~\cite[Remark 7.23 and Proposition 7.24(2)]{GW} to the open immersion
\[ (\Spec R(V),\OO_{\Spec R(V)}) \cong (V,\OO_{X|V}) \la (U,\OO_{X|U}) \cong (\Spec R(U),\OO_{\Spec R(U)}) \]
and the corresponding ring homomorphism $R^U_V\dd R(U) \to R(V)$.

Thus, viewing the sheaf $M$ as a contravariant functor from the poset of open sets of $X$ to $\Ab$, we may restrict the functor to $I = \U\op$. This way we assign to $M \in \Qco X$ an $R$-module $F(M)$ and $F(M)$ is quasi-coherent by the above discussion. This assignment is obviously functorial, so that we get an additive functor
\[ F\dd \Qco X \la \QcoR. \]
\end{constr}

Seemingly, there is much more structure in $M \in \Qco X$ than in $F(M)$. While the former is a sheaf of modules over a possibly complicated topological space $X$, the latter is only a collection of modules satisfying a certain coherence condition. However, the fact that $M$ is quasi-coherent is itself very restrictive and we have the following crucial result; see~\cite[\S2]{EE}.

\begin{thm} \label{thm:modules-vs-sheaves}
The functor $F$ from Construction~\ref{constr:modules-vs-sheaves} (which depends on the choice of $\U$ in Construction~\ref{constr:scheme-to-rep}) is an equivalence of categories.
\end{thm}

\begin{proof}
Note that a quasi-coherent sheaf of modules $M \in \Qco X$ is determined up to a canonical isomorphism by its image under $F$. Indeed, this follows from conditions~(1) and~(2) in Construction~\ref{constr:scheme-to-rep}, \cite[Theorem 7.16(1)]{GW} and the sheaf axiom. Similarly a morphism $f\dd M \to N$ in $\Qco X$ is fully determined by $F(f)$. In particular $F$ is a faithful functor.

In order to prove that $F$ is dense, fix a module $A \in \QcoR$ and let us introduce some notation. Given an upper subset $L \subseteq I = \U\op$ \wrt the partial order $\le$ on $I$, then $\bigcup L \subseteq X$ is an open subset of $X$, so that we can consider quasi-coherent sheaves over $\bigcup L$. We can also restrict the representation $R\dd I \to \mathrm{CommRings}$ to the representation $L \to \mathrm{CommRings}$, which we denote by $R_L$. Clearly $R_L$ is a flat representation and Construction~\ref{constr:modules-vs-sheaves} provides us with a functor
\[ F_L\dd \Qco{\bigcup L} \to \Qco{R_L}. \]
Finally, we have the restriction functor $\QcoR \to \Qco{R_L}$ and we will denote the image of $A$ under this functor by $A|_L$. 

Now we shall consider the collection $\Lambda$ of all upper subsets $L \subseteq I = \U\op$ \st there is a quasi-coherent sheaf of modules $M_L \in \Qco{\bigcup L}$ with $F_L(M_L) = A|_L$. As such $M_L$ is unique up to a canonical isomorphism, the collection $\Lambda$ is closed under unions of chains. Hence by Zorn's lemma there is an upper subset $L \subseteq I$ which belongs to $\Lambda$ and is maximal such \wrt inclusion. We claim that $L = I$. Suppose by way of contradiction that $L \subsetneqq I$. Then there is $U \in \U\op = I$ \st $U \not\subseteq \bigcup L$ and we consider the unique quasi-coherent sheaf $\widetilde{A(U)} \in \Qco U$ whose global section module is $A(U)$. Now we invoke condition~(2) of Construction~\ref{constr:scheme-to-rep} which, together with~\cite[Theorem 7.16(1)]{GW} and the sheaf axiom, allows us to construct a canonical isomorphism $(M_L)|_V \cong \widetilde{A(U)}|_V$ of sheaves over the open set $V = \bigcup L \cap U$. Thus we can glue $M_L$ and $\widetilde{A(U)}$ to a quasi-coherent sheaf over $\bigcup L \cup U$, showing that $L \cup \{W \in I \mid W \ge U \textrm{ in } I \}$ belongs to $\Lambda$, in contradiction to the choice of $L$. This proves the claim and the density of $F$.

The fact that $F$ is full is proved in a similar way. Given a morphism $g\dd A \to B$ in $\QcoR$, we denote by $\Lambda'$ the collection of all upper subsets $L \subseteq I$ \st $g|_L$ lifts to a morphism of sheaves of modules over the open subscheme $\bigcup L \subseteq X$. We ought to prove that $L = I$ and we again do so using Zorn's lemma.
\end{proof}

\section{Exact categories of Grothendieck type}
\label{sec:groth-exact}

In various contexts (see~\cite{SaSt,StPo} for example), it is useful to consider more general categories than Grothendieck categories. The rest of the text, however, is perfectly relevant when read as if it were written for Grothendieck categories or even for module categories. Thus, the reader who wishes to avoid the related technicalities may skip the section and read further from Section~\ref{sec:wfs}.

\subsection{Efficient exact categories}
\label{subsec:efficient}

In order describe our object of interest, we recall some terminology. The central concept is that of an exact category, which is originally due to Quillen, but the common reference for a simple axiomatic description is~\cite[Appendix A]{Kst} and an extensive treatment is given in~\cite{Bu}.

An \emph{exact category} is an additive category $\E$ together with a distinguished class of diagrams of the form
\[ 0 \la X \overset{i}\la Y \overset{d}\la Z \la 0, \]
called \emph{conflations}, satisfying certain axioms which make conflations behave similar to short exact sequences in an abelian category and allow to define Yoneda Ext groups with usual properties (see Section~\ref{sec:cotorsion}). Adopting the terminology from~\cite{Kst}, the second map in a conflation (denoted by $i$) is called \emph{inflation}, while the third map (denoted by $d$) is referred to as \emph{deflation}.

Morally, an exact category is an extension closed subcategory of an abelian category, which is made precise in the following statement.

\begin{prop} \label{prop:exact-categories} \cite{Kst,Bu} ~
\begin{enumerate}
\item Let $\A$ be an abelian category. Then $\A$ considered together with all short exact sequences as conflations is an exact category.
\item Let $\E$ be an exact category and $\E' \subseteq \E$ be an extension closed subcategory (i.e.\ if $0 \to X \to Y \to Z \to 0$ is a conflation and $X,Z \in \E'$, then $Y \in \E'$). Then $\E'$, considered together with all conflations in $\E$ whose all terms belong to $\E'$, is again an exact category.
\item Every small exact category arises up to equivalence as an extension closed subcategory of an abelian category in the sense of~(1) and~(2).
\end{enumerate}
\end{prop}

For the results presented later to work, we need to impose extra conditions on our exact categories. Long story short---we need to impose requirements on the exact category which make it sufficiently resemble a Grothendieck category. As the requirement that $\E$ be a cocomplete category seems too restrictive in practice, we need to specify first which direct limits we are interested in in connection with the analogue of the left exactness property (Definition~\ref{def:groth}(Gr2)). In order to do so, the following definition is handy.

\begin{defn} \label{def:transf}
Let $\C$ be an arbitrary category, let $\lambda$ be an ordinal number, and let $(X_\alpha, f_{\alpha\beta})_{\alpha<\beta<\lambda}$ be a direct system indexed by $\lambda$ in a category $\C$:
\[
\xymatrix{
X_0 \ar[r]^{f_{01}}
\ar@/_1pc/[rr]_(.85){f_{02}} \ar@/_1.8pc/[rrr]_(.8){f_{03}}
\ar@/_2.8pc/[rrrrr]_(.85){f_{0\omega}} \ar@/_3.5pc/[rrrrrr]_(.85){f_{0,\omega+1}} &
X_1 \ar[r]^{f_{12}} &
X_2 \ar[r]^{f_{23}} &
X_3 \ar[r] &
\cdots \ar[r] &
X_\omega \ar[r]^{f_{\omega,\omega+1}} &
X_{\omega+1} \ar[r] &
\cdots \\
~
}
\]
Such a system is called a \emph{$\lambda$-sequence} if for each limit ordinal $\mu<\lambda$, the object $X_\mu$ together with the morphisms $f_{\alpha\mu}\dd X_\alpha \to X_\mu$, $\alpha<\mu$, is a colimit of the direct subsystem $(X_\alpha, f_{\alpha\beta})_{\alpha<\beta<\mu}$. From now on, whenever we are going to depict a $\lambda$-sequence, we are going to draw only the morphism of the form $f_{\alpha,\alpha+1}$.

The \emph{composition} of the $\lambda$-sequence is the colimit morphism
\[ X_0 \la \li_{\alpha<\lambda} X_\alpha, \]
if it exists in $\C$.

Finally, if $\I$ is a class of morphisms of $\C$, then a \emph{transfinite compositions} of morphisms of $\I$ are defined as the compositions of $\lambda$-sequences $(X_\alpha, f_{\alpha\beta})_{\alpha<\beta<\lambda}$ with $f_{\alpha,\alpha+1} \in \I$ for every $\alpha+1<\lambda$.
\end{defn}

For technical reasons for Quillen's small object argument in Section~\ref{sec:wfs}, we need one more definition. The necessary set theoretical concepts can be found in~\cite{Jech}.

\begin{defn} \label{def:small}
If $\C$ is a category, $\kappa$ is a cardinal number and $\D$ is a class of morphisms of $\C$, then an object $X \in \C$ is called \emph{$\kappa$-small relative to $\D$} if, for every infinite regular cardinal $\lambda\ge\kappa$ and every $\lambda$-sequence
\[ E_0 \la E_1 \la E_2 \la \cdots \la E_\alpha \la E_{\alpha+1} \la \cdots \]
in $\C$ \st $f_{\alpha,\alpha+1}\dd E_\alpha \to E_{\alpha+1}$ is in $\D$ for all $\alpha+1<\lambda$, the canonical map of sets
\[ \li_{\alpha<\lambda} \C(X,E_\alpha) \la \C(X,\li_{\alpha<\lambda} E_\alpha) \]
is an isomorphism.

The object $X$ is called \emph{small relative to $\D$} if it is $\kappa$-small relative to $\D$ for some cardinal $\kappa$.
\end{defn}

Then, modifying slightly the corresponding concept from~\cite{SaSt}, we state a first version of our specification which exact categories we are interested in.

\begin{defn} \label{def:efficient-exact}
An exact category $\E$ is called \emph{efficient} if

\begin{enumerate}
\item[(Ef0)] $\E$ is weakly idempotent complete. That is, every section $s\dd X \to Y$ in $\E$ has a cokernel or, equivalently by~\cite[Lemma 7.1]{Bu}, every retraction $r\dd Y \to Z$ in $\E$ has a kernel.
\item[(Ef1)] Arbitrary transfinite compositions of inflations exist and are themselves inflations.
\item[(Ef2)] Every object of $\E$ is small relative to the class of all inflations.
\item[(Ef3)] $\E$ admits a generator. That is, there is an object $G \in \E$ \st every $X \in \E$ admits a conflation $G^{(I)} \to X \to 0$.
\end{enumerate}
\end{defn}

Note that~(Ef1) is a weak analogue of cocompleteness of $\E$ and left exactness of direct limits, while (Ef2) is a technical condition necessary for the small object argument in Corollary~\ref{cor:small-obj-exact}. Typical examples of efficient exact categories which we have in mind are Grothendieck categories (see Proposition~\ref{prop:exact-Grothendieck} below) and various Frobenius exact categories used for the construction of algebraic triangulated categories (see for example~\cite[Theorem 4.2]{SaSt} or~\cite[Remark 2.15]{StPo}). A systematic way of constructing more examples is given in~\S\ref{subsec:groth-type} using the notion of deconstructible classes.

Let us look closer at what we can say about efficient exact categories in general. An important fact whose proof is postponed to Corollary~\ref{cor:small-ext} is that the Yoneda Ext groups are always \emph{sets} rather than proper classes. This is not a priori clear since efficient exact categories are practically never small. We also have infinite coproducts and these are exact.

\begin{lem} \cite[Lemma 1.4]{SaSt} \label{lem:efficient-coprod}
Let $\E$ be an exact category satisfying (Ef1) of Definition~\ref{def:efficient-exact}. Then the following hold:
\begin{enumerate}
\item The category $\E$ has small coproducts.
\item Small coproducts of conflations are conflations.
\end{enumerate}
\end{lem}

Condition (Ef0) is important because of the consequence stated in the next lemma.

\begin{lem} \cite[Proposition 7.6]{Bu} \label{lem:id-compl}
Let $\E$ be a weakly idempotent complete exact category and $f\dd X \to Y$ and $g\dd Y \to Z$ be a pair of composable morphisms. Then:
\begin{enumerate}
\item If $gf$ is a inflation, then $f$ is a inflation.
\item If $gf$ is a deflation, then $g$ is a deflation.
\end{enumerate}
\end{lem}

Beware, however, that unlike for Grothendieck categories an analog of Theorem~\ref{thm:enough-inj} (enough injective objects) does not hold for efficient exact categories---see coming Example~\ref{expl:bad-efficient}. We are going to give a remedy for that in the next subsection, at the cost of imposing another condition on~$\E$.

\subsection{Exact categories of Grothendieck type}
\label{subsec:groth-type}

Here we are going to define exact categories of Grothendieck type which are even closer analogs of Grothendieck exact categories in that they always have enough injectives. Moreover, we shall describe a systematic procedure of producing these categories which seems to encompass all existing examples so far. The central concepts here are those of a filtration and deconstructibility.

\begin{defn} \label{def:filtr}
Let $\E$ be an exact category and $\clS$ be a class of objects in $\E$. By an \emph{$\clS$-filtration} we mean a $\lambda$-sequence 
\[
\xymatrix@1{
0 = X_0 \ar[r]^{f_{01}} &
X_1 \ar[r]^{f_{12}} &
X_2 \ar[r]^{f_{23}} &
X_3 \ar[r] &
\cdots \ar[r] &
X_\omega \ar[r]^{f_{\omega,\omega+1}} &
X_{\omega+1} \ar[r] &
\cdots
}
\]
\st $X_0 = 0$ and for each $\alpha+1<\lambda$, the morphism $f_{\alpha,\alpha+1}$ is an inflation whose cokernel belongs to $\clS$. That is, we have conflations
\[
0 \la X_\alpha \overset{f_{\alpha,\alpha+1}}\la X_{\alpha+1} \la S_\alpha \la 0
\]
with $S_\alpha \in \clS$.

An object $X \in \E$ is called \emph{$\clS$-filtered} if $0 \to X$ is the composition (in the sense of Definition~\ref{def:transf}) of some $\clS$-filtration. The class of all $\clS$-filtered objects will be denoted by $\Filt\clS$.
\end{defn}

\begin{rem} \label{rem:filtr}
Informally, an $\clS$-filtration is just a transfinite extension of objects from $\clS$. If, say, $\E = \ModR$ is the category of right $R$-modules for a ring $R$ and we consider $\ModR$ with the abelian exact structure (i.e.\ we take precisely all short exact sequences as conflations), then $X$ is $\clS$-filtered \iff there is a well ordered continuous chain $(X_\alpha \mid \alpha\le \lambda)$ of submodules of $X$ \st $X_0 = 0$, $X_\lambda = X$ and $X_{\alpha+1}/X_\alpha$ is isomorphic to an object of $\clS$ for each $\alpha+1\le\lambda$. This rather intuitive notion has been already successfully used for some time; see~\cite{GT} and references there. As observed in~\cite{SaSt}, Definition~\ref{def:filtr} is a fairly well-behaved generalization for efficient exact categories, although some care is due. For example, some properties of the generalized filtrations which would be clear for filtrations in module categories require a non-trivial proof for efficient exact categories---\cite[Lemma 2.10]{SaSt} serves as an example.
\end{rem}

A closely related concept is a deconstructible class in an exact category.

\begin{defn} \label{def:deconstructible}
Let $\E$ be an exact category and $\F \subseteq \E$ be a class of objects. Then $\F$ is called \emph{deconstructible} if there exists a set (not a proper class!) $\clS \subseteq \E$ of objects \st $\F = \Filt\clS$.
\end{defn}

Let us summarize some basic properties of deconstructible classes, which we shall use freely in the sequel.

\begin{lem} \label{lem:deconstr-basic}
Let $\E$ be an exact category satisfying (Ef1) of Definition~\ref{def:efficient-exact} and let $\F \subseteq \E$ be a deconstructible class. Then any $\F$-filtered object of $\E$ belongs to $\F$. In particular, $\F$ is closed in $\E$ under taking coproducts and extensions.
\end{lem}

\begin{proof}
We refer to~\cite[Corollary 2.11]{SaSt} for the fact that $\F = \Filt\F$. While at the level of generality which we assume some work is necessary, the equality is an easy exercise in the special case where $\E = \ModR$ with the abelian exact structure. The fact that $\F$ is closed under extensions is then clear, and the closure under coproducts follows from the construction of coproducts via filtrations in the proof of~\cite[Lemma 1.4]{SaSt}.
\end{proof}

Now we can define exact categories of Grothendieck type.

\begin{defn} \label{def:Grothendieck-exact}
An exact category $\E$ is said to be of \emph{Grothendieck type} if

\begin{enumerate}
\item[(GT0)] $\E$ is weakly idempotent complete.
\item[(GT1)] Arbitrary transfinite compositions of inflations exist and are inflations.
\item[(GT2)] Every object of $\E$ is small relative to the class of all inflations.
\item[(GT3)] $\E$ admits a generator.
\item[(GT4)] $\E$ is deconstructible in itself. That is, there exists a set of objects $\clS \subseteq \E$ \st $\E = \Filt\clS$.
\end{enumerate}
\end{defn}

In other words, $\E$ is of Grothendieck type if $\E$ is efficient and satisfies (GT4). We recalled the former axioms for the reader's convenience. As already mentioned, the key consequence of (GT4) is that $\E$ has enough injective objects. We postpone the proof of the fact to Corollary~\ref{cor:enough-inj} when we will have developed the necessary theory. Here we rather focus on how exact categories of Grothendieck type occur, but first we show a non-example demonstrating that (GT4) is indeed necessary.

\begin{expl} \label{expl:bad-efficient}
Let $\D \subseteq \Ab$ be the category of all flat Mittag-Leffler abelian groups in the sense of~\cite{RG}. These groups are characterized by the property that every countable subgroup is free; see~\cite[Proposition 7]{AzFa}. It is not difficult to prove that $\D$ is closed under retracts, any $\D$-filtered object belongs to $\D$, and that $\Z$ is a generator for $\D$. Hence $\D$ is an efficient exact category. On the other hand, $\D$ cannot be deconstructible in itself by~\cite[Corollary 7.3]{EGPT}. Even worse, the only injective object of $\D$ is the zero object, so $\D$ does not have enough injectives. To see that, observe that by~\cite[Theorem 5.5]{EGPT} $X \in \D$ is injective in $\D$ \iff $X$ is a so-called cotorsion group. However, the only flat Mittag-Leffler cotorsion group is by~\cite[Corollary V.2.10(ii)]{EM} the zero group.
\end{expl}

Turning back to examples, all Grothendieck categories are exact categories of Grothendieck type. In order to show that, we recall a few standard facts. One can define a subobject of an object $X$ in an abelian category $\G$ as an equivalence class of monomorphisms $Y \to X$, \cite[\S IV.2]{S}. We shall as usual write $Y \subseteq X$ in such a case. If $\G$ is a Grothendieck category, then the subobjects of a given object form a modular upper continuous complete lattice $(\Subobj X, +, \cap)$, \cite[Propositions IV.5.3 and V.1.1 (c)]{S}. Recall a complete lattice $(\clL, \vee, \wedge)$ is \emph{upper continuous} (cf.~\cite[\S III.5]{S}) if $(\bigvee_{d \in D} d) \wedge a = \bigvee_{d \in D} (d \wedge a)$ whenever $a \in \clL$ and $D \subseteq \clL$ is a directed subset.

\begin{prop} \label{prop:exact-Grothendieck}
Let $\G$ be a Grothendieck category considered with the abelian exact structure. Then $\G$ is of Grothendieck type as an exact category.
\end{prop}

\begin{proof}
Condition (GT0) is clearly true for every abelian category. Since $\G$ is cocomplete and has exact direct limits, (GT1) holds. The existence of a generator (GT3) is a part of the definition of a Grothendieck category. Regarding the smallness assumption (GT2), it is a well-known fact that any object $A \in \G$ is small relative to the class \emph{all} morphisms. Indeed, the Popescu-Gabriel theorem~\cite[Theorem X.4.1]{S} guarantees that for any fixed $A \in \G$, there is a cardinal number $\kappa$ \st
\begin{enumerate}
\item $\G$ identifies with a full subcategory of $\ModR$, the category of right modules for some ring $R$, \st $\G$ is closed in $\ModR$ under taking colimits of $\lambda$-sequences for every infinite regular cardinal $\lambda\ge\kappa$, and
\item under this identification $A$ can be presented, as an $R$-module, by fewer than $\kappa$ generators and relations, so that the functor
\[ \Hom_R(A,-)\dd \ModR \la \Ab \]
commutes with colimits of $\lambda$-sequences for every infinite regular $\lambda\ge\kappa$.
\end{enumerate}

Finally, let $\clS$ be a representative set of quotients $\{G/Y \mid Y \subseteq G\}$, where $G$ is a generator of $\G$. Given any object $X \in \G$ and a fixed epimorphism $p\dd G^{(I)} \to X$, we shall construct an $\clS$-filtration of $X$. This will prove (GT4). To this end, we can without loss of generality assume that $I = \lambda$ is an ordinal number and define $X_\alpha \subseteq X$ for each $\alpha\le\lambda$ as the image of restriction of $p$ to $G^{(\alpha)}$. It is not difficult to convince oneself that $(X_\alpha)_{\alpha\le\lambda}$ is (by slightly abusing the notation) an $\clS$-filtration of $X$.
\end{proof}

Our next goal is to prove that every deconstructible class of a Grothendieck category, viewed as a full subcategory, is naturally an exact category of Grothendieck type. In order to do so, we need a technical tool: the generalized Hill Lemma. Here we take a slightly restricted version of the result from~\cite{St2}, where also references to other versions and evolution stages of the Hill Lemma can be found.

\begin{prop} \label{prop:hill-lemma}
Let $\G$ be a Grothendieck category and $\clS$ be a set of objects. Then there exists an infinite regular cardinal $\kappa$ with the following property: For every $X \in \G$, which is the union of an $\clS$-filtration
\[
0 = X_0 \subseteq X_1 \subseteq X_2 \subseteq \dots
\subseteq X_\alpha \subseteq X_{\alpha+1} \subseteq \dots
\subseteq X_\sigma = X
\]
for some ordinal $\sigma$, there is a complete sublattice $\clL$ of the power set $\big(\Pow\sigma,\cup,\cap \big)$ of $\sigma$ and a map
\[ \ell: \clL \la \Subobj X \]
which assigns to each $S \in \clL$ a subobject $\ell(S)$ of $X$, \st the following hold:
\begin{enumerate}
\item[(H1)] For each $\alpha \leq \sigma$ we have $\alpha = \{\gamma \mid \gamma<\alpha \} \in \clL$ and $\ell(\alpha) = X_\alpha$.

\item[(H2)] If $(S_i)_{i \in I}$ is any family of elements of $\clL$, then $\ell(\bigcup S_i) = \sum \ell(S_i)$ and $\ell(\bigcap S_i) = \bigcap \ell(S_i)$. In particular, $\ell$ is a complete lattice homomorphism from $(\clL,\cup,\cap)$ to the lattice $(\Subobj X, +, \cap)$ of subobjects of $X$.

\item[(H3)] If $S,T \in \clL$ are \st $S \subseteq T$, then the object $N = \ell(T)/\ell(S)$ is $\clS$-filtered. More precisely, there is an $\clS$-filtration $(N_\beta \mid \beta\le\tau)$ and a bijection  $b: T \setminus S \to \tau$ $(= \{\beta \mid \beta<\tau\})$ \st $X_{\alpha+1}/X_\alpha \cong N_{b(\alpha)+1}/N_{b(\alpha)}$ for each $\alpha \in T \setminus S$.

\item[(H4)] For each $<\kappa$-generated subobject $Y \subseteq X$ there is $S \in \clL$ of cardinality $<\kappa$ (so that $\ell(S)$ admits an $\clS$-filtration of length $<\kappa$ by \emph{(H3)}) \st $Y \subseteq \ell(S) \subseteq X$.
\end{enumerate}
\end{prop}

\begin{proof}
The proof, which is rather technical, can be found in~\cite[Theorem 2.1]{St2}. While omitting the argument here, we shall at least indicate the construction of the class $\clL$ and explain the main idea behind the proof in the simplest non-trivial case in Example~\ref{expl:hill-lemma}. First we need to choose $\kappa$ given $\G$ and $\clS$. The only condition for this is that $\G$ is locally $<\kappa$-presentable and each $Y \in \clS$ is $<\kappa$-presented in the sense of~\cite{AR,GU71}. Given the $\clS$-filtration $(X_\alpha \mid \alpha \le \sigma)$ of $X$, we can fix a family $(A_\alpha)_{\alpha<\sigma}$ of $<\kappa$-generated subobjects of $X$ (again in the sense of~\cite{GU71}) \st $X_{\alpha+1} = X_\alpha + A_\alpha$ for each $\alpha<\sigma$. Then we call a subset $S \subseteq \sigma$ \emph{closed} if every $\alpha \in S$ satisfies
\[ X_\alpha \cap A_\alpha \subseteq \sum_{^{\gamma \in S,}_{\gamma < \alpha}} A_\gamma. \]
The set $\clL \subseteq \Pow\sigma$ can be chosen as $\clL = \{ S \subseteq \sigma \mid S \textrm{ is closed} \}$ and the map $\ell$ assigns to $S \in \clL$ the subobject $\sum_{\alpha \in S} A_\alpha$.
\end{proof}

\begin{expl} \label{expl:hill-lemma}
Suppose that $R$ is a ring, $\clS$ is a collection of finitely presented modules, and that $X \in \ModR$ has an $\clS$-filtration of the form
\[
0 = X_0 \subseteq X_1 \subseteq X_2 \subseteq \dots \subseteq X_\omega \subseteq X_{\omega+1} = X
\]
We will choose $\kappa = \aleph_0$ and fix a sequence of finitely generated submodules $(A_\alpha)_{\alpha\le\omega}$, and $\clL$ and $\ell$ as above. In order to prove (H4), suppose that $Y \subseteq X$ is finitely generated. If $Y \subseteq X_\omega$, then $Y \subseteq X_n$ for some $n<\omega$ and can put $S = \{0,1,\dots,n-1\}$. If $Y \not\subseteq X_\omega$, we can without loss of generality assume that $A_\omega \subseteq Y$ since otherwise we could replace $Y$ by $Y+A_\omega$. Then we have a short exact sequence $0 \to X_\omega \cap Y \to Y \to X/X_\omega \to 0$ and so $X_\omega \cap Y$ is finitely generated by~\cite[Ch. 5, \S25]{Wis}. In particular there must exist $n_0 < \omega$ \st $X_\omega \cap Y \subseteq X_{n_0}$ and the set $S = \{0,1,\dots,n_0-1,\omega\}$ will do.
\end{expl}

The significance of Proposition~\ref{prop:hill-lemma} is that, starting with one filtration of $X$, it allows us to construct many more filtrations of $X$. The crucial point is property~(H4).

As a first application, we shall prove a theorem which provides us with a major source of examples of exact categories of Grothendieck type.

\begin{thm} \label{thm:deconstr-Groth}
Let $\G$ be a Grothendieck category and $\E \subseteq \G$ be a deconstructible class which is closed under retracts. Then $\E$ together with the collection of all short exact sequences in $\G$ whose all terms belong to $\E$ (cf.\ Proposition~\ref{prop:exact-categories}) is an exact category of Grothendieck type.
\end{thm}

\begin{proof}
(GT0) and (GT4) are clear from the definition, and (GT1) and (GT2) have been proved in~\cite[Lemma 1.11(i)]{StPo}. The main difficulty remains in proving (GT3), the existence of a generator for $\E$.

Denote by $\clS$ a set of objects of $\E$ \st $\E = \Filt\clS$ and let $\kappa$ be a regular cardinal which is in accordance with Proposition~\ref{prop:hill-lemma}. Let $\E_\kappa$ be the class of all objects in $\G$ which have an $\clS$-filtration of length $<\kappa$. One observes that $\E_\kappa$ has only a set of representatives up to isomorphism, and we denote by $G$ the sum of all such representatives.

Clearly $G \in \E$ by Lemma~\ref{lem:deconstr-basic}, and we claim that $G$ is a generator for $\E$. We must prove that each $X \in \E$ admits a short exact sequence in $\G$ of the form
\[ 0 \la K \la G^{(I)} \la X \la 0 \]
with $K \in \E$, and we shall do so by induction on the the least ordinal $\sigma = \sigma(X)$ \st $X$ has an $\clS$-filtration of length $\sigma$ in $\G$. More precisely, given $X \in \E$ with an $\clS$-filtration
\[
0 = X_0 \subseteq X_1 \subseteq X_2 \subseteq \dots
\subseteq X_\alpha \subseteq X_{\alpha+1} \subseteq \dots
\subseteq X_\sigma = X
\]
of length $\sigma$, we shall inductively construct in $\G$ a well-ordered direct system of short exact sequences
\[ \Big( \ep_\alpha \dd \quad 0 \la K_\alpha \la \coprod_{\beta<\alpha} G_\beta \overset{p_\alpha}\la X_\alpha \la 0 \Big)_{\alpha\le\sigma} \]
\st
\begin{enumerate}
\item $K_\alpha \in \E$ for each $\alpha\le\sigma$,
\item $G_\beta \in \E_\kappa$ for each $\beta<\sigma$, and
\item the vertical morphisms in the diagram
\[
\begin{CD}
\phantom{_{+1}}\ep_\alpha\dd \quad 0 @>>> K_\alpha     @>>>   \coprod_{\beta<\alpha} G_\beta   @>{p_\alpha}>>     X_\alpha     @>>> 0   \\
                                   @.     @VVV                       @VVV                                         @VVV                  \\
           \ep_{\alpha+1}\dd \quad 0 @>>> K_{\alpha+1} @>>>   \coprod_{\beta\le\alpha} G_\beta @>{p_{\alpha+1}}>> X_{\alpha+1} @>>> 0   \\
\end{CD}
\]
are monomorphisms for each $\alpha+1 \le \sigma$. Moreover, the leftmost one has a cokernel in $\E$, the middle one is the canonical split monomorphism, and the rightmost one is the inclusion from the filtration of $X$.
\end{enumerate}

Before starting the construction, we shall fix a lattice map $\ell\dd \clL \to \Subobj X$ as in Proposition~\ref{prop:hill-lemma}.
Now for $\sigma = 0$ we start with the exact sequence of zeros and for the limit steps we just take the colimit sequences. Suppose that $\ep_\alpha$ has been constructed and we are to construct $\ep_{\alpha+1}$. By (H4) of Proposition~\ref{prop:hill-lemma}, there is an element $S \in \clL \subseteq \Pow\sigma$ \st $\alpha \in S$ and $\card{S} < \kappa$. Without loss of generality $S \subseteq (\alpha+1)$, since we can possibly replace $S$ by $S \cap (\alpha+1)$ thanks to~(H1) and~(H2). We shall put $G_\alpha = \ell(S)$---this is up to isomorphism an element of $\E_\kappa$.  Invoking~(H1) and~(H2) again, we have a bicartesian square of inclusions in $\G$ of the form
\[
\begin{CD}
\ell(S\setminus\{\alpha\}) @>>> G_\alpha       \\
@VVV                            @VV{i}V        \\
X_\alpha                   @>>> X_{\alpha+1}   \\
\end{CD}
\eqno{(\dag)}
\]
Let us define $K_{\alpha+1}$ by the pullback of $i\dd G_\alpha \to X_{\alpha+1}$ and the composition $t\dd \coprod_{\beta<\alpha} G_\beta \overset{p_\alpha}\to X_\alpha \overset{\subseteq}\to X_{\alpha+1}$:
\[
\begin{CD}
  @.   K_\alpha @>{j}>> K_{\alpha+1}                     @>>>    G_\alpha     @>{q}>> X_{\alpha+1}/X_\alpha @.       \\
@.     @|               @VVV                                     @V{i}VV              @|                             \\
0 @>>> K_\alpha @>>>    \coprod_{\beta<\alpha} G_\beta   @>{t}>> X_{\alpha+1} @>>>    X_{\alpha+1}/X_\alpha @>>> 0.  \\
\end{CD}
\]
As $\Img t = X_\alpha$ and $\Img i \supseteq A_\alpha$ (using the notation from the proof of Proposition~\ref{prop:hill-lemma}), we have $\Img t + \Img i = X_{\alpha+1}$ and we obtain a short exact sequence
\[
\begin{CD}
\ep_{\alpha+1} \dd \quad
0 @>>> K_{\alpha+1} @>>> \big(\coprod_{\beta<\alpha} G_\beta\big) \oplus G_\alpha @>{(t,i)}>> X_{\alpha+1} @>>> 0.
\end{CD}
\]
Further, $j\dd K_\alpha \to K_{\alpha+1}$ is clearly a monomorphism and it follows from the computations above that $q\dd G_\alpha \to X_{\alpha+1}/X_\alpha$ is an epimorphism. Now it only remains to prove that $\Coker j \in \E$. Appealing to $(\dag)$, one sees that $\Ker q \cong \ell(S\setminus\{\alpha\}) \in \E$, so that we have a short exact sequence
\[ 0 \la K_\alpha \overset{j}\la K_{\alpha+1} \la \ell(S\setminus\{\alpha\}) \la 0, \]
as required.
\end{proof}

\begin{expl} \label{expl:flat}
Let $R$ be a ring, $\G = \ModR$ and $\F = \FlatR$, the category of flat right $R$-modules. Then $\F$ satisfies the assumptions of Theorem~\ref{thm:deconstr-Groth} (see for instance~\cite[Lemma 3.2.7 and Theorem 4.1.1]{GT}) and so it is, with the exact structure induced from $\G$, an exact category of Grothendieck type. It came as a byproduct of Enoch's proof of the Flat Cover Conjecture~\cite{BEE} that $\F$ has enough injectives---these are usually called flat cotorsion modules.
\end{expl}

We conclude the section with a few more properties of deconstructible classes, which among others show that the assumption on $\E$ being closed under retracts in Theorem~\ref{thm:deconstr-Groth} is not very restrictive.

\begin{lem} \label{lem:deconstr-transitive} \cite[Lemma 1.11]{StPo}
Deconstructibility is transitive. Namely, let $\E$ be an exact category satisfying (Ef1) of Definition~\ref{def:efficient-exact}. If $\F$ is deconstructible in $\E$, then $\F$ also satisfies (Ef1) and, moreover, a class $\F' \subseteq \F$ is deconstructible in $\F$ \iff it is deconstructible in $\E$.
\end{lem}

\begin{prop} \label{prop:deconstr-advanced} \cite[Proposition 2.9]{St2}
Let $\G$ be a Grothendieck category, considered as an exact category with the abelian exact structure. Then:
\begin{enumerate}
\item The closure of a deconstructible class of $\G$ under retracts is deconstructible.
\item The intersection $\bigcap_{i \in I} \F_i$ of a collection of deconstructible classes $(\F_i \mid i \in I)$ of $\G$, indexed by a set $I$, is deconstructible.
\end{enumerate}
\end{prop}

\section{Weak factorization systems}
\label{sec:wfs}

The coming section deals with a rather abstract category theory which will be of use later in the construction of cotorsion pairs and model structures. The highlight is a version of Quillen's small object argument, specialized to nice enough exact categories. We omit several technical steps in proofs and refer to the monographs~\cite{H2,Hir} and paper~\cite{SaSt}. Although most of our arguments are included in these references, the term ``weak factorization system'' itself is not. This term and some notation has been taken from~\cite{AHRT}.

We start with an orthogonality relation on morphisms in an arbitrary category~$\C$.

\begin{defn} \label{def:lifting}
Given morphisms $f\dd A \to B$ and $g\dd X \to Y$ in $\C$, we write $f \lifts g$ if for any commutative square given by the solid arrows
\[
\xymatrix{
A \ar[r] \ar[d]_f     & X\phantom{,} \ar[d]^g      \\
B \ar[r] \ar@{.>}[ur] & Y,
}
\]
a morphism depicted by the diagonal dotted arrow exists \st both the triangles commute. We stress that we require only existence, not uniqueness of such a morphism. Note that if $\C$ is additive and we view $f$ and $g$ as complexes concentrated in two degrees, the orthogonality relation precisely says that every map from $f$ to $g$ is null-homotopic.

Given $f$, $g$ \st $f \lifts g$, we say that
$f$ has the \emph{left lifting property} for $g$ and
$g$ has the \emph{right lifting property} for $f$.
\end{defn}

Now we can define the central concept of the section.

\begin{defn} \label{def:wfs}
Let $\C$ be a category and $(\clL,\R)$ be a pair of classes of morphisms in $\C$. We say that $(\clL,\R)$ is a \emph{\wfs} if
\begin{enumerate}
\item[(FS0)] $\clL$ and $\R$ are closed under retracts. That is, given any commutative diagram
\[
\begin{CD}
A   @>>>   X   @>>>   A   \\
@V{h}VV  @V{f}VV  @V{h}VV \\
B   @>>>   Y   @>>>   B   \\
\end{CD}
\]
\st $f \in \clL$ and the rows compose to the identity morphisms, then $h \in \clL$ as well. We require the same for $\R$.

\item[(FS1)] $f \lifts g$ for all $f \in \clL$ and $g \in \R$.

\item[(FS2)] For every morphism $h\dd X \to Y$ in $\C$, there is a factorization
\[
\xymatrix{
X \ar[rr]^h \ar[dr]_f && Y  \\
& Z \ar[ur]_g
}
\]
with $f \in \clL$ and $g \in \R$.
\end{enumerate}

We say that $(\clL,\R)$ is a \emph{functorial \wfs} if the factorization in~(FS2) can be chosen functorially in $h$.
\end{defn}

\begin{rem} \label{rem:wfs-unique}
Although the factorization as in (FS2), be it functorial or not, is typically non-unique in the cases we are concerned with, we can ``compare'' different factorizations using the lifting property. Namely, given $h = gf = \tilde{g}\tilde{f}$ \st $f,\tilde{f} \in \clL$ and $g,\tilde{g} \in \R$, (FS1) in Definition~\ref{def:wfs} ensures that there is a morphism depicted by the dotted arrow making the diagram commutative:
\[
\xymatrix{
X \ar[rr]^f \ar@/_.7pc/[rrd]_{\tilde f} && Z \ar@{.>}[d] \ar[rr]^g && Y   \\
&& \tilde{Z} \ar@/_.7pc/[rru]_{\tilde g}
}
\]
\end{rem}

\begin{expl} \label{expl:wfs-misleading}
A well-known example of a \wfs in an abelian category $\A$ is $(\E,\M)$, where $\E$ is the class of all epimorphisms and $\M$ is the class of all monomorphisms. However, this example is rather misleading in our context. In Section~\ref{sec:cotorsion}, we will see \wfss $(\clL,\R)$ where $\clL$ is a class of monomorphisms and $\R$ is a class of epimorphisms. 

Another fact is that $(\E,\M)$ satisfies a stronger version of the lifting property from Definition~\ref{def:lifting}: the diagonal morphism is \emph{unique} for every square with $f \in \E$ and $g \in \M$, ensuring in view of Remark~\ref{rem:wfs-unique} that the factorizations as in (FS2) from Definition~\ref{def:wfs} are also unique. This is a strong property which our \wfss typically do not enjoy.
\end{expl}

A relatively easy observation concerning \wfss is, that the two classes of morphisms determine each other.

\begin{lem} \label{lem:wfs-ortho}
Let $(\clL,\R)$ be a \wfs in a category $\C$. Then
\[
\clL = \{ f \mid f \lifts g \textrm{ for all } g \in \R   \}
\quad \textrm{ and } \quad
\R   = \{ g \mid f \lifts g \textrm{ for all } f \in \clL \}.
\]
\end{lem}

\begin{proof}
Clearly $\clL \subseteq \{ f \mid f \lifts g \textrm{ for all } g \in \R \}$ and we must prove the other inclusion. Take any $h\dd X \to Y$ \st $h \lifts g$ for all $g \in \R$ and consider a factorization
\[
\xymatrix{
X \ar[rr]^h \ar[dr]_f && Y  \\
& Z \ar[ur]_g
}
\]
with $f \in \clL$ and $g \in \R$. Since $h \lifts g$, the dotted arrow making the following diagram commutative exists
\[
\xymatrix{
X \ar[r]^f \ar[d]_h   & Z \ar[d]^g      \\
Y \ar@{=}[r] \ar@{.>}[ur]|\hole \ar@{}[ur]|s & Y
}
\]
Placing the morphisms we have considered so far into the following diagram
\[
\begin{CD}
X   @=        X   @=        X\phantom{,}  \\
@V{h}VV     @V{f}VV     @V{h}VV           \\
Y   @>{s}>>   Z   @>{g}>>   Y,            \\
\end{CD}
\]
we observe that $gs = 1_Y$. Thus, $h$ is a retract of $f$ and as such it must belong to $\clL$.

The argument for $\R = \{ g \mid f \lifts g \textrm{ for all } f \in \clL \}$ is dual.
\end{proof}

Now we can deduce closure properties of the left orthogonal of a class of morphisms \wrt $\lifts$. Note in particular that the left hand side class of any \wfs $(\clL,\R)$ has these properties.

\begin{lem} \label{lem:wfs-closure}
Let $\R$ be a class of morphisms in a category $\C$ and denote
\[ \clL = \{ f \mid f \lifts g \textrm{ for each } g \in \R \}. \]
Then the following hold for $\clL$:
\begin{enumerate}
\item $\clL$ is closed under pushouts. That is, if we are given a diagram
\[
\begin{CD}
A @>>> \tilde A \\
@V{f}VV         \\
B
\end{CD}
\]
with $f \in \clL$ and if the pushout
\[
\begin{CD}
A   @>>>   \tilde{A}     \\
@V{f}VV  @VV{\tilde{f}}V \\
B   @>>>   \tilde{B}
\end{CD}
\]
exists in $\C$, then also $\tilde f \in \clL$.

\item $\clL$ is closed under transfinite compositions. That is, given a $\lambda$-sequence $(X_\alpha, f_{\alpha\beta})_{\alpha<\beta<\lambda}$ with $f_{\alpha,\alpha+1} \in \clL$ for every $\alpha+1<\lambda$, the composition $X_0 \to \li_{\alpha<\lambda} X_\alpha$, if it exists, belongs to $\clL$, too.
\end{enumerate}
\end{lem}

\begin{proof}
This is an easy fact whose proof is left to the reader.
\end{proof}

Inspired by the previous lemma, let us state another definition.

\begin{defn} \label{def:I-cell}
Given a set $\I$ of morphisms of a category $\C$, we define a \emph{relative $\I$-cell complex} as a transfinite composition of pushouts of morphisms from $\I$. The class of all relative $\I$-cell complexes will be denoted by $\Icell$.
\end{defn}

We are now ready to state the highlight of the section---Quillen's small object argument. We state it first in a very general form taken from~\cite{SaSt} and specialize it later. We will not give a full proof, but rather refer to the literature. A version of the argument can be found in~\cite{QHtp}, while nice treatments with notation very close to ours can be found in~\cite{Hir,H2}.

\begin{thm} \label{thm:small-obj}
Let $\C$ be a category and $\M$ be a class of morphisms \st
\begin{enumerate}
\item Arbitrary pushouts of morphisms of $\M$ exist and belong again to $\M$.
\item Arbitrary coproducts of morphisms of $\M$ exist and belong again to $\M$.
\item Arbitrary transfinite compositions of maps of $\M$ exist and belong to $\M$.
\end{enumerate}
Let $\I \subseteq \M$ be a set (not a proper class!) of morphisms and put
\begin{align*}
\R   &= \{ g \mid f \lifts g \textrm{ for all } f \in \I \},   \\
\clL &= \{ f \mid f \lifts g \textrm{ for all } g \in \R \}.
\end{align*}
If for every $i\dd A \to B$ in $\I$ the domain $A$ is small with respect to relative $\I$-cell complexes (Definition~\ref{def:small}), then $(\clL,\R)$ is a functorial \wfs in $\C$ and $\clL$ consists precisely of retracts of relative $\I$-cell complexes.
\end{thm}

\begin{proof}
Clearly, conditions~(FS0) and~(FS1) of Definition~\ref{def:wfs} are satisfied for $(\clL,\R)$, so we only have to prove the functorial version of condition~(FS2) and the structure of morphisms $\clL$. Here we refer to~\cite[Proposition 10.5.16 and Corollary 10.5.22]{Hir} or~\cite[Theorem 2.1.14 and Corollary 2.1.15]{H2}. Although the assumptions in~\cite{Hir,H2} differ slightly, identical proof works.
\end{proof}

Now we can specialize. First we state the correspondingly simplified version for Grothendieck categories.

\begin{cor} \label{cor:small-obj-Groth}
Let $\G$ be a Grothendieck category and $\I$ be a set of morphisms. If $(\clL,\R)$ is as in Theorem~\ref{thm:small-obj}, then $(\clL,\R)$ is a functorial \wfs in $\G$ and $\clL$ consists precisely of retracts of relative $\I$-cell complexes.
\end{cor}

\begin{proof}
Since $\G$ is cocomplete and has exact direct limits, hence also coproducts, by definition, we can take the class of all morphisms in $\G$ for $\M$. Regarding the smallness assumption on the domains of maps in $\I$, any object $A \in \G$ is small relative to the class \emph{all} morphisms---see the proof of Proposition~\ref{prop:exact-Grothendieck}.
\end{proof}

Another corollary of Theorem~\ref{thm:small-obj} can be stated for efficient exact categories (Definition~\ref{def:efficient-exact}), or even slightly more generally without requiring the existence of a generator.

\begin{cor} \label{cor:small-obj-exact}
Let $\E$ be an exact category \st
\begin{enumerate}
\item Arbitrary transfinite compositions of inflations exist and are again inflations.
\item Every object of $\E$ is small relative to the class of all inflations.
\end{enumerate}
Let $\I$ be a set of inflations and consider $(\clL,\R)$ as in Theorem~\ref{thm:small-obj}. Then $(\clL,\R)$ is a functorial \wfs in $\E$ and $\clL$ consists precisely of retracts of relative $\I$-cell complexes.
\end{cor}

\begin{proof}
In this case, we can take for $\M$ the class of all inflations. Then $\M$ is closed under taking pushouts by the axioms of an exact category, and $\M$ is closed under coproducts by Lemma~\ref{lem:efficient-coprod}.
\end{proof}

\section{Complete cotorsion pairs}
\label{sec:cotorsion}

In this section, we define complete cotorsion pairs and show how \wfss are related to them. Along with this, we shall prove some properties of exact categories promised in Section~\ref{sec:groth-exact}.

\subsection{Definitions}
\label{subsec:cotorsion-def}

Let us first briefly recall Yoneda's definition of Ext in exact categories, as our next discussion will be based on manipulation with these functors. We refer to~\cite[Chapter III and \S XII.5]{McL} for a detailed account on the topic and further properties of Yoneda's Ext. Given an exact category $\E$, a sequence of morphisms
\[
\ep\dd \quad 0 \la Z_n \la E_n \la \cdots \la E_1 \la Z_0 \la 0
\]
is called \emph{exact} if it arises by splicing $n$ conflations in $\E$ of the form
\[ 0 \la Z_i \la E_i \la Z_{i-1} \la 0. \]

Given further a pair of objects $X,Y \in \C$, let $E^n(X,Y)$ be the class of all exact sequences in $\E$ which are of the form
\[
\begin{CD}
\ep\dd \quad 0 @>>> Y @>>> E_n @>>> \cdots @>>> E_1 @>>> X @>>> 0.
\end{CD}
\]
We define an equivalence relation $\sim$ on $E^n(X,Y)$ by relating $\ep \sim \tilde\ep$ whenever we have a commutative diagram of the form
\[
\begin{CD}
\ep\dd       \quad 0 @>>> Y @>>>      E_n    @>>> \cdots @>>>     E_1     @>>> X @>>> 0  \\
                   @.    @|          @VVV            @.          @VVV         @|         \\
\tilde\ep\dd \quad 0 @>>> Y @>>> \tilde{E}_n @>>> \cdots @>>> \tilde{E}_1 @>>> X @>>> 0
\end{CD}
\]
and taking the symmetric and transitive closure. The $n$-th Yoneda Ext of $X$ and $Y$ is defined as $\Ext^n_\E(X,Y) = E^n(X,Y)/\sim$. Although it is not a priori clear whether the Yoneda Ext is a set (and indeed in general it may be a proper class even for $n=1$, see~\cite[Exercise 1, p. 131]{Fr}), we can nevertheless always give $\Ext^n_\E(X,Y)/\sim$ a structure of an abelian group using the so-called Baer sums and, whenever we are able to rule out the possible set theoretic pitfall,
\[ \Ext^n_\E(-,-)\dd \E\op \times \E \la \Ab \]
becomes an additive functor. Considering the case $n=1$, the zero element of $\Ext^1_\E(X,Y)$ is precisely the class of all split conflations.

Cotorsion pairs, originally defined in~\cite{Sal}, are simply pairs of classes of objects mutually orthogonal \wrt $\Ext^1_\E(-,-)$.

\begin{defn} \label{def:cotorsion}
Let $\E$ be an exact category. For a class $\clS$ of objects of $\E$ we define
\begin{align*}
\clS^\perp &= \{ B \in \E \mid \Ext^1_\E(S,B) = 0 \textrm{ for all } S \in \clS \},  \\
^\perp\clS &= \{ A \in \E \mid \Ext^1_\E(A,S) = 0 \textrm{ for all } S \in \clS \}.
\end{align*}

A pair $(\A,\B)$ of full subcategories of $\E$ is called a \emph{cotorsion pair} provided that
\[ \A = {^\perp \B} \qquad \textrm{ and } \qquad \A^\perp = \B. \]

A cotorsion pair $(\A,\B)$ is said to be \emph{complete} if every $X \in \E$ admits so-called \emph{approximation sequences}; that is, conflations of the form
\[
0 \la X \la B_X \la A_X \la 0 \quad \textrm{ and } \quad
0 \la B^X \la A^X \la X \la 0
\]
with $A_X, A^X \in \A$ and $B_X, B^X \in \B$.

The cotorsion pair is called \emph{functorially complete} if, moreover, the approximation sequences can be chosen to be functorial in $X$.
\end{defn}

\begin{expl} \label{expl:cotorsion-inj}
In every Grothendieck category $\G$ there is always a trivial complete cotorsion pair. Namely, denote by $\Inj\G$ the full subcategory of all injective objects. Then $(\G,\Inj\G)$ is a complete cotorsion pair, where for given $X \in \G$ we have the following approximation sequences:
\[
0 \la X \la EX \la EX/X \la 0 \; \textrm{ and } \;
0 \la 0 \la X \la X \la 0.
\]
We denote by $X \to EX$ an injective envelope of $X$ (cf.\ Theorem~\ref{thm:enough-inj}). In fact, the trivial cotorsion pair is functorially complete, and as we shall see in Corollary~\ref{cor:enough-inj}, we have a similar result for exact categories of Grothendieck type.
\end{expl}

\subsection{Generators for the Ext functors}
\label{subsec:generators-ext}

Next we shall develop a technical result necessary for the construction of complete cotorsion pairs using the small object argument from Section~\ref{sec:wfs}. A reader familiar with~\cite{H1} will quickly notice the connection to the concept of a small cotorsion pair from~\cite[Definition 6.4]{H1}. As an easy consequence we will also show that given an efficient exact category (Definition~\ref{def:efficient-exact}), $\Ext^n_\E(X,Y)$ is a set for every choice of $X,Y \in \E$ and $n \ge 1$. Thus, no set-theoretic problems are to arise in our setting.

\begin{prop} \label{prop:homolog-set}
Let $\E$ be an efficient exact category with a generator $G$, and let $S \in \E$ be an object. Then there exists a set $\I_S$ of inflations of $\E$ with the following properties:
\begin{enumerate}
\item Each $f \in \I_S$ fits into a conflation
\[ 0 \la K \overset{f}\la G^{(I)} \la S \la 0. \]

\item If $h\dd X \to Y$ is an inflation in $\E$ with cokernel isomorphic to $S$, then $h$ is a pushout of some $f \in \I_S$. In other words, for each such $h$ there exists $f \in \I_S$ so that we have a commutative diagram with conflations in rows:
\[
\begin{CD}
0 @>>> K @>{f}>> G^{(I)} @>>>    S @>>> 0\phantom{.}  \\
@.   @VVV         @VVV           @|                   \\
0 @>>> X @>{h}>>    Y    @>{p}>> S @>>> 0.
\end{CD}
\]
\end{enumerate}
\end{prop}

\begin{proof}
We shall give a direct construction of $\I_S$ and then prove that it has the properties we require. Consider a set $I \subseteq \E(G,S)$ and let 
\[ p_I\dd G^{(I)} \la S, \]
be the canonical morphism---the one defined by $p_I\circ j_i = i$ for all $i \in I$, where $j_i\dd G \to G^{(I)}$ is the $i$-th coproduct injection. Let $\D_S$ be the collection of all $I \subseteq \E(G,S)$ \st $p_I$ is a deflation in $\E$. As $G$ is a generator, we have $\E(G,S) \in \D_S$---for module and Grothendieck categories this is clear while for efficient exact categories this is shown in the proof of~\cite[Proposition 2.7]{SaSt}.

We define $\I_S$ as the collection of the kernels $k_I$ of the morphisms $p_I$ for all $I \in \D_S$. That is, given $I \in \D_S$, we take a conflation
\[ 0 \la K_I \overset{k_I}\la G^{(I)} \overset{p_I}\la S \la 0. \]
and include $k_I$ in $\I_S$.

As (1) is obviously satisfied for such an $\I_S$, we shall focus on (2). In fact here we give a proof only under the additional assumption that $\E$ is a Grothendieck category with the abelian exact structure. The fully general argument can be found in~\cite[Proposition 2.7]{SaSt} and it is similar, but more technical and not so enlightening.

Thus, consider a morphism $h\dd X \to Y$ which fits into a short exact sequence
\[ 0 \la X \overset{h}\la Y \overset{p}\la S \la 0. \]
The object $G$ being a generator, there is an epimorphism $G^{(J)} \overset{g}\to Y$. Consider the composition $pg\dd G^{(J)} \to S$, which is necessarily an epimorphism too. Denoting by $\ell_i\dd G \to G^{(J)}$ the coproduct inclusions, we may get the same compositions $pg\ell_i\dd G \to S$ for distinct elements $i \in J$. We therefore define an equivalence relation on $J$ by putting
\[ i \sim i' \textrm{ for } i,i' \in J \qquad \textrm{ if } \qquad pg\ell_i = pg\ell_{i'}. \]

Let now $J' \subseteq J$ be any set of representatives for the equivalence classes \wrt $\sim$, and define $g'\dd G^{(J')} \to Y$ as the restriction of $g$ to $G^{(J')}$. We claim that the composition
\[ pg'\dd G^{(J')} \la S \]
is still an epimorphism. Indeed, as the morphisms $pg\ell_i\dd G \to S$ for equivalent $i \in J$ contribute equally to the image of $pg$, we have
\[ \Img pg' = \sum_{i \in J'} \Img pg\ell_i = \sum_{i \in J} \Img pg\ell_i = \Img pg = S. \]
This allows us to construct the diagram with exact rows as in (2), whose left hand side commutative square is necessarily bicartesian. We conclude the proof by noting that when identifying $J'$ with a subset $I \subseteq \E(G,S)$ via the injective mapping $J' \to \E(G,S)$ given by $i \mapsto pg'\ell_i$, the kernel map $f$ of $pg'$ belongs (up to isomorphism) to the set $\I_S$ constructed above.
\end{proof}

It is not so difficult to extend the statement for $n$-fold extensions.

\begin{cor} \label{cor:generators-ext}
Let $\E$ and $G$ be as in Proposition~\ref{prop:homolog-set}. Then any exact sequence 
\[
\begin{CD}
0 @>>> Y @>>> E_n @>>> \cdots @>>> E_1 @>>> X @>>> 0.
\end{CD}
\]
in $\E$ admits a commutative diagram with exact rows
\[
\begin{CD}
0 @>>> K @>>> G^{(I_n)} @>>> \cdots @>>> G^{(I_1)} @>>> X @>>> 0\phantom{,}  \\
@.    @VVV      @VVV            @.         @VVV         @|                   \\
0 @>>> Y @>>>   E_n     @>>> \cdots @>>>    E_1    @>>> X @>>> 0,
\end{CD}
\]
\st the cardinalities of $I_1, \dots, I_n$ are bounded by a cardinal $\kappa = \kappa(X,G)$, which only depends on $X,G \in \E$.
\end{cor}

\begin{proof}
We prove the statement by induction in $n$. The case $n=1$ is clear from Proposition~\ref{prop:homolog-set}. If $n>1$, we use the inductive hypothesis to obtain a diagram
\[
\begin{CD}
0 @>>>    K'   @>>> G^{(I_{n-1})} @>>> \cdots @>>> G^{(I_1)} @>>> X @>>> 0\phantom{.}  \\
@.     @V{t}VV          @VVV              @.         @VVV         @|                   \\
0 @>>> Z_{n-1} @>>>    E_{n-1}    @>>> \cdots @>>>    E_1    @>>> X @>>> 0.
\end{CD}
\]
We also have the commutative diagram with conflations in rows
\[
\begin{CD}
0 @>>> K @>>> G^{(I_n)} @>>>    K'   @>>> 0\phantom{,}  \\
@.   @VVV       @VVV            @|                      \\
0 @>>> Y @>>>    E'     @>>>    K'   @>>> 0\phantom{,}  \\
@.     @|       @VVV         @VV{t}V                    \\
0 @>>> Y @>>>   E_n     @>>> Z_{n-1} @>>> 0,
\end{CD}
\]
where the middle row is the pullback of the bottom row along $t$ and the upper row is obtained using Proposition~\ref{prop:homolog-set}. The diagram from the statement arises by ignoring the middle row in the last diagram and splicing it with the second last diagram.
\end{proof}

As an easy corollary, we get the promised rectification of the potential set-theoretic issue.

\begin{cor} \label{cor:small-ext}
Let $\E$ be an efficient exact category. Then $\Ext^n_\E(X,Y)$ is a set for each $X,Y \in \E$ and $n \ge 1$.
\end{cor}

\begin{proof}
Let $\E$ and $G \in \E$ be as in Proposition~\ref{prop:homolog-set}, and fix $X \in \E$ and $n \ge 1$. Let $\kappa = \kappa(X,G)$ be a cardinal as in Corollary~\ref{cor:small-ext} and consider an exact sequence
\[
\ep\dd \quad 0 \la K_\ep \la G^{(I_n)} \la \cdots \la G^{(I_1)} \la X \la 0
\]
\st $\card{I_1}, \dots, \card{I_n} \le \kappa$. Given $Y \in \E$, there is an obvious map $\E(K_\ep,Y) \to \Ext^n_\E(X,Y)$ which acts by taking the pushouts of $\ep$ along maps in $\E(K_\ep,Y)$. The assignment is functorial in $Y$, so that we have a natural transformation $\E(K_\ep,-) \to \Ext^n_\E(X,-)$. If we sum these transformations over all possible sequences $\ep$, we get a transformation
\[ \coprod_\ep \E(K_\ep,-) \la \Ext^n_\E(X,-), \]
which is surjective for every $Y \in \E$ by Corollary~\ref{cor:generators-ext}.
\end{proof}

\subsection{From orthogonality on morphisms to vanishing of Ext}
\label{subsec:lifting-to-Ext}

We start to work out the connection between complete cotorsion pairs and \wfss and prove that exact categories of Grothendieck type have enough injectives. We shall start with two statements, which are taken from~\cite{SaSt} in the generality in which we state them.

\begin{lem} \label{lem:homolog-set} \cite[2.5]{SaSt}
Let $\E$ be an efficient exact category, $G \in \E$ a generator and $S \in \E$ any fixed object. Consider the set $\I_S$ from Proposition~\ref{prop:homolog-set}. Then the following are equivalent for $Y \in \E$:
\begin{enumerate}
\item $\Ext^1_\E(S,Y) = 0$.
\item $f \lifts (Y \to 0)$ for every $f \in \I_S$ (cf.\ Definition~\ref{def:lifting}).
\end{enumerate}
\end{lem}

\begin{proof}
(1) $\implies$ (2) is easy and left to the reader. For (2) $\implies$ (1), we will assume that $Y \to 0$ has the right lifting property for all $f \in \I_S$. We must prove that any fixed conflation
\[ \ep\dd \quad 0 \la Y \la E \la S \la 0 \]
with $S \in \clS$ splits. By the choice of $\I_S$, we know that we have the solid part of the commutative diagram
\[
\xymatrix{
0 \ar[r] & K \ar[r]^f \ar[d] & G^{(J)} \ar[r] \ar[r] \ar@{.>}[dl] \ar[d] & S \ar[r] \ar@{=}[d] & 0  \\
0 \ar[r] & Y \ar[r] & E \ar[r] & S \ar[r] & 0
}
\]
with conflations in rows and \st $f \in \I_S$. Since $f \lifts (Y \to 0)$, we can fill in the dotted arrow so that the upper triangle commutes. This precisely means that $\ep$ splits.
\end{proof}

The second result is, in a way, an analogue of Lemma~\ref{lem:wfs-closure}(2) for cotorsion pairs. For modules, the corresponding result is called the Eklof lemma, see~\cite[Lemma 3.1.2]{GT}. We will present a rather different proof, however, which is taken from~\cite[Proposition 2.12]{SaSt}.

\begin{prop} \label{prop:eklof} \cite[2.12]{SaSt}
Let $\E$ be an exact category satisfying (Ef1) of Definition~\ref{def:efficient-exact}. Let $\B$ be a class of objects in $\E$ and denote $\A = {^\perp\B}$ (see Definition~\ref{def:cotorsion}). Then any $\A$-filtered object belongs to $\A$.
\end{prop}

\begin{proof}
Consider an $\A$-filtered object $X$ and $Y \in \B$. We must prove that any fixed extension
\[ \ep\dd \quad 0 \la Y \overset{j}\la E \overset{p}\la X \la 0 \]
splits. Let us also fix an $\A$-filtration (see Definition~\ref{def:filtr})
\[
\xymatrix@1{
0 = X_0 \ar[r]^{f_{01}} &
X_1 \ar[r]^{f_{12}} &
X_2 \ar[r]^{f_{23}} &
X_3 \ar[r] &
\cdots \ar[r] &
X_\omega \ar[r]^{f_{\omega,\omega+1}} &
X_{\omega+1} \ar[r] &
\cdots
}
\]
for $X$ and let $\lambda$ be the ordinal by which this filtration is indexed. To facilitate the notation, we put $X_\lambda = X$.

Next we shall construct a $\lambda$-sequence of inflations
\[
\xymatrix@1{
Y = E_0 \ar[r]^{j_{01}} &
E_1 \ar[r]^{j_{12}} &
E_2 \ar[r]^{j_{23}} &
E_3 \ar[r] &
\cdots \ar[r] &
E_\omega \ar[r]^{j_{\omega,\omega+1}} &
E_{\omega+1} \ar[r] &
\cdots
}
\]
whose transfinite composition is $j\dd Y \to E = E_\lambda$ and \st $\Coker f_{\alpha,\alpha+1} = \Coker j_{\alpha,\alpha+1}$ for each $\alpha+1<\lambda$. If $\E$ is a Grothendieck category with the abelian exact structure, we just take the preimages $E_ \alpha = p\inv(X_\alpha)$. In the general case, the construction is slightly more technical and we refer to~\cite[Lemma 2.10]{SaSt}.

Finally we shall inductively construct a collection of morphisms $g_\alpha\dd E_\alpha \to Y$ such that $g_0 = 1_Y$ and for each $\alpha<\beta\le\lambda$, the triangle
\[
\xymatrix{
E_\alpha \ar[r]^{j_{\alpha\beta}} \ar[d]_{g_\alpha} & E_\beta \ar[dl]^(.37){g_\beta}  \\
Y
}
\]
commutes. For $\alpha = 0$ and $\beta = \lambda$, this precisely says that $\ep$ splits since then $j_{0\lambda} = j$.

Regarding the induction, we put $g_0 = 1_Y$ as required and at limit steps, we construct $g_\alpha$ as the colimit map of $(g_\gamma)_{\gamma<\alpha}$. For ordinal successors $\alpha+1\le\lambda$, suppose we have constructed $g_\alpha\dd E_\alpha \to Y$. We need to construct a factorization
\[
\xymatrix{
0 \ar[r] & E_\alpha \ar[r]^{j_{\alpha,\alpha+1}} \ar[d]_{g_\alpha} & E_{\alpha+1} \ar[r] \ar@{.>}[dl]^(.37){g_{\alpha+1}} & A_\alpha \ar[r] & 0,  \\
& Y
}
\]
but this exists since $\Ext^1_\E(A_\alpha,Y) = 0$.
\end{proof}

Now we are in a position to give a criterion for the existence of one of the approximation sequences required for the completeness of a cotorsion pair in Definition~\ref{def:cotorsion}.

\begin{prop} \label{prop:cotorsion-semicomplete} \cite[2.13(4)]{SaSt}
Let $\E$ be an efficient exact category, let $\clS$ be a set (not a proper class!) of objects, and put $\B = \clS^\perp$. Then there exist for every $X \in \E$ a short exact sequence
\[ 0 \la X \la B_X \la A_X \la 0 \]
\st $B_X \in \B$ and $A_X \in \Filt\clS$ ($\subseteq {^\perp\B}$ by Proposition~\ref{prop:eklof}). The exact sequences can be chosen to be functorial in $X$.
\end{prop}

\begin{proof}
Choose a generator $G$ for $\E$ and denote for each $S \in \clS$ by $\I_S$ a set of inflations as in Proposition~\ref{prop:homolog-set}. Let $\I = \bigcup_{S \in \clS} \I_S$ and put
\[ \R = \{ g \mid f \lifts g \textrm{ for all } f \in \I \} \qquad \textrm{ and } \qquad \clL = \{ f \mid f \lifts g \textrm{ for all } g \in \R \}. \]
Then $(\clL,\R)$ is a functorial \wfs by Corollary~\ref{cor:small-obj-exact}. If $X \in \E$ is an arbitrary object, we have a functorial factorization
\[
\xymatrix{
X \ar[rr] \ar[dr]_{f_X} && 0  \\
& B_X \ar[ur]_{g_X}
}
\]
with $f_X \in \Icell$ (see Definition~\ref{def:I-cell}) and $g_X \in \R$. Now $B_X \in \B$ by Lemma~\ref{lem:homolog-set} and $f_X$ is an inflation with $\clS$-filtered cokernel by (Ef1) of Definition~\ref{def:efficient-exact}. Hence we have functorial short exact sequences $ 0 \la X \overset{f_X}\la B_X \la A_X \la 0 $ as required.
\end{proof}

Now we can prove another aforementioned result---the existence of enough injectives for exact categories of Grothendieck type. Recall that $I \in \E$ is called \emph{injective} if $\Ext^1_\E(-,I) \equiv 0$. Denote the class of all injective objects by $\Inj\E$.

\begin{cor} \label{cor:enough-inj}
Let $\E$ be an exact category of Grothendieck type (Definition~\ref{def:Grothendieck-exact}). Then $(\E,\Inj\E)$ is a functorially complete cotorsion pair.
\end{cor}

\begin{proof}
Fix a set $\clS \subseteq \E$ \st $\E = \Filt\clS$---we can do that thanks to (GT4) of Definition~\ref{def:Grothendieck-exact}. Then $\clS^\perp = \Inj\E$ by Proposition~\ref{prop:eklof} and, by Proposition~\ref{prop:cotorsion-semicomplete}, there are exact sequences
\[ 0 \la X \la E(X) \la \Cosyz{}X \la 0 \]
functorial in $X$ \st $E(X) \in \Inj\E$. The approximation sequences of the other type (Definition~\ref{def:cotorsion}) are trivial, see Example~\ref{expl:cotorsion-inj}.
\end{proof}

\begin{rem} \label{rem:enough-inj}
We cannot really speak of an injective cogenerator in general since exact categories of Grothendieck type may not have products. For instance, $\E = \FlatR$ from Example~\ref{expl:flat} has products \iff $R$ is left coherent by~\cite[Theorem 2.1]{CB}. This deficiency, however, does not pose any problems here.
\end{rem}

\subsection{Weak factorization systems versus cotorsion pairs}
\label{subsec:wfs-to-cotorsion}

Now we are going to relate cotorsion pairs to \wfss compatible with the exact structure. As a consequence, we prove a vast generalization of the criterion for completeness of cotorsion pairs from~\cite{ET01}. The ideas are mostly taken from~\cite[\S\S4--5]{H1} and~\cite{Ro}. A similar presentation has been later given, specifically for categories of complexes of modules, in~\cite[Chapter 6]{EJv2}.

Let us make precise in which way we want our \wfs $(\clL,\R)$ to be compatible with the exact structure. Note first that $f \in \clL$ is an inflation, then $0 \to \Coker f$ belongs to $\clL$ by Lemma~\ref{lem:wfs-closure}(1). Dually, if $f \in \R$ is a deflation, then $\Ker f \to 0$ is in $\R$. Using the terminology in accordance with~\cite{G5} we define:

\begin{defn} \label{def:exact-wfs}
Let $\E$ be an arbitrary exact category and $(\clL,\R)$ be a weak factorization system.
We call $(\clL,\R)$ an \emph{exact \wfs}if the following two conditions are satisfied for a morphism $f$ in $\E$:
\begin{enumerate}
\item[(FS3)] $f \in \clL$ \iff $f$ is an inflation and $0 \to \Coker f$ belongs to $\clL$.
\item[(FS4)] $f \in \R$ \iff $f$ is an deflation and $\Ker f \to 0$ belongs to $\R$.
\end{enumerate}
\emph{Exact functorial \wfss}are defined analogously.
\end{defn}

Thus, an exact \wfs is determined by a pair of classes of objects---the cokernels of morphisms in $\clL$ and the kernels of morphisms in $\R$. Before stating the main result of the section, we shall introduce notations for this correspondence.

\begin{nota} \label{nota:wfs-to-cotorsion}
Given a class $\clL$ of inflations, we put $\Coker\clL = \{ A \mid A \cong \Coker f \textrm{ for some } f \in \clL \}$. Dually for a class $\R$ of deflations, $\Ker\R = \{ B \mid B \cong \Ker f \textrm{ for some } f \in \R \}$.

Conversely, given classes of objects $\A,\B \subseteq \E$, we denote by $\Infl\A$ the class of all inflations with cokernel in $\A$, and by $\Defl\B$ the class of all deflations with kernel in $\B$.
\end{nota}

\begin{thm} \label{thm:wfs-to-cotorsion}
Let $\E$ be an exact category. Then
\[ (\clL,\R) \longmapsto (\Coker\clL,\Ker\R) \qquad \textrm{ and } \qquad (\A,\B) \mapsto (\Infl\A,\Defl\B) \]
define mutually inverse bijective mappings between exact \wfss $(\clL,\R)$ and complete cotorsion pairs $(\A,\B)$. The bijections restrict to mutually inverse mappings between exact functorial \wfss $(\clL,\R)$ and functorially complete cotorsion pairs $(\A,\B)$.
\end{thm}

Before proving the theorem, we shall establish two auxiliary lemmas, which should be rather self-explanatory.

\begin{lem} \label{lem:zero-ext-to-lifting}
Let $\E$ be an exact category, $f$ be an inflation and $g$ be a deflation in $\E$. If $\Ext^1_\E(\Coker f,\Ker g) = 0$, then $f \lifts g$.
\end{lem}

\begin{proof}
Suppose that we have a commutative square formed by the solid arrows
\[
\xymatrix{
A \ar[r]^u \ar[d]_f     & X\phantom{.} \ar[d]^g      \\
B \ar[r]^v \ar@{.>}[ur]|h & Y.
}
\]
We must prove the existence of the dotted morphism $h\dd B \to X$. If we denote $C = \Coker f$ and $K = \Ker g$, we obtain the following commutative diagram of abelian groups with exact rows:
\[
\xymatrix{
0 \ar[r] & \E(C,X) \ar[r] \ar[d] & \E(B,X) \ar[d]_{\E(B,g)} \ar[r]^{\E(f,X)} & \E(A,X) \ar[r]^(.45)\dif \ar[d]^{\E(A,g)} & \Ext^1_\E(C,X) \ar[r] \ar[d]^{\Ext^1_\E(C,g)} & \cdots \\
0 \ar[r] & \E(C,Y) \ar[r] & \E(B,Y) \ar[r]^{\E(f,Y)} & \E(A,Y) \ar[r]^(.45)\dif & \Ext^1_\E(C,Y) \ar[r] & \cdots \\
}
\]
If we view the diagram as a map of complexes, the mapping cone
\[
\cdots \to \E(B,X) \oplus \E(C,Y) \to \E(A,X) \oplus \E(B,Y) \overset{\dif'}\to \Ext^1_\E(C,X) \oplus \E(A,Y) \to \cdots
\]
must be exact as well by~\cite[Proposition II.4.3]{McL}. Given the construction of the cone, we only need to prove that $\dif'(-u,v) = 0$. This amounts to showing that
\[ -\dif(-u) = 0 \qquad \textrm{and} \qquad \E(f,Y)(v) + \E(A,g)(-u) = 0. \]
The second equality holds since $gu = vf$, while for the first one we use the equalities
\[ \Ext^1_\E(C,g)\big(\dif(-u)\big) = \dif\big(\E(A,g)(-u)\big) = \dif\big(\E(f,Y)(-v)\big) = 0 \]
and the fact that $\Ext^1_\E(C,g)$ is a monomorphism because of
\[ 0 = \Ext^1_\E(C,K) \xrightarrow{\phantom{\Ext^1_\E(C,g)}} \Ext^1_\E(C,X) \xrightarrow{\Ext^1_\E(C,g)} \Ext^1_\E(C,Y). \qedhere \]
\end{proof}

\begin{lem} \label{lem:lifting-to-zero-ext}
Let $\E$ be an exact category. If $\A \subseteq \E$ is a class of objects and $g$ is a deflation \st $f \lifts g$ for each inflation $f$ with $\Coker f \in \A$, then $\Ker g \in \A^\perp$. Dually if $\B \subseteq \E$ and $f$ is an inflation \st $f \lifts g$ for each deflation $g$ with $\Ker g \in \B$, then $\Coker f \in {^\perp\B}$.
\end{lem}

\begin{proof}
Suppose that $g\dd X \to Y$ is a deflation \st $f \lifts g$ for each inflation $f$ with $\Coker f \in \A$. Denote $K = \Ker g$ and consider a conflation
\[ \ep\dd \quad 0 \la K \overset{i}\la E \la A \la 0 \]
with $A \in \A$. Then in particular $i \lifts g$ and by the dual of Lemma~\ref{lem:wfs-closure}(1) we also have $i \lifts (K \to 0)$. Applying the latter fact to the commutative square
\[
\xymatrix{
K \ar[r]^{1_K} \ar[d]_i & K \ar[d]      \\
E \ar[r]   \ar@{.>}[ur] & 0
}
\]
we see that $\ep$ splits. Since $\ep$ was chosen arbitrarily, we have $K \in \A^\perp$. The other part of the lemma is dual.
\end{proof}

Now we can finish the proof of  the theorem.

\begin{proof}[Proof of Theorem~\ref{thm:wfs-to-cotorsion}]
Suppose first that $(\clL,\R)$ is an exact \wfs and put $\A = \Coker\clL$ and $\B = \Ker\R$. Both $\A$ and $\B$ are closed under retracts since so are $\clL$ and $\R$, and $\Ext^1_\E(A,B) = 0$ for each $A \in \A$ and $B \in \B$ by Lemma~\ref{lem:lifting-to-zero-ext}. In order to prove the existence of approximation sequences from Definition~\ref{def:cotorsion}, consider $X \in \E$ and the following two factorizations
\[
\xymatrix{
X \ar[rr] \ar[dr]_{i_X} && 0  \\
& B_X \ar[ur]
}
\qquad \textrm{and} \qquad
\xymatrix{
0 \ar[rr] \ar[dr] && X  \\
& A^X \ar[ur]_{p^X}
}
\]
\wrt $(\clL,\R)$. It follows directly from Definition~\ref{def:exact-wfs} that there are conflations
\[
0 \la X \overset{i_X}\la B_X \la A_X \la 0 \quad \textrm{ and } \quad
0 \la B^X \la A^X \overset{p^X}\la X \la 0
\]
with $A_X, A^X \in \A$ and $B_X, B^X \in \B$. Finally, if $X \in {^\perp\B}$ then the second approximation sequence splits and so $X \in \A$. Similarly $\B = \A^\perp$.

Let conversely $(\A,\B)$ be a complete cotorsion pair in $\E$, and put $\clL = \Infl\A$ and $\R = \Defl\B$. Then $\clL$ and $\R$ are closed under retracts since so are $\A$ and $\B$. If $f \in \clL$ and $g \in \R$, then $f \lifts g$ by Lemma~\ref{lem:zero-ext-to-lifting}.

Next we shall prove that every morphism $h\dd X \to Y$ factorizes as $h = gf$ with $f \in \clL$ and $g \in \R$. Suppose first that $h$ is an inflation and consider the following pullback diagram with an approximation sequence for $C = \Coker f$ in the rightmost column. 
\[
\begin{CD}
  @.     @.       0  @.    0           \\
@.     @.       @VVV     @VVV          \\
  @.     @.      B^C @=   B^C          \\
@.     @.       @VVV     @VVV          \\
0 @>>> X @>{f}>>  Z  @>>> A^C @>>> 0   \\
@.     @|      @V{g}VV   @VVV          \\
0 @>>> X @>{h}>>  Y  @>>>  C  @>>> 0   \\
@.     @.       @VVV     @VVV          \\
  @.     @.       0  @.    0           \\
\end{CD}
\eqno{(\ddag)} \label{eqn:approx}
\]
The corresponding factorization of $h$ then appears in the leftmost square. A dual argument applies if $h$ is a deflation. If $h$ is an arbitrary morphism, we can functorially factorize it as
\[
\begin{CD}
X @>{\mathrm{inc}}>> X \oplus Y @>{(h, 1_Y)}>> Y.
\end{CD}
\]
As $X \to X \oplus Y$ is a split inflation, we can factor it as $\textrm{inc} = g_1f_1$ with $f_1 \in \clL$ and $g_1 \in \R$. Now the composition $(h, 1_Y) \circ g_1$ is a deflation, so that we can factor it as $g_2f_2$ with $f_2 \in \clL$ and $g_2 \in \R$:
\[
\begin{CD}
X @>{f_1}>>             Z_1     @>{f_2}>>     Z_2      \\
@|                   @VV{g_1}V             @VV{g_2}V   \\
X @>{\mathrm{inc}}>> X \oplus Y @>{(h, 1_Y)}>> Y       \\
\end{CD}
\]
It follows that $h = g_2(f_2f_1)$ and that $f_2f_1 \in \clL = \Infl\A$ by~\cite[Lemma 3.5]{Bu} since $\A$ is closed under extensions.

Conditions (FS3) and (FS4) from Definition~\ref{def:exact-wfs} are clearly satisfied by the very definition of $(\clL,\R)$. Thus, $(\clL,\R)$ is an exact weak factorization system.

Finally, the passage from factorization of morphisms to approximation sequences and back is clearly functorial. This proves the last claim.
\end{proof}

As a consequence, we can improve Proposition~\ref{prop:cotorsion-semicomplete} and recover an existence result for complete cotorsion pairs from~\cite[\S2]{SaSt}, which vastly generalizes the main result of~\cite{ET01}. The argument itself can be traced back to~\cite[Theorem 6.5]{H1} and~\cite{Ro}. To underline the significance of this type of result, we note that~\cite{ET01} was one of the starting points of the monograph~\cite{GT}, which discusses various results and techniques to study infinitely generated modules. Taking into account our Section~\ref{sec:Qco}, many of the techniques carry over to the categories of quasi-coherent sheaves directly.

\begin{thm} \label{thm:cotorsion-complete}
Let $\E$ be an efficient category and $\clS$ be a set (not a proper class!) of objects \st $\Filt\clS$ contains a generator for $\E$. Put
\[ \B = \clS^\perp \qquad \textrm{ and } \qquad \A = {^\perp\B}. \]
Then $(\A,\B)$ is a functorially complete cotorsion pair in $\E$ (Definition~\ref{def:cotorsion}) and $\A$ consists precisely of retracts of $\clS$-filtered objects. 
\end{thm}

\begin{proof}
As in the proof of Proposition~\ref{prop:cotorsion-semicomplete}, we choose a generator $G$ for $\E$ and denote for each $S \in \clS$ by $\I_S$ a set of inflations given by Proposition~\ref{prop:homolog-set}. We further put $\I = \bigcup_{S \in \clS} \I_S$ and
\[ \R = \{ g \mid f \lifts g \textrm{ for all } f \in \I \} \qquad \textrm{ and } \qquad \clL = \{ f \mid f \lifts g \textrm{ for all } g \in \R \}, \]
obtaining a functorial \wfs $(\clL,\R)$.

Denote by $\F$ the closure of $\Filt\clS$ (see Definition~\ref{def:filtr}) under retracts. Thanks to the construction of $\I$, a morphisms $f'$ is an inflation with a cokernel from $\clS$ \iff it is a pushout of some $f \in \I$. It follows from the description of $\clL$ in Corollary~\ref{cor:small-obj-exact} that $\clL = \Infl\F$.

We claim that $\R$ consists of deflations with kernels in $\B$. Let $g\dd X \to Y$ be in $\R$ and consider a deflation $p\dd F \to Y$ with $F \in \F$---such a $p$ must exist since $\F$ contains a generator. Then the dotted arrow in the square
\[
\xymatrix{
0 \ar[r] \ar[d]         & X \ar[d]^g      \\
F \ar[r]^p \ar@{.>}[ur] & Y
}
\]
can be filled in in such a way that both triangles commute. Since $\E$ is weakly idempotent complete, $g$ is a deflation by Lemma~\ref{lem:id-compl}(2). Then $\Ker g \in \A^\perp = \B$ by Lemmas~\ref{lem:lifting-to-zero-ext} and~\ref{lem:wfs-closure}(1), finishing the proof of the claim.

In fact, we even have $\R = \Defl\B$ by Lemma~\ref{lem:zero-ext-to-lifting}. Hence $(\F,\B)$ is a functorially complete cotorsion pair by Theorem~\ref{thm:wfs-to-cotorsion} and necessarily $\F = \A$, which concludes the proof.
\end{proof}

As a consequence we obtain a convenient criterion for recognizing the left hand side class of a functorially complete cotorsion pair.

\begin{cor} \label{cor:lhs-cotorsion}
Let $\E$ be an efficient exact category and $\F \subseteq \C$ be a class of objects which is deconstructible, closed under retracts and contains a generator. Then $(\F,\F^\perp)$ is a functorially complete cotorsion pair in $\E$.
\end{cor}

\begin{proof}
If $\clS$ is a set of objects \st $\F = \Filt\clS$, then $\F = {^\perp (\clS^\perp)}$ by Theorem~\ref{thm:cotorsion-complete}.
\end{proof}

A partial converse holds for Grothendieck categories, or more generally for exact categories of Grothendieck type which arise as in Theorem~\ref{thm:deconstr-Groth}.

\begin{cor} \label{cor:lhs-cot-converse}
Let $\E$ be a deconstructible subcategory closed under retracts in a Grothendieck category, with the induced exact structure (hence $\E$ is of Grothendieck type). If $(\A,\B)$ is a (functorially) complete cotorsion pair \st $\B = \clS^\perp$ for a set $\clS$, then $\A$ is deconstructible, closed under retracts and contains a generator.
\end{cor}

\begin{proof}
The only problem consists in proving that $\A$ is deconstructible, as we know that $\A$ is the collection of all \emph{retracts of} objects in $\Filt\clS$. But this follows from Proposition~\ref{prop:deconstr-advanced}(1) and Lemma~\ref{lem:deconstr-transitive}. For the sake of completeness, we note that the proof of Proposition~\ref{prop:deconstr-advanced} in~\cite{St2} uses the Hill Lemma (Proposition~\ref{prop:hill-lemma}).
\end{proof}

\section{Exact and hereditary model categories}
\label{sec:model}

In this section we shall give an account on Quillen model categories. The motivation for model categories is the following situation. We start with a category $\C$ and a class of morphisms $\we$, and we would like to understand the category $\C[\we\inv]$ where the morphisms in $\we$ are made artificially invertible. It is not difficult to construct such $\C[\we\inv]$ formally (see~\cite{GZ67}), up to a set-theoretic difficulty. Namely, the collections of morphisms $\C[\we\inv](X,Y)$ may be proper classes rather than sets---see~\cite{CN09} for an example of this pathology. If $(\C, \we)$ admits some extra structure making it a model category, the latter difficulty disappears, but there are more advantages as we shall see.

After recalling the classical properties of model categories, we aim at showing how model categories can be obtained using the tools from previous sections. The idea here is due to Hovey~\cite{H1}.

\subsection{The homotopy category of a model category}
\label{subsec:model-htp}

First, we shall briefly recall the general theory, as at this point this can be done rather quickly and efficiently. Standard sources for more information about the topic are~\cite{Hir,H2}.

\begin{defn} \label{def:model}
Let $\C$ be a category. A \emph{model structure} on $\C$ is a triple $(\cof,\we,\fib)$ of classes of morphisms, called \emph{cofibrations, weak equivalences} and \emph{fibrations}, respectively, \st
\begin{enumerate}
\item[(MS1)] $\we$ is closed under retracts and satisfies the 2-out-of-3 property for composition. That is, if $f,g$ is a pair of composable morphisms in $\C$ and two of $f,g,gf$ are in $\we$, so is the third.
\item[(MS2)] $(\cof,\we\cap\fib)$ and $(\cof\cap\we,\fib)$ are \wfss in $\C$.
\end{enumerate}
Morphisms in $\cof\cap\we$ are called \emph{trivial cofibrations} and morphisms in $\we\cap\fib$ are \emph{trivial fibrations}.

A \emph{model category} is a category with a model structure \st $\C$ has an initial object $\emptyset$, a terminal object $*$, all pushouts of trivial cofibrations along trivial cofibrations exists, and dually all pullbacks of trivial fibrations along trivial fibrations exist.
\end{defn}

\begin{rem} \label{rem:model}
The definition of a model structure on $\C$, although seemingly different, is equivalent to~\cite[Definition 1.1.3]{H2} up to one detail. We have dropped the functoriality of the \wfss from (MS2). This subtle change is inessential for the theory discussed here and we have mostly done the change to be able to use Proposition~\ref{prop:lifting-to-cpx}. On the other hand, almost all \wfss available in practice seem to be generated by a set of morphisms from the left hand side as in Theorem~\ref{thm:small-obj} (such model structures are usually called cofibrantly generated, \cite[11.1.2]{Hir}), so they functorial.

Our definition of a model category differs slightly from~\cite[Definition 7.1.3]{Hir} and~\cite[Definition 1.1.4]{H2} in that we do not require existence of all limits and colimits in $\C$, but only $\emptyset$, $*$, and certain pushouts and pullbacks. The reason for this digression is that exact categories of Grothendieck type are often not complete and cocomplete (see Example~\ref{expl:flat}), but our assumptions are enough for the fundamental Theorem~\ref{thm:htp-of-model} below. Moreover, if $\C$ is an exact category and the two \wfss in (MS2) are exact, then the existence of the pushouts and pullbacks which we need is implicit in the definition of an exact category, so that we even need no additional assumptions on (co)completeness of~$\C$!
\end{rem}

If we have a model category, the existence of the initial and terminal object allows us to define cofibrant and fibrant objects of $\C$, as well as cofibrant and fibrant ``approximations'' of any object.

\begin{defn} \label{def:replacements}
An object $B \in \C$ is called \emph{cofibrant} (\emph{trivially cofibrant}) if $\emptyset \to B$ is a cofibration (trivial cofibration, resp.) Dually $X \in \C$ is \emph{fibrant} (\emph{trivially fibrant}) if $X \to *$ is a fibration (trivial fibration, resp.)

If $X \in \C$ is arbitrary and $\emptyset \to CX \to X$ is a factorization of $\emptyset \to X$ \wrt $(\cof,\we\cap\fib)$, then $CX$ is called a \emph{cofibrant replacement} of $X$. Dually, the object $FX$ in a factorization $X \to FX \to *$ \wrt $(\cof\cap\we,\fib)$ is called a \emph{fibrant replacement} of $X$.
\end{defn}

It is further useful to notice that a model structure is overdetermined in that two of $\cof$, $\we$, $\fib$ determine the third. Before discussing the homotopy category of a model category, we need to briefly recall the homotopy relations.

\begin{defn} \label{def:model-htp} \cite[1.2.4]{H2}
Let $\C$ be a model category and $f,g\dd B \to X$ be two morphisms in $\C$.
\begin{enumerate}
\item A \emph{cylinder object} for $B$ is a factorization of the codiagonal map $\nabla\dd B \amalg B \to B$ into a cofibration $(i_0, i_1)$ followed by a weak equivalence $s$:
\[
\begin{CD}
B \amalg B @>{(i_0, i_1)}>> B' @>{s}>> B.
\end{CD}
\]

\item A \emph{path object} for $X$ is dually a factorization of the diagonal map $\Delta\dd X \to X \times X$ into a weak equivalence followed by a fibration:
\[
\begin{CD}
X @>{r}>> X' @>{\left(\begin{smallmatrix} p_0 \\ p_1 \end{smallmatrix}\right)}>> X \times X.
\end{CD}
\]

\item A \emph{left homotopy} from $f$ to $g$ is a map $H\dd B' \to X$ for some cylinder object $B'$ for $B$ \st such that $Hi_0 = f$ and $Hi_1 = g$. We say that $f$ and $g$ are \emph{left homotopic}, written $f \lhtp g$, if there is a left homotopy from $f$ to $g$. 

\item A \emph{right homotopy} from $f$ to $g$ is a map $K\dd B \to X'$ for some path object $X'$ for $X$ \st such that $p_0K = f$ and $p_1K = g$. We say that $f$ and $g$ are \emph{right homotopic}, $f \rhtp g$, if there is a right homotopy from $f$ to $g$. 

\item Morphisms $f$ and $g$ are said to be \emph{homotopic}, written $f \sim g$, if they are both left and right homotopic.
\end{enumerate}
\end{defn}

Let us summarize the relevant properties of the homotopy relations.

\begin{prop} \label{prop:model-htp}
Let $\C$ be a model category and suppose that $A,B \in \C$ are cofibrant and $X,Y \in \C$ are fibrant. Then the following hold:
\begin{enumerate}
\item Morphisms $f,g\dd B \to X$ are left homotopic \iff they are right homotopic \iff they are homotopic.
\item The homotopy is an equivalence relation on $\C(B,X)$.
\item If $h\dd A \to B$ is a morphism and $f \sim g$, then $fh \sim gh$.
\item If $h'\dd X \to Y$ is a morphism and $f \sim g$, then $h'f \sim h'g$.
\end{enumerate}
\end{prop}

\begin{proof}
The proof of~\cite[Proposition 1.2.5]{H2} applies. Here we need the existence of the pullbacks and pushouts from Definition~\ref{def:model} for the transitivity of the homotopy relation.
\end{proof}

After introducing a piece of traditional terminology, we shall state the key theorem about the localization $\C[\we\inv]$ of $\C$.

\begin{defn} \label{def:htp-catg}
Let $\C$ be a model category and $\we$ the class of weak equivalences. Then the category $\C[\we\inv]$ is called the \emph{homotopy category} of $\C$ and denoted by $\Ho\C$.
\end{defn}

\begin{thm} \label{thm:htp-of-model}
Let $\C$ together with $(\cof,\we,\fib)$ be a model category, let $Q\dd \C \to \Ho\C$ be the canonical localization functor, and denote by $\C_{cf}$ the full subcategory given by the objects which are cofibrant and fibrant.

\begin{enumerate}
\item The composition $\C_{cf} \overset{\subseteq}\la \C \overset{Q}\la \Ho\C$ induces a category equivalence $\mbox{$(\C_{cf}/\sim)$} \to \Ho\C$, where $\C_{cf}/\sim$ is defined by $(\C_{cf}/\sim)(X,Y) = \C_{cf}(X,Y)/\sim$. In particular, $\Ho\C$ is a legal category with small homomorphisms spaces.

\item If $X \in \C$ is cofibrant and $Y \in \C$ is fibrant, then $Q$ induces an isomorphism $\C(X,Y)/\sim\, \overset{\cong}\la \Ho\C(QX,QY)$. In particular, there are canonical isomorphisms $\C(CX,FY)/\sim\, \overset{\cong}\la \Ho\C(QX,QY)$ for arbitrary $X,Y \in \C$, where $CX$ is a cofibrant replacement of $X$ and $FY$ is a fibrant replacement of $Y$.

\item If $f\dd X \to Y$ is a morphism in $\C$, them $Qf$ is an isomorphisms \iff $f$ is a weak equivalence.
\end{enumerate}
\end{thm}

\begin{proof}
These are standard facts dating back to~\cite{QHtp}. We refer to~\cite[Theorem 1.2.10]{H2} or~\cite[\S8.4]{Hir} for a proof.
\end{proof}

\subsection{Exact model structures}
\label{subsec:exact-models}

It is no surprise that following~\cite{H1}, we shall focus on exact categories with model structure compatible with the exact structure.

\begin{defn} \label{def:exact-models}
Let $\E$ be an exact category. An \emph{exact model structure} on $\E$ is a model structure $(\cof,\we,\fib)$ \st $(\cof,\we\cap\fib)$ and $(\cof\cap\we,\fib)$ are exact \wfss in the sense of Definition~\ref{def:exact-wfs}. An \emph{exact model category} is a weakly idempotent complete exact category $\E$ together with an exact model structure.
\end{defn}

Exact model structures are clearly determined by the classes of cofibrant and fibrant objects $\Cof$ and $\Fib$, respectively. Following the same philosophy, we shall consider the class $\We$ of trivial objects instead of the class $\we$ of weak equivalences. Here, and object $X \in \E$ is \emph{trivial} if $0 \to X$ is a weak equivalence.

The next theorem is a result due to Hovey~\cite{H1} which describes what conditions the classes $\Cof$, $\We$, $\Fib$ must satisfy in order to determine an exact model structure on $\E$. The result was stated in this generality in~\cite{G5}, but we prefer to give a full proof here.

\begin{thm} \label{thm:exact-models}
Let $\E$ be a weakly idempotent complete exact category. There is a bijective correspondence between exact model structures on $\E$ and the triples of classes $(\Cof,\We,\Fib)$ satisfying the following conditions
\begin{enumerate}
\item $\We$ is closed under retracts and satisfies the 2-out-of-3 property for extensions. That is, if $0 \to X \to Y \to Z \to 0$ is a conflation in $\E$ and two of $X,Y,Z$ are in $\We$, then so is the third.
\item $(\Cof,\We\cap\Fib)$ and $(\Cof\cap\We,\Fib)$ are complete cotorsion pairs in $\E$.
\end{enumerate}
The correspondence assigns to $(\Cof,\We,\Fib)$ the model structure $(\cof,\we,\fib)$ \st
\begin{enumerate}
\item[(a)] $c \in \cof$ \iff $c$ is an inflation with cokernel in $\Cof$,
\item[(b)] $w \in \we$ \iff $w = w_dw_i$ for an inflation $w_i$ with cokernel in $\We$ and a deflation $w_d$ with kernel in $\We$,
\item[(c)] $f \in \fib$ \iff $f$ is a deflation with kernel in $\Fib$.
\end{enumerate}
The inverse assigns to $(\cof,\we,\fib)$ the triple $(\Cof,\We,\Fib)$ \st $\Cof$ is the class of cofibrant objects, $\We$ is the class of trivial objects, and $\Fib$ is the class of fibrant objects.
\end{thm}

\begin{rem} \label{rem:cofibrantly-generated}
If $\E$ is an efficient exact category (Definition~\ref{def:efficient-exact}) and the two cotorsion pairs $(\Cof,\We\cap\Fib)$ and $(\Cof\cap\We,\Fib)$ are both of the form $\big({^\perp (\clS^\perp)}, \clS^\perp\big)$ for a set $\clS$ as in Theorem~\ref{thm:cotorsion-complete}, then the resulting exact model structure is cofibrantly generated in the sense of~\cite[Definition 11.1.2]{Hir}. This is a very usual situation thanks to results in~\cite{GT,St2}. In such a case, approximations \wrt the two cotorsion pairs and the corresponding \wfss from Definition~\ref{def:model} are all functorial.
\end{rem}

Once we prove the theorem, there is an immediate corollary which poses further restrictions on the possible classes of trivial objects.

\begin{cor} \label{cor:exact-model-trivial}
Let $\E$ be an exact model category with the exact structure given by $(\Cof,\We,\Fib)$. Then the class $\We$ of trivial objects is generating and cogenerating in~$\E$.
\end{cor}

Before proving the theorem, we need to establish three lemmas. The first of them is syntactically very similar to a necessary condition for the calculus of non-commutative fractions, see~\cite[p. 12, cond. 2.2.c)]{GZ67}.

\begin{lem} \label{lem:fraction-swap} \cite[5.6]{H1}
Let $\E$ be a weakly idempotent complete exact category and $(\Cof,\We,\Fib)$ be a triple satisfying (1) and (2) from the statement of Theorem~\ref{thm:exact-models}. Denote $\tcof = \Infl(\Cof\cap\We)$ and $\tfib = \Defl(\We\cap\Fib)$, and suppose we have a pair of composable morphisms $p \in \tfib$ and $i \in \tcof$. Then there is commutative square
\[
\begin{CD}
  X   @>{j}>>   U     \\
@V{p}VV      @VV{q}V  \\ 
  Y   @>{i}>>   Z     \\
\end{CD} 
\]
with $j \in \tcof$ and $q \in \tfib$.
\end{lem}

\begin{proof}
If $\cof = \Infl\C$ as in (a) in the theorem, then $(\cof,\tfib)$ is an exact \wfs by Theorem~\ref{thm:wfs-to-cotorsion}. Consider a factorization of $ip = qj$ of $ip$ \wrt $(\cof,\tfib)$. We shall prove that $q,j$ have the desired properties. Using Lemma~\ref{lem:id-compl}, \cite[Proposition 3.1]{Bu} and the exact $3\times 3$ lemma~\cite[Corollary 3.6]{Bu}, we obtain a commutative diagram with conflations in rows and columns
\[
\begin{CD}
   @.      0     @.         0    @.       0              \\
@.        @VVV            @VVV          @VVV         @.  \\ 
0  @>>>  \Ker p  @>>>    \Ker q  @>>>     V    @>>>  0   \\
@.        @VVV            @VVV          @VVV         @.  \\ 
0  @>>>    X     @>{j}>>    U    @>>> \Coker j @>>>  0   \\
@.       @V{p}VV         @VV{q}V        @VVV         @.  \\ 
0  @>>>    Y     @>{i}>>    Z    @>>> \Coker i @>>>  0   \\
@.        @VVV            @VVV          @VVV         @.  \\ 
   @.      0     @.         0    @.       0              \\
\end{CD}
\]

Our task reduces to proving that $\Coker j \in \We$. To this end, $V \in \We$ since $\Ker p, \Ker q \in \We$, and $\Coker j \in \We$ since $V,\Coker i \in \We$.
\end{proof}

The second one will be useful for proving the 2-out-of-3 property for $\we$ when we start from $(\Cof,\We,\Fib)$.

\begin{lem} \label{lem:cancel}
Let $\E$ be a weakly idempotent complete exact category, $(\Cof,\We,\Fib)$ be a triple satisfying (1) and (2) from Theorem~\ref{thm:exact-models}, and let $\cof, \tcof, \fib, \tfib$ be as above. Let further $f\dd X \to Y$ and $g\dd Y \to Z$ be a pair of composable morphisms.
\begin{enumerate}
\item If $gf \in \fib$, $f \in \cof$ and two of the following conditions hold, then so does the third:
  \begin{enumerate}
    \item $f \in \tcof$,
    \item $g$ is a deflation and $\Coker g \in \We$,
    \item $gf \in \tfib$.
  \end{enumerate}

\item If $gf \in \cof$ and $g \in \fib$ and two of the following conditions hold, then so does the third:
  \begin{enumerate}
    \item $f$ is an inflation and $\Ker f \in \We$,
    \item $g \in \tfib$,
    \item $gf \in \tcof$.
  \end{enumerate}
\end{enumerate}
\end{lem}

\begin{proof}
We will prove only (1) as (2) is dual. Our assumption says that $gf$ is a deflation, and so is $g$ by Lemma~\ref{lem:id-compl}. Using Lemma~\ref{lem:id-compl} again together with~\cite[Proposition 2.12]{Bu}, we construct a commutative diagram with conflations in rows and columns:
\[
\begin{CD}
   @.       0     @.       0                             \\
@.        @VVV           @VVV                            \\
0  @>>>  \Ker gf  @>>>     X    @>{gf}>>  Z  @>>>  0     \\
@.        @VVV          @V{f}VV           @|             \\
0  @>>>  \Ker g   @>>>     Y    @>{g}>>   Z  @>>>  0     \\
@.        @VVV           @VVV                            \\
   @.   \Coker f  @=   \Coker f                          \\
@.        @VVV           @VVV                            \\
   @.       0     @.       0                             \\
\end{CD}
\]
The conclusion follows from the 2-out-of-3 property of $\We$ applied to the leftmost vertical conflation.
\end{proof}

The last lemma characterizes weak equivalences among inflations and deflations when starting with an exact model structure $(\cof,\we,\fib)$.

\begin{lem} \label{lem:2-out-of-3}
Let $\E$ together with $(\cof,\we,\fib)$ be an exact model category. If $f$ is an inflation, then $f \in \we$ \iff $\Coker f$ is a trivial object. Dually if $g$ is a deflation, then $g \in \we$ \iff $\Ker g$ is a trivial object.
\end{lem}

\begin{proof}
We will only prove the statement for inflations, the other case being dual. Let $f$ be an inflation and factor $f$ as $f = pi$ \wrt $(\cof,\we\cap\fib)$. Using Lemma~\ref{lem:id-compl}, we get a commutative diagram
\[
\begin{CD}
0  @>>>     X   @>{i}>>    Z    @>>>     C    @>>>  0     \\
@.         @|           @VV{p}V       @VV{q}V             \\
0  @>>>     X   @>{f}>>    Y    @>>> \Coker f @>>>  0     \\
\end{CD}
\]
Note that $q$ is a trivial cofibration since the right hand side square is a pushout square by~\cite[Proposition 2.12]{Bu}.
Now $f$ is a weak equivalence \iff $i$ is a trivial cofibration \iff $C$ is trivially cofibrant. Inspecting the composition $C \overset{q}\to \Coker f \to 0$, the latter is also equivalent to $\Coker f \to 0$ being a weak equivalence, or in other words, to the fact that $\Coker f$ is a trivial object.
\end{proof}

Now we can finish the proof of Theorem~\ref{thm:exact-models}.

\begin{proof}[Proof of Theorem~\ref{thm:exact-models}]
Using the notation $\tcof = \Infl(\Cof\cap\We)$ and $\tfib = \Defl(\We\cap\Fib)$, it is clear from Theorem~\ref{thm:wfs-to-cotorsion} that $(\cof,\tfib)$ is an exact \wfs \iff $(\Cof,\We\cap\Fib)$ is a complete cotorsion pair, and similarly for $(\tcof,\fib)$ and $(\Cof\cap\We,\Fib)$.

Suppose now that $(\Cof,\We,\Fib)$ is a triple satisfying (1) and (2). Note that the assumptions make sure that $(\Cof\cap\We,\We\cap\Fib)$ is a complete cotorsion pair in $\We$ when the latter is viewed as an exact category with the exact structure induced from~$\E$ (see Proposition~\ref{prop:exact-categories}).

Let $\we$ be the class of morphisms as in (b). The same argument as in the proof of Theorem~\ref{thm:wfs-to-cotorsion} shows that all morphisms $w \in \we$ factorize as $w = pi$ with $i \in \tcof$ and $p \in \tfib$. Indeed, if $w$ is an inflation with cokernel in $\We$ or a deflation with kernel in $\We$, we construct a diagram similar to~$(\ddag)$ in page~\pageref{eqn:approx} with approximation sequences \wrt $(\Cof\cap\We,\We\cap\Fib)$. If $w = w_dw_i$, we construct a commutative diagram
\[
\begin{CD}
X @>{i_1}>>     Z_1    @>{i_2}>>  Z_2      \\
@|           @VV{p_1}V        @VV{p_2}V    \\
X @>{w_i}>>      Z     @>{w_d}>>   Y       \\
\end{CD}
\]
with $i_1,i_2 \in \tcof$ and $p_1,p_2$ is $\tfib$, and we use the fact that $\Cof\cap\We$ is closed under extensions.

An immediate consequence of the latter description of $\we$ and Lemma~\ref{lem:cancel} is that $\tcof = \cof\cap\we$ and $\tfib=\we\cap\fib$. This yields the two \wfss required by Definition~\ref{def:model}(MS2).

It also easily follows that $\we$ is closed under compositions. Indeed, if $i_1,i_2 \in \tcof$ and $p_1,p_2 \in \tfib$ \st the composition $c = (p_2i_2)(p_1i_1)$ exists, we can write $c = (p_2q)(ji_1)$ for $j \in \tcof$ and $q \in \tfib$ by Lemma~\ref{lem:fraction-swap}.

Next we claim that $\we$ has the 2-out-of-3 property for composition. We shall prove that if $w = pi \in \we$ with $i \in \tcof$ and $p \in \tfib$, and if $wf = pif \in \we$ for some morphism $f$, then $f \in \we$. Since we can factor $f$ into a morphism from $\cof$ followed by a morphism from $\tfib$, we can without loss of generality assume that $f \in \cof$. As $wf \in \we$, we can write $p(if) = qj$ for some $j \in \tcof$ and $q \in \tfib$. If we form the pushout of $if \in \cof$ and $j \in \tcof$, we get a commutative diagram
\[
\xymatrix{
X \ar[r]^{if} \ar[d]_{j} & U \ar[d]_k \ar@/^/[ddr]^p \\
Y \ar[r]^\ell \ar@/_/[drr]_q & V \ar@{.>}[dr]|(.4)\hole|(.4)s  \\
&& Z
}
\]
As clearly $k \in \tcof$ and $p = sk \in \tfib$, the morphism $s$ is a deflation with $\Ker s \in \We$ by Lemma~\ref{lem:cancel}. As $q = s\ell$, we can apply Lemma~\ref{lem:cancel} again and we show that $\ell \in \tcof$. Then clearly $if \in \tcof$ as the square in the above diagram is a pushout square. Finally, the conflation
\[ 0 \la \Coker f \la \Coker if \la \Coker i \la 0 \]
for the composition $if$ obtained using~\cite[Lemma 3.5]{Bu} shows that $\Coker f \in \We$. Hence $f \in \tcof \subseteq \we$. The remaining case is dual and we have proved the claim.

Finally, let $w = w_1 \oplus w_2\dd X_1 \oplus X_2 \to Y_1 \oplus Y_2$ and suppose $w \in \we$. Then $w_n$ factor as $w_n = p_n i_n$, $i_n \in \cof$, $p_n \in \tfib$, for $n = 1,2$. Comparing this factorization with the one for $w = pi$ with $i \in \tcof$, $p \in \tfib$ for $w$, and using the 2-out-of-3 property for $\we$ proved above, we see that $i_n \in \tcof$ for $n=1,2$. Thus $\we$ us closed under retracts, finishing the proof that $(\cof,\we,\fib)$ is an exact model structure on $\E$.

Suppose conversely that $(\cof,\we,\fib)$ is an exact model structure, and denote by $\Cof$, $\We$, $\Fib$ the classes of cofibrant, trivial, and fibrant objects, respectively. Then $(\Cof,\We\cap\Fib)$ and $(\Cof\cap\We,\Fib)$ are functorially complete cotorsion pairs in $\E$ by Theorem~\ref{thm:wfs-to-cotorsion}. As $\we$ is closed under retractions, so is $\We$. In order to prove the 2-out-of-3 property for extensions for $\We$, let
\[ 0 \la W_1 \la W_2 \la W_3 \la 0 \]
be a conflation. Then if two of $0 \to W_1$, $W_1 \to W_2$ and $0 \to W_2$ belong to $\we$, so does the third. If we apply Lemma~\ref{lem:2-out-of-3}, we deduce that if two of $W_1, W_2, W_3$ are trivial objects, so is the third. Hence $(\Cof,\We,\Fib)$ satisfies Theorem~\ref{thm:exact-models}(1) and~(2), as required.
\end{proof}

\subsection{Hereditary model structures and triangles}
\label{subsec:hered-models}

An expected but important point is that the homotopy category of an exact model category often carries a natural algebraic triangulated structure. To this end, recall that a \emph{Frobenius exact category} is an exact category $\E'$ \st
\begin{enumerate}
\item $\E'$ has enough projective and injective objects, and
\item the classes of projective and injective objects coincide.
\end{enumerate}
If $\E'$ is such a category, then the quotient $\underline{\E'}$ of $\E'$ modulo the two-sided ideal of all morphisms which factor through projective injective objects is naturally a triangulated category. We refer to~\cite[Chapter I]{Ha} for details. Triangulated categories which arise in this way are called \emph{algebraic}~\cite[\S7]{Kr}.

In order to describe the prospective triangles, we adapt the definition of cofiber sequences from~\cite[Chapter 6]{H2}, incorporating the ideas from~\cite{Ha}.

\begin{defn} \label{def:co-fiber}
Let $\E$ together with $(\Cof,\We,\Fib)$ be an exact model category. A sequence of morphisms
\[
\begin{CD}
X @>{u}>> Y @>{v}>> Z @>{w}>> \Sigma X
\end{CD}
\]
is called a \emph{cofiber sequence} if $u,v,w$ fit into the following diagram
\[
\begin{CD}
0  @>>>    X  @>>>   W   @>>>    \Sigma X\phantom{,} @>>>  0   \\
@.       @V{u}VV   @VVV             @|                         \\
0  @>>>    Y @>{v}>> Z   @>{w}>> \Sigma X,           @>>>  0   \\
\end{CD}
\]
where both rows are conflations in $\E$ and $W \in \We$.
\end{defn}

In order to understand the homotopy category of an exact model category, we must first better understand the homotopy relations. We summarize the essentials in the following lemma.

\begin{lem} \label{lem:exact-homotopies}
Let $\E$ be an exact model category and $f,g\dd X \to Y$ be two morphisms.
\begin{enumerate}
\item $f$ and $g$ are right homotopic \iff $f-g$ factors through a trivially cofibrant object.
\item $f$ and $g$ are left homotopic \iff $f-g$ factors through a trivially fibrant object.
\item If $X$ is cofibrant and $Y$ is fibrant, then $f$ and $g$ are homotopic \iff $f-g$ factors through a trivially fibrant and cofibrant object.
\end{enumerate}
\end{lem}

\begin{proof}
We refer to~\cite[Proposition 4.4]{G5} or~\cite[Proposition 1.1.14]{Beck12}.
\end{proof}

As indicated above, we need another condition in order to construct an algebraic triangulated structure---we need our model structure to be hereditary. The significance of this notion in connection to exact model structures has been very well demonstrated in the work of Gillespie~\cite{G3,G2,G1,G4,G5}. The property we are interested in is described by the following lemma.

\begin{lem} \label{lem:hered-cot}
Let $\E$ be a weakly idempotent complete exact category and $(\A,\B)$ be a cotorsion pair in $\E$ \st $\A$ is generating and $\B$ is cogenerating (e.g.\ if $(\A,\B)$ is complete). Then the following are equivalent:
\begin{enumerate}
\item $\Ext^n_\E(A,B) = 0$ for each $A \in \A$, $B \in \B$ and $n \ge 1$.
\item $\Ext^2_\E(A,B) = 0$ for each $A \in \A$, $B \in \B$.
\item If $f\dd A_1 \to A_2$ is a deflation with $A_1,A_2 \in \A$, then $\Ker f \in \A$.
\item If $g\dd B_1 \to B_2$ is an inflation with $B_1,B_2 \in \B$, then $\Coker g \in \B$.
\end{enumerate}
\end{lem}

\begin{proof}
Essentially the same proof as in~\cite[Lemma 4.25]{SaSt} applies and the straightforward modifications for our setting can be made using results in~\cite[\S7]{Bu}. Alternatively, a different proof for $(\A,\B)$ complete was given in~\cite[Proposition 1.1.12]{Beck12}.

The crucial implications are (3)~$\implies$~(1) and (4)~$\implies$~(1). The other being dual, it suffices to prove the first one. The key point is that using an argument very similar to~\cite[Lemma I.4.6(i)]{Hart} we construct for any $n$-fold extension $\ep\dd 0 \to B \to E_n \to \cdots \to E_1 \to A \to 0$ a commutative diagram
\[
\begin{CD}
\tilde\ep\dd \quad 0 @>>> B @>{\ell}>>      A_n    @>>> \cdots @>>>     A_1     @>>> A @>>> 0  \\
                   @.    @|                @VVV            @.          @VVV         @|         \\
\ep\dd       \quad 0 @>>> B @>>>            E_n    @>>> \cdots @>>>     E_1     @>>> A @>>> 0  \\
\end{CD}
\]
with $A_i \in \A$ for $i = 1, \dots, n-1$ and deflations in all columns. Since $\tilde\ep$ clearly represents the same element of $\Ext^n_\E(A,B)$ as $\ep$ and the kernel of $A_i \to A_{i-1}$ belongs to $\A$ by (3) for every $1 \le i \le n-1$ (here $A_0 = A$ by convention), we deduce that $\ell\dd B \to A_n$ splits and $\tilde\ep$ represents the zero element.
\end{proof}

\begin{defn} \label{def:hered-cot}
Let $\E$ be a weakly idempotent complete exact category. A~\emph{hereditary cotorsion pair} is a cotorsion pair $(\A,\B)$ in $\E$ satisfying the equivalent conditions of Lemma~\ref{lem:hered-cot} (including that $\A$ is generating and $\B$ cogenerating).
\end{defn}

\begin{defn} \label{def:hered-model}
Let $\E$ be a weakly idempotent complete exact category and let $(\Cof,\We,\Fib)$ be a triple specifying an exact model structure as in Theorem~\ref{thm:exact-models}. The model structure is called \emph{hereditary} if the cotorsion pairs $(\Cof,\We\cap\Fib)$ and $(\Cof\cap\We,\Fib)$ are hereditary.

A \emph{hereditary model category} is an exact model category with a hereditary model structure.
\end{defn}

One of the key points about hereditary complete cotorsion pairs is that the following version of the Horseshoe Lemma holds:

\begin{lem} \label{lem:horseshoe}
Let $\E$ be a weakly idempotent complete exact category and $(\A,\B)$ be a complete hereditary cotorsion pair. Suppose we have the solid part of the commutative diagram
\[
\xymatrix{
              & 0          \ar[d]       & 0          \ar@{.>}[d]       & 0          \ar[d]             \\
0 \ar[r]      & X   \ar[r] \ar[d]_{i_1} & Y   \ar[r] \ar@{.>}[d]_{i_2} & Z   \ar[r] \ar[d]_{i_3} & 0   \\
0 \ar@{.>}[r] & B_X \ar@{.>}[r] \ar[d]  & B_Y \ar@{.>}[r] \ar@{.>}[d]  & B_Z \ar@{.>}[r] \ar[d]  & 0   \\
0 \ar@{.>}[r] & A_X \ar@{.>}[r] \ar[d]  & A_Y \ar@{.>}[r] \ar@{.>}[d]  & A_Z \ar@{.>}[r] \ar[d]  & 0   \\
              & 0                       & 0                            & 0 
}
\]
with a conflation in the upper row and approximation sequences in the two columns. Then the dotted part can be completed so that all rows are conflations and all columns are approximation sequences.
\end{lem}

\begin{proof}
A proof can be found for instance in~\cite[Lemma 5.2.3]{GT} and other references in~\cite[Lemma 3.3]{YL11}. We construct the pushout
\[
\xymatrix{
0 \ar[r] & X   \ar[r] \ar[d]_{i_1} & Y   \ar[r] \ar[d]_{i'}  & Z   \ar[r] \ar@{=}[d] & 0   \\
0 \ar[r] & B_X \ar[r]              & P   \ar[r]              & Z   \ar[r]            & 0   \\
}
\]
Since $\Ext^2_\E(A_Z,B_X) = 0$, we can use the long exact sequence of $\Ext$'s and construct the commutative diagram
\[
\xymatrix{
0 \ar[r] & B_X   \ar[r] \ar@{=}[d] & P   \ar[r] \ar[d]_{i''} & Z     \ar[r] \ar[d]_{i_3} & 0   \\
0 \ar[r] & B_X   \ar[r]            & B_Y \ar[r]              & B_Z   \ar[r]              & 0   \\
}
\]
Finally we simply put $i_2 = i''i'$.
\end{proof}

Now we are in a position to state and prove the main result of the section.

\begin{thm} \label{thm:alg-tria}
Let $\E$ be a hereditary model category, the exact model structure given by $(\Cof,\We,\Fib)$, and denote $\E_{cf} = \Cof\cap\Fib$ and $\omega = \E_{cf} \cap \We$.
\begin{enumerate}
\item $\E_{cf}$ with the induced exact structure is a Frobenius exact category and $\omega$ is precisely the class of projective injective objects in $\E_{cf}$. In particular, $\Ho\E$ is an algebraic triangulated category.
\item Every cofiber sequence in $\E$ becomes a triangle in $\Ho\E$, and conversely every triangle is isomorphic in $\Ho\E$ to a cofiber sequence. In particular, every conflation $0 \to X \overset{u}\to Y \overset{v}\to Z \to 0$ yields a triangle in $\Ho\E$:
\[ X \overset{u}\la Y \overset{v}\la Z \overset{w}\la \Sigma X. \]
\end{enumerate}
\end{thm}

\begin{proof}
Part (1) is an easy consequence of Lemmas~\ref{lem:exact-homotopies} and~\ref{lem:hered-cot} and it is proved in~\cite[Proposition 5.2]{G5}; see also~\cite[Proposition 1.1.15]{Beck12}.

We focus on (2) and suppose that we have a cofiber sequence given by the diagram
\[
\begin{CD}
0  @>>>    X  @>>>   W   @>>>    \Sigma X\phantom{,} @>>>  0   \\
@.       @V{u}VV   @VVV             @|                         \\
0  @>>>    Y @>{v}>> Z   @>{w}>> \Sigma X,           @>>>  0   \\
\end{CD}
\eqno{(*)}
\]
with $W \in \We$. If we choose approximation sequences of $X,Y$ and $\Sigma X$ \wrt the cotorsion pair $(\Cof\cap\We,\Fib)$, we can use Lemma~\ref{lem:horseshoe} to construct the following commutative diagram
\[
\xymatrix@C=1.5em{
& 0 \ar[rr] && X \ar[rr] \ar[dl]_u \ar[dd]|\hole && W \ar[rr] \ar[dl] \ar[dd]|\hole && \Sigma X \ar[rr] \ar@{=}[dl] \ar[dd]|\hole && 0  \\
0 \ar[rr] && Y \ar[rr]^(.6)v \ar[dd] && Z \ar[rr]^(.6)w \ar[dd] && \Sigma X \ar[rr] \ar[dd] && 0  \\
& 0 \ar[rr]|(.47)\hole && F_X \ar[rr]|\hole \ar[dl]_{u'} && F_W \ar[rr]|(.48)\hole \ar[dl] && \Sigma F_X \ar[rr] \ar@{=}[dl] && 0  \\
0 \ar[rr] && F_Y \ar[rr]^(.6){v'} && F_Z \ar[rr]^(.6){w'} && \Sigma F_X \ar[rr] && 0,  \\
}
\]
where all rows are conflations, all objects at the bottom are fibrant, and all vertical morphisms are trivial cofibrations. In particular,
\[
\begin{CD}
F_X @>{u'}>> F_Y @>{v'}>> F_Z @>{w'}>> \Sigma F_X
\end{CD}
\]
is a cofiber sequence of fibrant objects and it is isomorphic in $\Ho\E$ to the original one. In a similar way, we can find a cofiber sequence in $\E_{cf}$ which is isomorphic to the original one in $\Ho\E$. But the latter one is necessarily a triangle by~\cite[I.2.5, p. 15--16]{Ha}.

The last statement is clear since $\We$ is generating, and thus every short exact sequence $0 \to Y \overset{v}\to Z \overset{w}\to \Sigma X \to 0$ is a part of a diagram $(*)$ which defines a cofiber sequence.
\end{proof}

\section{Models for the derived category}
\label{sec:cpxs}

In this section we will be concerned with the construction of models for the unbounded derived category of an exact category of Grothendieck type (Definition~\ref{def:Grothendieck-exact}). We will show that there always exists at least one model structure (the one induced by injectives) if the exact category $\E$ arises as in Theorem~\ref{thm:deconstr-Groth}. There is usually a lot of model structures for the derived category of a Grothendieck category. A recipe for their construction has been found by Gillespie~\cite{G3,G2,G1} and various versions of the result were discussed in~\cite{EGPT,SaSt}, but our presentation here is based on an elegant idea from~\cite{YL11}. Finally, we will discuss probably the most interesting example available so far which is essentially due to Neeman: the projective model structure on the derived category of flat modules.

The guiding principle for our constructions is the following consequence of Theorem~\ref{thm:exact-models} and Corollary~\ref{cor:exact-model-trivial}.

\begin{prop} \label{prop:model-constr}
Let $\E$ be a weakly idempotent complete exact category and $\We \subseteq \E$ be a generating and cogenerating class of objects which is closed under retracts in $\E$ and satisfies the 2-out-of-3 property for extensions.

If $(\A_0,\B_0)$ is a complete hereditary cotorsion pair in $\We$ \st there exist complete cotorsion pairs $(\Cof,\B_0)$ and $(\A_0,\Fib)$ in $\E$, then $\E$ together with $(\Cof,\We,\Fib)$ is a hereditary model category.
\end{prop}

\begin{proof}
This is a direct consequence of Theorem~\ref{thm:exact-models} and Lemma~\ref{lem:hered-cot}.
\end{proof}

\subsection{Lifting cotorsion pairs to categories of complexes}
\label{subsec:lifting-cpx}

Let us briefly discuss our setting and the tools to achieve the above goals. We start with an exact category $\E$ of Grothendieck type and consider the category $\Cpx\E$ of cochain complexes over $\E$ with the natural exact structure: A sequence $0 \to X \to Y \to Z \to 0$ of complexes is a conflation precisely when all the component sequences $0 \to X^n \to Y^n \to Z^n \to 0$, $n \in \Z$, are conflations.

As we are interested in the derived category $\Der\E$, our class of trivial objects will be $\We = \Cac\E$, the class of all acyclic complexes. Here, we call a complex
\[
X\dd \qquad \cdots \la X^0 \la X^1 \la X^2 \la X^3 \la \cdots
\]
\emph{acyclic} (or \emph{exact}) if it arises by splicing countably many conflations in $\E$ of the form
\[ 0 \la Z^i \la X^i \la Z^{i+1} \la 0. \]

In order to justify that this is a valid choice for $\We$, we need to prove that the necessary conditions from Theorem~\ref{thm:exact-models} and Corollary~\ref{cor:exact-model-trivial} are satisfied. To this end, we need a definition and a lemma.

\begin{defn} \label{def:admissible} \cite[8.1 and 8.8]{Bu}
Let $\E$ be an exact category. A morphism $f\dd X \to Y$ is called \emph{admissible} if it can be written as a composition $f = id$, where $i$ is an inflation and $d$ is a deflation. A sequence $X \overset{f}\to Y \overset{f'}\to Z$ of admissible morphisms $f = id$ and $f' = i'd'$ is \emph{exact} if $0 \to I \overset{i}\to Y \overset{d'}\to I' \to 0$ is a conflation.
\end{defn}

In other words, admissible morphisms are precisely those for which it makes sense to speak of an image. If $\E$ is abelian, every morphism is admissible.

\begin{lem} \label{lem:admissible-ext}
Let $\E$ be an exact category and suppose that we have a commutative diagram
\[
\begin{CD}
0   @>>>  X_1  @>{j_1}>>  Y_1  @>{p_1}>>  Z_1  @>>>   0   \\
@.      @VV{f_1}V       @VV{g_1}V       @VV{h_1}V         \\
0   @>>>  X_2  @>{j_2}>>  Y_2  @>{p_2}>>  Z_2  @>>>   0   \\
@.      @VV{f_2}V       @VV{g_2}V       @VV{h_2}V         \\
0   @>>>  X_3  @>{j_3}>>  Y_3  @>{p_3}>>  Z_3  @>>>   0   \\
\end{CD}
\]
with conflations in rows. If $f_2, h_1, h_2$ are admissible, $g_2g_1 = 0$ and $Z_1 \overset{h_1}\to Z_2 \overset{h_2}\to Z_3$ is exact, then $g_2$ is admissible and we have a commutative diagram.
\[
\begin{CD}
0   @>>>    X_3     @>{j_3}>>  Y_3     @>{p_3}>>  Z_3     @>>>   0   \\
@.         @VVV               @VVV               @VVV                \\
0   @>>> \Coker f_2 @>>>    \Coker g_2 @>>>    \Coker h_2 @>>>   0   \\
\end{CD}
\]
with conflations in rows and cokernel homomorphisms in columns.
\end{lem}

\begin{proof}
If we take the pushout of the lower conflation along $X_3 \to \Coker f_2$, we get a diagram of the form 
\[
\begin{CD}
\phantom{\ep\dd} \quad 0   @>>>    X_2     @>>>  Y_2  @>{p_2}>> Z_2  @>>>   0\phantom{.}   \\
                       @.        @VV{0}V      @VV{\ell}V     @VV{h_2}V                     \\
\ep\dd           \quad 0   @>>> \Coker f_2 @>>>   E  @>{p}>>    Z_3  @>>>   0.             \\
\end{CD}
\]
An easy diagram chase reveals that there exists $k\dd Z_2 \to E$ \st $\ell = kp_2$ and $h_2 = pk$. Since $0 = \ell g_1 = kp_2g_1 = kh_1p_1$ and $p_1$ is an epimorphism, we get $kh_1 = 0$. Hence $k$ factors through the image of $h_2$, and if we denote the inflation $\Img h_2 \to Z_3$ by $i$, we obtain a morphisms $k'\dd \Img h_2 \to Y_3$ \st $i = pk'$. Considering the equivalence class $[\ep] \in \Ext^1_\E(Z_3,\Coker f_2)$, the existence of $k'$ just says that
\[
\Ext^1_\E(i,\Coker f_2)([\ep]) = 0.
\]
Using the long exact sequence of Ext groups, this precisely means that there is a commutative diagram
\[
\begin{CD}
\ep\dd           \quad 0   @>>> \Coker f_2 @>>>  E   @>{p}>>    Z_3    @>>>   0\phantom{.}   \\
                       @.           @|         @VVV            @VVV                          \\
\phantom{\ep\dd} \quad 0   @>>> \Coker f_2 @>>>  E'  @>>>   \Coker h_2 @>>>   0.             \\
\end{CD}
\]

It follows from the construction that the composition $Y_3 \to E \to E'$ is a deflation and it is an easy application of~\cite[Corollaries 3.2 and 3.6]{Bu} that $g_2$ is admissible and the morphism $Y_3 \to E'$ is a cokernel of $g_2$.
\end{proof}

Now we can justify that the class $\We = \Cac\E$ is a well chosen class of trivial objects for an exact model structure. We also introduce some notation which will be useful later.

\begin{nota} \label{nota:discs-and-spheres}
Let $\E$ be an additive category, $X \in \E$ an object and $n$ an integer. We denote by $D^n(X)$ the complex
\[
D^n(X)\dd \qquad \cdots \la 0 \la X \overset{1_X}\la X \la 0 \la \cdots
\]
concentrated in cohomological degrees $n$ and $n+1$, and by $S^n(X)$ the complex
\[
S^n(X)\dd \qquad \cdots \la 0 \la X \la 0 \la 0 \la \cdots
\]
concentrated in cohomological degree $n$.
\end{nota}

\begin{prop} \label{prop:acyclics-ok}
Let $\E$ be a weakly idempotent complete exact category and consider $\Cpx\E$ with the exact structure induced from $\E$ (i.e.\ conflations in $\Cpx\E$ are defined as component wise conflations). Then $\We = \Cac\E$ is generating and cogenerating in $\Cpx\E$, it is closed under retracts and has the 2-out-of-3 property for extensions.
\end{prop}

\begin{proof}
The fact that $\We$ is generating and cogenerating in $\Cpx\E$ is easily seen by the adjunctions
\[ \E(X,Y^n) \cong \Cpx\E(D^n(X),Y) \qquad \textrm{and} \qquad \E(Y^n,X) \cong \Cpx\E(Y,D^{n-1}(X)) \]
for each $X \in \E$ and $Y \in \Cpx\E$. The 2-out-of-3 property for extensions follows by Lemma~\ref{lem:admissible-ext} and the $3\times3$ Lemma~\cite[Corollary 3.6]{Bu}. The closure of $\We$ under retracts is clear.
\end{proof}

In order to apply Proposition~\ref{prop:model-constr}, we first need to construct complete hereditary cotorsion pairs in $\We$. In that respect, we show that every complete hereditary cotorsion pair $(\A,\B)$ in $\E$ can be lifted to a complete hereditary cotorsion pair in $\We$. Note that there are also alternative ways to lift cotorsion pairs to categories of complexes, see e.g.~\cite{G4} or~\cite[Chapter 7]{EJv2}. Before stating the result, we again introduce some notation.

\begin{nota} \label{nota:a-tilde}
Let $\E$ be an exact category and $\A \subseteq \E$ be extension closed. Then $\A$ is an exact category by Proposition~\ref{prop:exact-categories} and $\Cac\A$ can be viewed as a full subcategory of $\Cpx\E$. In this context, we denote for brevity $\Cac\A$ by $\tilde\A$ (this notation has been used by Gillespie in~\cite{G3,G2,G1}).
\end{nota}

\begin{prop} \label{prop:lifting-to-cpx}
Let $\E$ be a weakly idempotent complete exact category and consider $\We = \Cac\E$ with the exact structure induced from $\E$. If $(\A,\B)$ is a complete hereditary cotorsion pair in $\E$, then $(\tilde\A,\tilde\B)$ is a complete hereditary cotorsion pair in $\We$.
\end{prop}

\begin{proof}
We first claim that $\Ext^1_\We(A,B) = 0$ for all $A \in \tilde\A$ and $B \in \tilde\B$. Using the trick from~\cite[\S7.3]{EJv2}, we write every $A \in \tilde\A$ as an extension
\[
\begin{CD}
@.           0     @.             0    @.             0    @.             0                 \\
  @.        @VVV                 @VVV                @VVV                @VVV               \\
\cdots @>>> Z^0(A) @>{0}>>      Z^1(A) @>{0}>>      Z^2(A) @>{0}>>      Z^3(A) @>>> \cdots  \\
  @.        @VVV                 @VVV                @VVV                @VVV               \\
\cdots @>>>  A^0   @>{\dif^0}>>   A^1  @>{\dif^1}>>   A^2  @>{\dif^2}>>   A^3  @>>> \cdots  \\
  @.        @VVV                 @VVV                @VVV                @VVV               \\
\cdots @>>> Z^1(A) @>{0}>>      Z^2(A) @>{0}>>      Z^3(A) @>{0}>>      Z^4(A) @>>> \cdots  \\
  @.        @VVV                 @VVV                @VVV                @VVV               \\
@.           0     @.             0    @.             0    @.             0                 \\
\end{CD}
\]
in $\Cpx\E$, where $Z^n(A) = \Ker\dif^n$ for $n \in \Z$. Hence it suffices to prove that $\Ext^1_{\Cpx\E}(S^n(A'),B) = 0$ for every $A' \in \A$ and $B \in \tilde\B$ (see Notation~\ref{nota:discs-and-spheres}). As any extension of $B$ by $S^n(A')$ is necessarily component wise split, it is enough to show that every cochain complex map $S^n(A') \to B$ is null-homotopic, but this is clear from $\Ext^1_\E(A',Z^{n-1}(B)) = 0$. This proves the claim.

Clearly $\tilde\A$ and $\tilde\B$ are closed under retracts since $\A$ and $\B$ are. Moreover, if $X \in \We$ and we fix approximation sequences $0 \to Z^n(X) \to B'_n \to A'_n \to 0$ and $0 \to B''_n \to A''_n \to Z^n(X) \to 0$ in $\E$ \wrt $(\A,\B)$ for each $n \in \Z$, then Lemma~\ref{lem:horseshoe} allows us to construct approximation sequences of $X$ in $\We$ \wrt $(\tilde\A,\tilde\B)$ (this trick appeared in~\cite{YL11}). It follows that $(\tilde\A,\tilde\B)$ is a complete cotorsion pair in $\We$.

Finally, the fact that $(\tilde\A,\tilde\B)$ is hereditary is a rather easy consequence of Lemma~\ref{lem:hered-cot}.
\end{proof}

\begin{rem} \label{rem:lifting-decostr}
For the sake of completeness we note that if $\E$ is deconstructible and closed under retracts in a Grothendieck category (hence exact of Grothendieck type by Theorem~\ref{thm:deconstr-Groth}), and if $\A$ is deconstructible in $\E$, then $\tilde\A$ is deconstructible in $\We = \Cac\E$. This follows from Lemma~\ref{lem:deconstr-transitive} and~\cite[Proposition 4.4]{St2}. In such a case, $(\tilde\A,\tilde\B)$ is even functorially complete in $\We$ by Theorem~\ref{thm:cotorsion-complete}.
\end{rem}

The final task in constructing a hereditary model structure using Proposition~\ref{prop:model-constr} consists in making good choices for the cotorsion pair $(\A,\B)$ in $\E$ so that $({^\perp\tilde\B},\tilde\B)$ and $(\tilde\A,\tilde\A^\perp)$ are complete cotorsion pairs in $\Cpx\E$. There does not seem to be any known method for proving that $({^\perp\tilde\B},\tilde\B)$ is a complete cotorsion pair at the level of generality we have worked in so far. It is, however, possible to show this in various still very general situations, and this is what we are going to discuss in the rest of this section.

\subsection{The injective model structure}
\label{subsec:injective-model}

First we are going to prove that for all known examples of exact categories $\E$ of Grothendieck type we have a so-called injective model structure for the unbounded derived category $\Der\E$. As indicated, we need an additional mild assumption on $\E$, namely that $\Cac\E$ is deconstructible in $\Cpx\E$. This is true for all exact categories of Grothendieck type which arise as in Theorem~\ref{thm:deconstr-Groth}.

\begin{lem} \label{lem:deconstr-acyclic}
Let $\E$ be a full subcategory of a Grothendieck category which is deconstructible and closed under retracts, considered with the induced exact structure (hence $\E$ is exact of Grothendieck type by Theorem~\ref{thm:deconstr-Groth}). Then $\Cac\E$ is deconstructible in $\Cpx\E$.
\end{lem}

\begin{proof}
This is a special case of~\cite[Proposition 4.4]{St2}, since $\Cac\E = \tilde\E$ in Notation~\ref{nota:a-tilde}.
\end{proof}

Under this additional assumption we can show that $(\tilde\A,\tilde\A^\perp)$ from the previous section is always a complete cotorsion pair. The lemma will be also useful later.

\begin{lem} \label{lem:lifting-completeness} \cite[3.5]{YL11}
Let $\E$ be an exact category of Grothendieck type \st $\Cac\E$ is deconstructible in $\Cpx\E$. If $(\A,\B)$ is a complete hereditary cotorsion pair in $\E$, then $(\tilde\A,\tilde\A^\perp)$ is a complete (and hereditary) cotorsion pair in $\Cpx\E$.
\end{lem}

\begin{proof}
The special case for $\A = \E$ follows from Theorem~\ref{thm:cotorsion-complete}. Suppose that $(\A,\B)$ is an arbitrary complete hereditary cotorsion pair in $\E$. Given a complex $X \in \Cpx\E$, we must construct approximation sequences \wrt $(\tilde\A,\tilde\A^\perp)$. To this end, consider an approximation sequence $0 \to X \to I \to E \to 0$ \wrt $(\tilde\E,\tilde\E^\perp)$, i.e.\ $E \in \tilde\E = \Cac\E$ and $I \in \tilde\E^\perp$. Using Proposition~\ref{prop:lifting-to-cpx} we obtain an approximation sequence $0 \to B^E \to A^E \to E \to 0$ in $\Cac\E$ \wrt $(\tilde\A,\tilde\B)$. Then we find an approximation sequence for $X$ in $\Cpx\E$ \wrt $(\tilde\A,\tilde\A^\perp)$ in the second row of the following pullback diagram
\[
\begin{CD}
  @.       @.   0   @.   0            \\
@.     @.     @VVV     @VVV           \\
  @.       @.   B^E @=   B^E          \\
@.     @.     @VVV     @VVV           \\
0 @>>> X   @>>> B^X @>>> A^E @>>> 0   \\
@.     @|     @VVV     @VVV           \\
0 @>>> X   @>>> I   @>>> E   @>>> 0   \\
@.     @.     @VVV     @VVV           \\
  @.       @.   0   @.   0            \\
\end{CD}
\]
The other approximation sequence for $X$ is obtained analogously by taking an approximation sequence $0 \to I' \overset{f}\to E' \to X \to 0$ with $E' \in \tilde\E$ and $I' \in \tilde\E^\perp$, and considering the pullback of an approximation sequence $0 \to B^{E'} \to A^{E'} \to E' \to 0$ along $f$ (see the proof of~\cite[Theorem 3.5]{YL11} for details).
\end{proof}

Now we can prove the existence of a model structure for $\Der\E$.

\begin{thm} \label{thm:inj-model-for-D(E)}
Let $\E$ be an exact category of Grothendieck type (Definition~\ref{def:Grothendieck-exact}) \st $\Cac\E$ is deconstructible in $\Cpx\E$. Then there is a hereditary model structure (Definition~\ref{def:hered-model}) on $\Cpx\E$ \st
\begin{enumerate}
\item Every object of $\Cpx\E$ is cofibrant.
\item The trivial objects are precisely the acyclic complexes, $\We = \Cac\E$.
\item The class of fibrant objects is $\F = \Cac\E^\perp$.
\end{enumerate}
In particular we have $\Ho{\Cpx\E} = \Der\E$, the unbounded derived category of $\E$.
\end{thm}

\begin{proof}
We start with the trivial cotorsion pair $(\E,\I)$ in $\E$, whose existence is guaranteed by Corollary~\ref{cor:enough-inj} and which clearly is hereditary. Then we lift it to a complete cotorsion pair $(\tilde\E,\tilde\I)$ in $\Cac\E$---this is in fact an overkill since $\tilde\I = \Inj\Cpx\E$. Then clearly $(\Cpx\E,\tilde\I)$ is a complete cotorsion pair (the trivial one) and $(\tilde\E,\tilde\E^\perp)$ is complete by Lemma~\ref{lem:lifting-completeness}.

Thus, we can apply Proposition~\ref{prop:model-constr} to obtain a hereditary model structure on $\Cpx\E$ determined by $(\Cof,\We,\Fib) = (\Cpx\E,\Cac\E,\tilde\E^\perp)$ (see Theorem~\ref{thm:exact-models}). Note that the weak equivalences are the quasi-isomorphisms in the usual sense, that is the morphisms whose mapping cone is acyclic. To see that, note that the usual mapping cone construction is a special case of a cofiber sequence from Definition~\ref{def:co-fiber} and that $f\dd X \to Y$ is a weak equivalence \iff the third object in some (or any) cofiber sequence is trivial by Lemma~\ref{lem:2-out-of-3}. Hence $\Ho{\Cpx\E} = \Der\E$.
\end{proof}

\subsection{Models for Grothendieck categories}
\label{subsec:groth-model}

If $\G$ is a Grothendieck category, we can use any complete hereditary cotorsion pair to construct a model structure for $\Der\G$. This freedom in the construction is very useful when attempting to construct monoidal model structures as we shall show in Section~\ref{sec:tensor}. Various weaker versions of this result have been previously proved in~\cite{EGPT,G3,G2,G1,SaSt,YL11}.

Let us explain how the construction works. First of all, we will use the following simple adjunction formulas for Ext.

\begin{lem} \label{lem:ext-adj} \cite[4.2]{G4}
Let $\G$ be an abelian category, $X \in \G$ be an object and $Y \in \Cpx\G$ a complex. Then there are natural monomorphisms
\[ \Ext^1_\G(X,Z^n(Y)) \la \Ext^1_{\Cpx\G}(S^n(X),Y) \]
and
\[ \Ext^1_\G(Y^n/B^n(Y),X) \la \Ext^1_{\Cpx\G}(Y,S^n(X)), \]
where $Z^n(Y) = \Ker(Y^n \to Y^{n+1})$ and $B^n(Y) = \Img(Y^{n-1} \to Y^n)$ are the usual $n$-th cocycle and coboundary objects of $Y$, respectively. If, moreover, $Y \in \Cac\G$ is an acyclic complex, then these are isomorphisms.
\end{lem}

Given a complete hereditary cotorsion pair $(\A,\B)$ in $\G$, this allows us to deduce the existence of cotorsion pairs $({^\perp\tilde\B},\tilde\B)$ and $(\tilde\A,\tilde\A^\perp)$ in $\Cpx\G$.

\begin{lem} \label{lem:extension-of-tildes-abel}
Let $\G$ be an abelian category and $(\A,\B)$ be a cotorsion pair in $\G$ \st $\A$ is generating and $\B$ is cogenerating (e.g.\ if $(\A,\B)$ is complete). Then
\[
\tilde\B = \{ S^n(A) \mid A \in \A \textrm{ and } n \in \Z \}^\perp
\quad \textrm{and} \quad
\tilde\A = {^\perp \{ S^n(B) \mid B \in \B \textrm{ and } n \in \Z \}}
\]
In particular we have cotorsion pairs $({^\perp\tilde\B},\tilde\B)$ and $(\tilde\A,\tilde\A^\perp)$ in $\Cpx\G$.
\end{lem}

\begin{proof}
Denote $\X = \{ S^n(A) \mid A \in \A \textrm{ and } n \in \Z \}^\perp$. Since $\A$ is generating, the class $\X$ consists of acyclic complexes. Indeed, component wise split short exact sequences $0 \to X \to E \to S^n(A) \to 0$ must split for every $X \in \X$, so that every cochain complex homomorphism $S^{n+1}(A) \to X$ must be null-homotopic, and consequently $X^n \to Z^{n+1}(X)$ must be an epimorphism for every $n \in \Z$. Then $\X \subseteq \tilde\B$ by Lemma~\ref{lem:ext-adj}. The inclusion $\X \supseteq \tilde\B$ is easy and the other case is dual.
\end{proof}

The important feature is completeness of $({^\perp\tilde\B},\tilde\B)$ and $(\tilde\A,\tilde\A^\perp)$, which we can prove for Grothendieck categories.

\begin{prop} \label{prop:lift-for-Grothendieck}
Let $\G$ be a Grothendieck category and $(\A,\B)$ be a complete hereditary cotorsion pair in $\G$. Then $({^\perp\tilde\B},\tilde\B)$ and $(\tilde\A,\tilde\A^\perp)$ are complete (and hereditary) cotorsion pairs in $\Cpx\G$.
\end{prop}

\begin{proof}
The pair $(\tilde\A,\tilde\A^\perp)$ is complete by Lemmas~\ref{lem:deconstr-acyclic} and~\ref{lem:lifting-completeness}. For $({^\perp\tilde\B},\tilde\B)$, let $G \in \A$ be a generator for $\G$ and consider the cotorsion pair $(\clP,\clP^\perp)$ generated from the left hand side by the set $\{ S^n(G) \mid n \in \Z \}$. Then $\clP \subseteq {^\perp\tilde\B}$ by Lemma~\ref{lem:extension-of-tildes-abel} and also $\clP^\perp \subseteq \Cac\G$. Now we can use the dual version of the proof for Lemma~\ref{lem:lifting-completeness}, just taking the corresponding approximations \wrt $(\clP,\clP^\perp)$ instead of the ones for $(\tilde\G,\tilde\G^\perp)$.
\end{proof}

\begin{rem} \label{rem:dg-notation}
Gillespie~\cite{G3,G2,G1} denotes ${^\perp\tilde\B}$ by $dg\tilde\A$ and $\tilde\A^\perp$ by $dg\tilde\B$. If $\A$ is deconstructible in $\G$, i.e.\ $\A = \Filt\clS$ for a set of objects $\clS$, then both $\tilde\A$ and $dg\tilde\A$ are deconstructible by~\cite[Theorem 4.2]{St2}. This implies that $(dg\tilde\A,\tilde\B)$ and $(\tilde\A,dg\tilde\B)$ are functorially complete and the model structure from Theorem~\ref{thm:inj-model-for-D(Groth)} is cofibrantly generated (see Remark~\ref{rem:cofibrantly-generated}). The objects of $dg\tilde\A$ are then characterized as retracts of objects in $\Filt\{ S^n(X) \mid X \in \clS \textrm{ and } n \in \Z \}$ (see also~\cite[Proposition 4.5]{St2}).
\end{rem}

Now we can state and prove the result on the existence of models for $\Der\G$.

\begin{thm} \label{thm:inj-model-for-D(Groth)}
Let $\G$ be a Grothendieck category and $(\A,\B)$ a complete hereditary cotorsion pair in $\G$. Then there is a hereditary model structure on $\Cpx\G$ \st
\begin{enumerate}
\item The class of cofibrant objects equals ${^\perp\tilde\B}$ (using Notation~\ref{nota:a-tilde}).
\item The trivial objects are precisely the acyclic complexes, $\We = \Cac\G$.
\item The class of fibrant objects equals is $\tilde\A^\perp$.
\end{enumerate}
In particular we have $\Ho{\Cpx\G} = \Der\G$, the unbounded derived category of $\G$.
\end{thm}

\begin{proof}
We lift $(\A,\B)$ to a complete cotorsion pair $(\tilde\A,\tilde\B)$ in $\Cac\G$ and use Proposition~\ref{prop:lift-for-Grothendieck} to prove that $({^\perp\tilde\B},\tilde\B)$ and $(\tilde\A,\tilde\A^\perp)$ are complete cotorsion pairs in $\Cpx\G$. The hereditary model structure determined by $(\Cof,\We,\Fib) = ({^\perp\tilde\B},\Cac\G,\tilde\A^\perp)$ is obtained from Proposition~\ref{prop:model-constr}. The weak equivalences are usual quasi-isomorphism by the same argument as for Theorem~\ref{thm:inj-model-for-D(E)}, so that $\Ho{\Cpx\G} = \Der\G$.
\end{proof}

\subsection{Neeman's pure derived category of flat modules}
\label{subsec:proj-d(flat)}

The last example that we discuss here is due to Neeman~\cite{Nee3}. Although model categories are not mentioned at all in~\cite{Nee3}, there is a direct interpretation of the main results in our language. Basically, Neeman considers two model structures for the derived category of flat modules over a ring---the injective one described in Section~\ref{subsec:injective-model}, and the projective one described in the following theorem.

The significance of $\Der\FlatR$ is that for a left coherent ring $R$ it is compactly generated and the category of compact objects is triangle equivalent to $\Db{\modR\op}\op$, the opposite category of the bounded derived category of the category $\modR\op$ of finitely presented left modules. If $R$ is commutative noetherian, there are interesting connections to the classical Grothendieck duality~\cite{Hart}.

\begin{thm} \label{thm:neeman's}
Let $R$ be a ring and $\FlatR$ be the category of all flat right $R$-modules with the induced exact structure. Then there is a cofibrantly generated hereditary model structure on $\Cpx\FlatR$ \st
\begin{enumerate}
\item The class of cofibrant objects is precisely $\Cpx\ProjR$, the complexes with projective components.
\item The trivial objects are precisely the acyclic complexes, $\We = \Cac\FlatR$.
\item Every object of $\Cpx\FlatR$ is fibrant.
\end{enumerate}
In particular we have $\Ho{\Cpx\FlatR} = \Der\FlatR$.
\end{thm}

\begin{proof}
In view of~\cite[\S3]{SaSt} Neeman proves in~\cite[Theorems 8.6 and 5.9]{Nee3} that $(\Cpx\ProjR,\We)$ is a complete cotorsion pair in $\Cpx\FlatR$ generated by a set $\clS$ as in Theorem~\ref{thm:cotorsion-complete} and that $\Cpx\ProjR \cap \We = \Cac\ProjR$. One also easily checks that $\Cac\ProjR$ are precisely the projective objects in $\Cpx\FlatR$. Thus we see that if $\A = \ProjR$ and $\B = \FlatR$, then $(\A,\B)$ is a complete hereditary cotorsion pair in $\FlatR$ and both $({^\perp\tilde\B},\tilde\B)$ and $(\tilde\A,\tilde\A^\perp)$ are complete in $\Cpx\FlatR$. It remains to apply Proposition~\ref{prop:model-constr} and Remark~\ref{rem:cofibrantly-generated}.
\end{proof}

Another main result of~\cite{Nee3} comes as an easy consequence.

\begin{cor} \label{cor:htp-proj}
Let $R$ be a ring. Then $\Der\FlatR$ is triangle equivalent to $\Htp\ProjR$, the homotopy category of complexes of projective modules.
\end{cor}

\begin{proof}
By restriction, the model structure on $\Cpx\FlatR$ induces a hereditary model structure on $\Cpx\ProjR$ given by
\[ (\Cof,\We,\Fib) = (\Cpx\ProjR,\Cac\ProjR,\Cpx\ProjR). \]
There is an obvious functor $\Ho\Cpx\ProjR \to \Ho\Cpx\FlatR$ which is a triangle equivalence by~\cite[Corollary 5.4]{G5} and Theorem~\ref{thm:alg-tria}. Moreover, $\Ho\Cpx\ProjR$ is none other than $\Htp\ProjR$ by Lemma~\ref{lem:exact-homotopies}(3).
\end{proof}

\section{Monoidal model categories from flat sheaves and vector bundles}
\label{sec:tensor}

In the final section we will consider exact model structures which are compatible with a symmetric monoidal product on the underlying category. As an illustration we will show, under appropriate assumptions, constructions of model structures compatible with the tensor product for the derived category of the category of quasi-coherent sheaves over a scheme.

\subsection{Tensor product for quasi-coherent modules}
\label{subsec:tensor-Qco}

Let us briefly recall what a closed symmetric monoidal category is, and refer to~\cite[Chapter VII]{McL2} for details. A category $\C$ is called \emph{symmetric monoidal} if it is equipped with a functor $\otimes\dd \C \times \C \to \C$, often called tensor product, and an object $\unit \in \C$ \st $\otimes$ is associative, commutative and $\unit$ is a unit \wrt $\otimes$. More precisely, there are natural isomorphisms $X \otimes (Y \otimes Z) \cong (X \otimes Y) \otimes Z$, $X \otimes Y \cong Y \otimes X$ and $\unit \otimes X \cong X$ satisfying certain coherence conditions. A symmetric monoidal category is called \emph{closed} if the functor $- \otimes Y$ has a right adjoint for each $Y \in \C$. If we for each $Y$ fix such an adjoint and denote it by $\HOM(Y,-)$, it is a standard fact that $\HOM$ is a functor $\C\op \times \C \to \C$ and we obtain an isomorphism
\[ \C(X \otimes Y, Z) \overset{\cong}\la \C(X, \HOM(Y, Z)), \]
which is natural in all three variables $X,Y,Z$.

If $R\dd I \to \mathrm{CommRings}$ is a flat representation of a poset in the category of commutative rings (Definitions~\ref{def:poset-ring} and~\ref{def:flat}), then $\QcoR$ is a closed symmetric monoidal category. Indeed, using the notation from Definition~\ref{def:poset-module} we define $L \otimes M$ by $(L \otimes M)(i) = L(i) \otimes_{R(i)} M(i)$ and $(L \otimes M)^i_j = L^i_j \otimes M^i_j$ for each $i\le j$ from $I$. The tensor unit $\unit$ is clearly $R$ itself viewed as an object of $\QcoR$. 

It is not difficult to see that $\QcoR$ is closed, see~\cite[Proposition 6.15]{Mur}. Indeed, for $M,N \in \QcoR$ we can define an $R$-module $\underline{\HOM}(M,N) \in \ModR$ by putting $\underline{\HOM}(M,N)(i) = \Hom_{R(i)}(M(i),N(i))$ for each $i \in I$, and use the quasi-coherence of $M$ and $N$ to define the maps inside $\underline{\HOM}(M,N)$ as $\underline{\HOM}(M,N)^i_j\dd f \mapsto f \otimes_{R(i)} R(j)$ for $i\le j$. Although we obtain the adjunction
\[ \Hom_R(L \otimes M, N) \overset{\cong}\la \Hom_R(L, \underline{\HOM}(M, N)), \]
the $R$-module $\underline{\HOM}(M,N)$ may not be quasi-coherent unless all the $M(i)$ are finitely presented over $R(i)$. We can fix the problem by defining $\HOM(M,N) = Q(\underline{\HOM}(M,N))$, where $Q$ is the coherator from Remark~\ref{rem:products}.

There is a standard way to extend this to a closed symmetric monoidal structure on $\Cpx\QcoR$. Given $X,Y \in \Cpx\QcoR$, we define components of the complex $X \otimes Y$ by
\[ (X \otimes Y)^n = \coprod_{i+j=n} X^i \otimes Y^j \qquad \textrm{for each n} \in \Z, \]
and the differentials by $\dif|_{X^i \otimes Y^j} = \dif^i_X \otimes Y^j + (-1)^i X^i \otimes \dif^j_Y$ for all $i,j \in \Z$. The unit $\unit$ is the complex $S^0(R)$ (Notation~\ref{nota:discs-and-spheres}). Similarly, for $Y,Z \in \Cpx\QcoR$, the complex $\HOM(Y,Z)$ is defined by
\[ \HOM(Y,Z)^n = \prod_{j-i=n} \HOM(Y^i, Z^j) \qquad \textrm{for each n} \in \Z, \]
where the product is taken in $\QcoR$ (see Remark~\ref{rem:products}), and by $\dif|_{\HOM(Y^i, Z^j)} = \HOM(Y^i,\dif^j_Z) - (-1)^{j-i} \HOM(\dif^i_Y,Z^j)$ for all $i,j \in \Z$.

This closed symmetric monoidal category $(\Cpx\QcoR, \otimes, \unit, \HOM)$ is the one to keep in mind in the discussion to come.

\subsection{Derived functors}
\label{subsec:derived}

In order to give a good sense to what a compatibility of a monoidal structure with a model structure means, we need to give a brief account on derived functors. We refer to~\cite[\S\S8.4 and 8.5]{Hir} for more details. After all, what we need to do is to construct the derived functors of $\otimes\dd \C \times \C \to \C$ and $\HOM\dd \C\op \times \C \to \C$.

Let $\C,\D$ be model categories and $F\dd \C \to \D$ be a functor. We would like to define a functor $\Ho F$ making the diagram
\[
\xymatrix{
\C \ar[r]^F \ar[d]_{Q_\C} & \D \ar[d]^{Q_\D} \\
\Ho\C \ar@{.>}[r]^{\Ho F} & \Ho\D
}
\]
commutative, but we usually cannot as $F$ typically does not send weak equivalences to weak equivalences. What we usually can do, however, is to find the ``best approximation'' from the left or from the right.

\begin{defn} \label{def:total-derived} \cite[8.4.1 and 8.4.7]{Hir}
Let $F\dd \C \to \D$ be a functor between model categories $\C$ and $\D$. A \emph{total left derived functor} of $F$ is a functor $\Lder F\dd \Ho\C \to \Ho\D$ together with a natural transformation $\ep\dd \Lder F \circ Q_\C \to Q_D \circ F$ satisfying the following universal property. If $G\dd \Ho\C \to \Ho\D$ is another functor with a natural transformation $\zeta\dd G \circ Q_\C \to Q_D \circ F$, then there is a unique natural transformation $\tht\dd G \to \Lder F$ \st $\zeta = \ep (\tht\circ Q_\C)$.

A \emph{total right derived functor} of $F$ is a functor $\Rder F\dd \Ho\C \to \Ho\D$ together with a natural transformation $\eta\dd Q_D \circ F \to \Rder F \circ Q_\C$ which satisfies the obvious dual universal property.
\end{defn}

That is, we are often in the situation of one of the following universal squares:
\[
\xymatrix{
\C \ar[r]^F \ar[d]_{Q_\C} & \D \ar[d]^{Q_\D} \\
\Ho\C \ar@{.>}[r]_{\Lder F} & \Ho\D \ultwocell\omit
}
\qquad\qquad\qquad
\xymatrix{
\C \ar[r]^F \ar[d]_{Q_\C} \drtwocell\omit & \D \ar[d]^{Q_\D} \\
\Ho\C \ar@{.>}[r]_{\Rder F} & \Ho\D
}
\]
The key point is that if $\Lder F$ (or $\Rder F$) exists, it is defined uniquely up to a unique isomorphism, and does not depend on the actual model structures on $\C$ and $\D$, but only on the classes of weak equivalences in $\C$ and $\D$.

We shall list standard criteria for existence of total derived functors.

\begin{prop} \label{prop:derived-exist}
Let $F\dd \C \to \D$ be a functor between model categories $\C$ and $\D$.
\begin{enumerate}
\item If $F$ takes trivial cofibrations between cofibrant objects in $\C$ to weak equivalences in $\D$ (see also Lemma~\ref{lem:2-out-of-3} for exact model categories), then the total left derived functor $\Lder F\dd \Ho\C \to \Ho\D$ exists.
\item If $F$ takes trivial fibrations between fibrant objects in $\C$ to weak equivalences in $\D$, then the total right derived functor $\Rder F\dd \Ho\C \to \Ho\D$ exists.
\end{enumerate}
\end{prop}

\begin{proof}
This is just \cite[Proposition 8.4.8]{Hir}, but as we dropped the assumption of functoriality of the factorizations in Definition~\ref{def:model}(MS2) while~\cite{Hir} uses the functoriality, we indicate the necessary modification in the proof of~\cite[Proposition 8.4.4]{Hir}. Suppose we are in the situation of (1). Then we construct $\Lder F$ as follows. Given $X \in \C$, we fix a cofibrant replacement $f_X\dd CX \to X$ of $X$ (see Definition~\ref{def:replacements}) and put $\Lder F(X) = Q_DF(CX)$. Given a morphism $u\dd X \to Y$ in $\C$, we choose $f_u$ to make the square
\[
\begin{CD}
  CX @>{f_X}>> X     \\
@V{f_u}VV   @VV{u}V  \\
  CY @>{f_Y}>> Y     \\
\end{CD}
\]
commute and put $\Lder F(u) = Q_DF(f_u)$. This is well defined since if $f'_u\dd CX \to CY$ is another map which fits into the square, then $f_u$ and $f'_u$ are left homotopic (see also Lemma~\ref{lem:exact-homotopies} for exact model categories), hence $Q_DF(f_u) = Q_DF(f'_u)$. The rest of the proof is as in~\cite{Hir} and case (2) is dual.
\end{proof}

\begin{expl} \label{expl:derived}
As this formalism is still somewhat unusual in homological algebra, we shall illustrate it on the probably most usual situation with derived categories. Let $R,S$ be rings and $F\dd \ModR \to \Modr{S}$ be an additive functor. Then we have an obvious functor $F\dd \Cpx\ModR \to \Cpx{\Modr{S}}$, which we also denote by $F$. Now we equip $\Cpx\ModR$ and $\Cpx{\Modr{S}}$ with the two extremal hereditary model structures coming from Theorem~\ref{thm:inj-model-for-D(Groth)}: We use the cotorsion pair $(\ProjR,\ModR)$ for $\Cpx\ModR$ and the cotorsion pair $(\Modr{S},\Inj\textrm{-}S)$ for $\Cpx{\Modr{S}}$. It is an easy observation that $F$ satisfies Proposition~\ref{prop:derived-exist}(1). Thus $\Lder F\dd \Der\ModR \to \Der{\Modr{S}}$ exists.
\end{expl}

In connection with closed symmetric monoidal categories, we also need to know how derived functors of adjoint pairs behave. Typically if $(F,U)$ is an adjoint pair of functors and if $\Lder F$ and $\Rder G$ exist, they also form an adjoint pair. A very general setting where this holds has been described by Maltsiniotis in~\cite{Malt07}. For simplicity, we only state the result for classical Quillen pairs of functors.

\begin{prop} \label{prop:derived-adj}
Let $\C,\D$ be model categories and $(F,U)$, where $F\dd \C \to \D$ and $U\dd \D \to \C$, be an adjoint pair of functors. Then the following are equivalent:
\begin{enumerate}
\item The left adjoint $F$ preserves both cofibrations and trivial cofibrations.
\item The right adjoint $U$ preserves both fibrations and trivial fibrations.
\end{enumerate}
In such a case, both $\Lder F$ and $\Rder U$ exist and $(\Lder F,\Rder U)$ is an adjoint pair.
\end{prop}

\begin{proof}
This is proved in \cite[Proposition 8.5.3 and Theorem 8.5.18]{Hir}.
\end{proof}

\begin{defn} \label{def:quillen-pair}
Adjoint pairs $(F,U)$ as in Proposition~\ref{prop:derived-adj} are called \emph{Quillen pairs}.
\end{defn}

\subsection{Monoidal model categories}
\label{subsec:symmetric-model}

Now we are left with a relatively easy task: given a closed symmetric monoidal category, determine which condition a model structure on $\C$ must satisfy so that we can use it to construct the derived functors of $\otimes\dd \C \times \C \to \C$ and $\HOM\dd \C\op \times \C \to \C$. Note that if $\C$ and $\D$ carry model structures, then $\C \times \D$ admits a so-called product model structure, where a morphism is a cofibration, a weak equivalence, or a fibration, respectively, if both the components in $\C$ and $\D$ are such (cf.\ \cite[Example 1.1.6]{H2}). Similarly, if $\C$ is a model category, $\C\op$ is naturally a model category too, by swapping the roles of fibrations and cofibrations (see \cite[Remark 1.1.7]{H2}).

\begin{rem} \label{rem:derive-tensor}
An important point is that for computing the derived functor of $\otimes\dd \C \times \C \to \C$, we need \emph{not} equip all three copies of $\C$ with the same model structure. Indeed, if $R$ is a commutative ring and we are to compute $X \Lotimes_R Y$ in $\C = \Der\ModR$, it is enough to find a suitable flat resolution for $X$ only. By doing so, we actually consider the first copy of $\C$ with the model structure given by Theorem~\ref{thm:inj-model-for-D(Groth)} for the cotorsion pair $(\FlatR,\FlatR^\perp)$ (see also Section~\ref{subsec:flat-cpx}), while the other two copies of $\C$ are considered with the injective model structure given by the trivial cotorsion pair $(\ModR,\InjR)$.
\end{rem}

For theoretical reasons, however, it is desirable to be able to equip $\C$ with a single model structure which can be used for computing both the total left derived functor $\Lotimes$ of $\otimes$ and the total right derived functor $\RHOM$ of $\HOM$. A condition which the model structure on $\C$ should satisfy is stated in the following definition.

\begin{defn} \label{def:quillen-bifunctor}
Let $\C,\D,\E$ be model categories. A functor $\otimes\dd \C \times \D \to \E$ is called a \emph{Quillen bifunctor} if the following condition is satisfied. If $f\dd U \to V$ is a cofibration in $\C$ and $g\dd X \to Y$ is a cofibration in $\D$, we require that the map $f \boxtimes g$ defined as the dotted arrow in the following pushout diagram (which is assumed to exist)
\[
\xymatrix{
U \otimes X \ar[r] \ar[d] & V \otimes X \ar[d] \ar@/^/[ddr] \\
U \otimes Y \ar[r] \ar@/_/[drr] & P \ar@{.>}[dr]|(.4){f \boxtimes g}  \\
&& V \otimes Y
}
\]
be a cofibration in $\E$. Moreover, if one of $f$ and $g$ is trivial, we require $f \boxtimes g$ to be trivial.
\end{defn}

The following lemma is standard.

\begin{lem} \label{lem:adj-2-var} \cite[4.2.2]{H2}
Let $\C,\D,\E$ be model categories and let $\otimes\dd \C \times \D \to \E$ be a Quillen bifunctor. If $\HOM\dd \D\op \times \E \to \C$ is a functor \st we have a natural isomorphism
\[ \E(X \otimes Y, Z) \overset{\cong}\la \C(X, \HOM(Y, Z)), \]
in all three variables, then $\HOM$ is a Quillen bifunctor when we view it as a functor $\D \times \E\op \to \C\op$ of the opposite categories.
\end{lem}

Now we can present a slightly simplified version of~\cite[Definition 4.2.6]{H2} along with the expected theorem.

\begin{defn} \label{def:monoidal-model}
A \emph{monoidal model category} is a model category $\C$ (Definition~\ref{def:model}) which carries a closed symmetric monoidal structure $(\C,\otimes,\unit,\HOM)$ \st
\begin{enumerate}
\item[(MS3)] The tensor unit $\unit$ is cofibrant.
\item[(MS4)] The product $\otimes\dd \C \times \C \to \C$ is a Quillen bifunctor.
\end{enumerate}

An \emph{exact monoidal model category} and \emph{hereditary monoidal model category} are the corresponding variants where $\C$ is an exact model category (Definition~\ref{def:exact-models}) or hereditary model category (Definition~\ref{def:hered-model}), respectively.
\end{defn}

\begin{thm} \label{thm:monoidal-model}
Let $\C$ be a monoidal model category. Then the total derived functors
\[
\Lotimes\dd \Ho\C \times \Ho\C \la \Ho\C
\qquad \textrm{and} \qquad
\RHOM\dd \Ho\C\op \times \Ho\C \la \Ho\C
\]
exist and $(\Ho\C,\Lotimes,\unit,\RHOM)$ is a closed symmetric monoidal category.
\end{thm}

\begin{proof}
This is just a special case of~\cite[Theorem 4.3.2]{H2}, but at this point we have all the tools to understand the proof. Since $\otimes$ is a Quillen bifunctor, $f\otimes g\dd U \otimes X \to V \otimes Y$ is a cofibration whenever $f\dd U \to V$ and $g\dd X \to Y$ are cofibrations, and $f \otimes g$ is a trivial cofibration if one of $f,g$ is such. Similarly we deduce from Lemma~\ref{lem:adj-2-var} that $\HOM(g,p)$ is a fibration whenever $g$ is a cofibration and $p$ is a fibration, and it is a trivial fibration if one of $g,p$ is trivial. In particular, the total left derived functor $\Lotimes$ of $\otimes\dd \C \times \C \to \C$ and the total right derived functor $\RHOM$ of $\HOM\dd \C\op \times \C \to \C$ exist by Proposition~\ref{prop:derived-exist}.

It also follows that if $X$ is cofibrant, then $(- \otimes X,\HOM(X,-))$ is a Quillen pair. Thus $(- \Lotimes X,\RHOM(X,-))$ is a pair of adjoint functors by Proposition~\ref{prop:derived-adj}.

Using the construction of $\Lotimes$ in the proof of Proposition~\ref{prop:derived-exist}, one also sees that $- \Lotimes (- \Lotimes -)$ coincides with the total left derived functor of $- \otimes (- \otimes -)$, and $(- \Lotimes -) \Lotimes -$ is the left derived functor of $(- \otimes -) \otimes -$. Hence the associativity transformation for $\otimes$ lifts to associativity for $\Lotimes$, using the universal property of the derived functors. The commutativity transformation for $\Lotimes$ and the isomorphism $\unit \Lotimes - \cong 1_{\Ho\C}$ are obtained similarly. Finally, the transformations for $\Lotimes$ satisfy the coherence conditions since they were constructed from the transformations for $\otimes$ using the universal property.
\end{proof}

Naturally we are interested in what conditions a triple of classes $(\Cof,\We,\Fib)$ as in Theorem~\ref{thm:exact-models} must satisfy in order to define an exact or hereditary monoidal model structure. Now we can give a relatively easy answer, which was first obtained in~\cite[Theorem 7.2]{H1}.

\begin{thm} \label{thm:monoidal-exact} \cite[7.2]{H1}
Let $\E$ be an exact model category given in the sense of Theorem~\ref{thm:exact-models} by the triple $(\Cof,\We,\Fib)$ of cofibrant, trivial, and fibrant objects, respectively. Suppose further that $(\E,\otimes,\unit,\HOM)$ is also a closed symmetric monoidal category. Then $\E$ with all the structure forms an exact monoidal model category provided that
\begin{enumerate}
\item The tensor unit $\unit$ is cofibrant.
\item If $X,Y \in \C$, then $X \otimes Y \in \C$. If, moreover, $X$ or $Y$ is in $\Cof\cap\We$, then $X \otimes Y \in \Cof\cap\We$.
\item If $f\dd X \to Y$ is a cofibration and $U \in \E$ is arbitrary, then $U \otimes f\dd U \otimes X \to U \otimes Y$ is an inflation in $\E$.
\end{enumerate}
\end{thm}

\begin{rem} \label{rem:monoidal-exact}
Note a subtle point here: In principle the definition of a monoidal model category requires some extra pushouts to exist so that the maps $f \boxtimes g$ from Definition~\ref{def:quillen-bifunctor} are defined for any pair of cofibrations $f,g$. Condition (3) from the latter theorem ensures, however, that \emph{no} extra pushouts except for those guaranteed by the axioms of an exact category are required.
\end{rem}

\begin{proof}[Proof of Theorem~\ref{thm:monoidal-exact}]
We only need to show that (2) and (3) imply that $\otimes$ is a Quillen bifunctor (see Definition~\ref{def:quillen-bifunctor}). To this end, let $f\dd U \to V$ and $g\dd X \to Y$ be cofibrations with cokernels $C$ and $D$, respectively. Using (3) and~\cite[Proposition 3.1]{Bu}, we construct the following commutative diagram with conflations in rows:
\[
\begin{CD}
0   @>>>   U \otimes X   @>{f \otimes X}>>   V \otimes X   @>>>   C \otimes X   @>>>   0  \\
@.      @V{U \otimes g}VV                        @VVV                  @|                 \\
0   @>>>   U \otimes Y   @>>>                      P       @>>>   C \otimes X   @>>>   0  \\
@.              @|                        @V{f \boxtimes g}VV  @V{C \otimes g}VV          \\
0   @>>>   U \otimes Y   @>{f \otimes Y}>>   V \otimes Y   @>>>   C \otimes Y   @>>>   0  \\
\end{CD}
\]
Thanks to~\cite[Proposition 2.12]{Bu}, we have a conflation
\[
\begin{CD}
0 @>>> P @>{f \boxtimes g}>> V \otimes Y @>>> C \otimes D @>>> 0.
\end{CD}
\]
Now the conclusion follows from assumption (2).
\end{proof}

\subsection{Complexes of flat quasi-coherent modules}
\label{subsec:flat-cpx}

The last task is to instantiate Theorem~\ref{thm:monoidal-exact} for the closed symmetric monoidal structure on $\Cpx\QcoR$ discussed in \S\ref{subsec:tensor-Qco}. The hardest part is to fulfil Theorem~\ref{thm:monoidal-exact}(3), as this typically does not hold for the injective model structure on $\Cpx\QcoR$ from Theorem~\ref{thm:inj-model-for-D(E)}. Eventually we will need to impose more assumptions on the flat representation $R\dd I \to \mathrm{CommRings}$, but let us first discuss the aspects which work in full generality.

Recall from \S\ref{subsec:Qco-mod} that a left module $M$ over a ring $T$ is called \emph{flat} if the functor $- \otimes_T M\dd \Modr{T} \to \Ab$ is exact. By Lazard's theorem~\cite[Corollary 1.2.16]{GT}, $M$ is flat \iff it is a direct limit of finitely generated projective modules.

Analogously we call $M \in \QcoR$ \emph{flat} if the functor $-\otimes M\dd  \QcoR \to \QcoR$ is exact. This notion of flatness corresponds to the classical one for quasi-coherent sheaves if $R$ comes from a scheme in the sense of \S\ref{subsec:Qco-sheaves}. It is easy to check that $M$ is flat \iff $M(i)$ is a flat module over $R(i)$ for each $i \in I$. We shall denote the class of all flat quasi-coherent modules over $R$ by $\FlatR$. Our aim is to use Corollary~\ref{cor:lhs-cotorsion} to show that there is often a complete hereditary cotorsion pair $(\FlatR,\FlatR^\perp)$. We denote the class $\FlatR^\perp$ by $\CotR$ and, as customary, we call the quasi-coherent modules inside it \emph{cotorsion}.

First of all, the class of left flat modules over a ring $T$ is always deconstructible by the proof of~\cite[Theorem 3.2.9]{GT}, and it is a left hand side class of a complete hereditary cotorsion pair by~\cite[Theorem 4.1.1]{GT}. Regarding the deconstructibility of $\FlatR$ for $R\dd I \to \mathrm{CommRings}$, we can use the following technical result, whose proof involves the Hill Lemma (see Proposition~\ref{prop:hill-lemma}).

\begin{prop} \label{prop:flat-qco-R}
Let $R\dd I \to \mathrm{Rings}$ be a flat representation of $I$ in the category of rings, and suppose that we are given for each $i \in I$ a deconstructible class $\C_i \subseteq \Modr{R(i)}$. Then the class $\C \subseteq \QcoR$ defined as
\[ \C = \{ M \in \QcoR \mid M(i) \in \C_i \textrm{ for each } i \in I \} \]
is deconstructible in~$\QcoR$.
\end{prop}

\begin{proof}
The same proof as for~\cite[Theorem 3.13 and Corollary 3.14]{EGPT} applies.
\end{proof}

\begin{cor} \label{cor:flat-qco-R}
If $R\dd I \to \mathrm{CommRings}$ is a flat representation, then the class $\FlatR$ is deconstructible in $\QcoR$ and it is closed under retracts and kernels of epimorphisms.
\end{cor}

Thus, in view of Corollary~\ref{cor:lhs-cotorsion} and Lemma~\ref{lem:hered-cot}, the following is sufficient and necessary in order to obtain a complete hereditary cotorsion pair.

\begin{cor} \label{lem:flat-generator}
If $R\dd I \to \mathrm{CommRings}$ is a flat representation, then there exists a (functorially) complete hereditary cotorsion pair $(\FlatR,\CotR)$ in $\QcoR$ \iff $\QcoR$ admits a flat generator.
\end{cor}

Constructing a flat generator is indeed the most tricky part. It exists in $\Qco{X}$ for any quasi-compact and separated scheme $X$, as has been implicitly proved in~\cite[Proposition 1.1]{AJL97} and discussed explicitly and in detail by Murfet in~\cite[\S3.2]{Mur}.

To formulate this in our formalism from~\S\ref{subsec:Qco-mod}, recall that a poset $(I,\le)$ is an \emph{upper semilattice} if every pair $x,y \in I$ has a least upper bound $x \vee y$. We will call a representation $R\dd I \to \mathrm{CommRings}$ of an upper semilattice $I$ \emph{continuous} if it preserves pushouts. In other words, if $x,y,z \in I$ \st $x \le y$ and $x \le z$, then the ring homomorphism
\[ R^y_{y\vee z}\otimes_{R(x)} R^z_{y\vee z}\dd R(y) \otimes_{R(x)} R(z) \to R(y\vee z) \]
is required to be bijective. If $X$ is a quasi-compact separated scheme, then the results from \S\ref{subsec:Qco-sheaves} can be used to construct an equivalence $\QcoR \overset{\cong}\la \Qco{X}$ for a continuous flat representation $R$ of a finite upper semilattice $I$. This explains the motivation for the following statement which (at least formally) slightly generalizes the above mentioned result from~\cite{AJL97,Mur}.

\begin{prop} \label{prop:enough-flats}
Let $I$ be a finite upper semilattice and $R\dd I \to \mathrm{CommRings}$ be a continuous flat representation. Then $\QcoR$ admits a flat generator. In particular, there is a functorially complete hereditary cotorsion pair $(\FlatR,\CotR)$ in $\QcoR$.
\end{prop}

\begin{proof}
We shall closely follow the structure of the proof from~\cite{Mur}. Let us fix a finite number of elements $x_0, x_1, \dots, x_n \in I$ \st every $x \in I$ is of the form $x = x_{i_0} \vee \dots \vee x_{i_p}$ for some $0 \le i_0 < \cdots < i_p \le n$. Since $R$ is a continuous representation, the functor
\[ F_x^*\dd \QcoR \la \Modr{R(x)}, \quad M \longmapsto M(x) \]
for any given $x \in I$ has a right adjoint ${F_x}_*\dd \Modr{R(x)} \to \QcoR$ given by
\[ ({F_x}_*(X))(y) = X \otimes_{R(x)} R(x \vee y) \qquad \textrm{for every $y \in I$.} \]
These are the analogs of the inverse and the direct image functors in geometry; see~\cite[(7.8)]{GW}. Clearly, ${F_x}_*$ sends flat $R(x)$-modules to flat quasi-coherent $R$-modules. For any $(p+1)$-tuple $i_0 < \cdots < i_p$ of integers between $0$ and $n$ we denote $M_{i_0,\dots,i_p} = {F_x}_* F_x^*(M)$ where $x = x_{i_0} \vee \dots \vee x_{i_p}$. In particular we have the equality $M_{i_0,\dots,i_p}(z) = M(z\vee x_{i_0} \vee \dots \vee x_{i_p})$ for each $z \in I$.

Next we construct for each $M \in \QcoR$ a so-called \emph{\v{C}ech resolution}; see~\cite[\S III.4, p. 218]{Hart2}. This is a sequence of morphisms
\[
0 \la M \overset{\dif}\la \C^0M \overset{\dif}\la \C^1M \overset{\dif}\la \cdots \overset{\dif}\la \C^nM \la 0,
\eqno{(\Delta)}
\]
where $\C^pM = \bigoplus_{i_0<i_1<\cdots<i_p} M_{i_0,i_1,\dots i_p}$ for every $0 \le p \le n$. In order to describe the maps $\dif$, we need some more notation. Given $i_0<i_1<\cdots<i_p$, we shall denote by $\alpha_{i_0,i_1,\cdots,i_p}$ an element of $M_{i_0,i_1,\dots,i_p}$. Then $\dif\dd \C^pM \to \C^{p+1}M$ will be given by the rule
\[
\dif(z)\dd
(\alpha_{i_0,i_1,\cdots,i_p})_{i_0<i_1<\cdots<i_p}
\longmapsto
\Big(\sum_{k=0}^{p+1} (-1)^k \alpha_{j_0,\dots,\hat j_k,\dots,j_{p+1}}\Big)_{j_0<j_1<\cdots<j_{p+1}}
\]
for all $z \in I$, where $\hat j_k$ means ``omit $j_k$.'' Strictly speaking we should have written $M^x_y(\alpha_{j_0,\dots,\hat j_k,\dots,j_{p+1}})$ for suitable $x,y \in I$ on the right hand side, but we omit $M^x_y$ for the sake of clarity of the formula as there is no risk of confusion. It is rather easy to check that $(\Delta)$ is a complex in $\QcoR$.

More is true, however---we claim that $(\Delta)$ is an exact sequence. This is proved similarly to~\cite[Lemma III.4.2]{Hart2}. To this end we shall use the same convention as in~\cite[Remark 4.0.1]{Hart2}: if $i_1,\dots,i_p \in \{0,1,\dots,n\}$ is an arbitrary $(p+1)$-tuple, then $\alpha_{i_0,i_1,\cdots,i_p}$ will stand either for $0$ if there is a repeated index, or for $\sgn(\sigma) \cdot \alpha_{\sigma i_0,\sigma i_1,\cdots, \sigma i_p}$, where $\sigma$ is the permutation \st $\sigma i_0 < \sigma i_1 < \cdots < \sigma i_p$. Using this notation, observe that, for any fixed $z \in I$, the complex of $R(z)$-modules
\[
0 \la M(z) \la \C^0M(z) \la \C^1M(z) \la \cdots \la \C^nM(z) \la 0
\]
is contractible, where the null-homotopy $s\dd \C^pM(z) \to \C^{p-1}M(z)$ is given by the rule
\[
s\dd
(\alpha_{i_0,i_1,\cdots,i_p})_{i_0<i_1<\cdots<i_p}
\longmapsto
(\alpha_{z,j_0,j_1,\cdots,j_p})_{j_0<j_1<\cdots<j_{p-1}}
\]
This proves the claim.

From this point we continue exactly as in~\cite[Proposition 3.19]{Mur}. The cotorsion pair $(\FlatR(z),\CotR(z))$ is complete for each $z \in I$ by the above discussion. Using the approximation sequences, we construct a so-called proper flat resolution
\[ \cdots \la P_2^z \la P_1^z \la P_0^z \la M(z) \la 0 \]
of $M(z)$ in $\Modr{R(z)}$. This by definition means that all $P_i^z$ are flat, the sequence is exact, and it remains exact when we apply the functor $\Hom_{R(z)}(P',-)$ for any $P' \in \FlatR(z)$. Using the adjunction $(F_z^*,{F_z}_*)$, it is easy to see that
\[ \cdots \la {F_z}_*(P_2^z) \la {F_z}_*(P_1^z) \la {F_z}_*(P_0^z) \la {F_z}_* F_z^*(M) \la 0 \]
is a proper flat resolution in $\QcoR$. Summing over all $z = x_{i_0} \vee \dots \vee x_{i_p}$ where $i_0<i_1<\cdots<i_p$, we obtain a proper flat resolution for each $\C^pM$ with $0 \le p \le n$. Finally, the argument described in detail in the proof of~\cite[Proposition 3.19]{Mur} allows us to combine these to a proper flat resolution of $M$.
\end{proof}

For more special flat representations $R\dd I \to \mathrm{CommRings}$ we may obtain left hand side classes of complete hereditary cotorsion pairs which are even closer to projective objects. A variety of such candidates and their appropriateness in geometry has been studied in~\cite{EGT,EGPT}. Here we mention only the simplest of them.

\begin{defn} \label{def:bundle}
Let $R\dd I \to \mathrm{CommRings}$ be a flat representation. A module $M \in \QcoR$ is called \emph{locally projective} (or \emph{vector bundle} in~\cite[\S2]{Dr}) if $M(i)$ is a projective $R(i)$-module for every $i \in I$. We will denote the class of locally projective quasi-coherent sheaves by~$\VectR$.
\end{defn}

\begin{prop} \label{prop:generating-bundle}
Let $R\dd I \to \mathrm{CommRings}$ be \st $\QcoR \overset{\cong}\la \Qco{X}$ for a quasi-projective scheme $X$ over an affine scheme (see e.g. Examples~\ref{expl:P1k} and~\ref{expl:P2k}). Then $\QcoR$ has a locally projective generator. In particular, we obtain a functorially complete hereditary cotorsion pair $(\VectR,\VectR^\perp)$.
\end{prop}

\begin{proof}
The first part follows from~\cite[Lemma 2.1.3]{TT}. The deconstructibility of $\VectR$ is obtained by Proposition~\ref{prop:flat-qco-R} and the rest is similar as for flat modules.
\end{proof}

\begin{rem} \label{rem:vect-perp}
Interestingly enough, nothing seems to be known about the class $\VectR^\perp$, it even seems to lack a name so far.
\end{rem}

\subsection{Monoidal models for derived categories of quasi-coherent sheaves}
\label{subsec:monoidal-models-derived}

We conclude by summarizing our findings regarding the existence of monoidal model structures for $\Der\QcoR$.

\begin{thm} \label{thm:monoidal-D(Qco)}
Let $R\dd I \to \mathrm{CommRings}$ be a continuous flat representation of a finite upper semilattice $I$ (in particular for any quasi-compact separated scheme $X$ there is such an $R$ with $\QcoR \overset{\cong}\la \Qco{X}$). Then there is a hereditary monoidal model structure on $(\Cpx\QcoR,\otimes,S^0(R),\HOM)$ (see Definitions~\ref{def:hered-model} and~\ref{def:monoidal-model}) \st
\begin{enumerate}
\item An object is cofibrant \iff it is a retract of an $\{S^n(F) \mid F \in \FlatR \textrm{ and } n \in \Z \}$-filtered object.
\item The class of trivial objects equals $\We = \Cac\QcoR$.
\item The class of fibrant objects equals $\tilde\A^\perp$, where we use Notation~\ref{nota:a-tilde} for $\A =\FlatR$.
\end{enumerate}

If, moreover, $\QcoR \overset{\cong}\la \Qco{X}$ for a quasi-projective scheme $X$ over an affine scheme, then there is another hereditary monoidal model structure on the category $\Cpx\QcoR$ \st
\begin{enumerate}
\item[(1')] An object is cofibrant \iff it is a retract of an $\{S^n(V) \mid V \in \VectR \textrm{ and } n \in \Z \}$-filtered object.
\item[(2')] The class of trivial objects equals $\We = \Cac\QcoR$.
\item[(3')] The class of fibrant objects equals $(\tilde\A')^\perp$, where we use Notation~\ref{nota:a-tilde} for $\A' =\VectR$.
\end{enumerate}

In both cases, $\Ho\Cpx\QcoR$ is equal to $\Der\QcoR$ and the model structures are cofibrantly generated.
\end{thm}

\begin{proof}
In order to construct the hereditary model structures on $\Cpx\QcoR$, we simply input the two complete hereditary cotorsion pairs $(\FlatR,\CotR)$ and $(\VectR,\VectR^\perp)$ to Theorem~\ref{thm:inj-model-for-D(Groth)}. The fact that the model structures are cofibrantly generated follows, in view of Remark~\ref{rem:dg-notation}, from the results in \S\ref{subsec:flat-cpx}. To prove that the model structures are monoidal, it remains to notice that assumption (3) from Theorem~\ref{thm:monoidal-exact} is satisfied whenever every cofibrant complex has all components flat. This follows by analyzing the monoidal structure on $\Cpx\QcoR$ as described in~\S\ref{subsec:tensor-Qco} and boils down to the following well known fact: If a short exact sequence $\ep\dd 0 \to X \to Y \to Z \to 0$ of right modules over a ring $T$ has the last term $Z$ flat, then $\ep \otimes_T U$ is exact for any left $T$-module $U$.
\end{proof}

\bibliographystyle{abbrv}
\bibliography{references}

\end{document}